\documentclass[11pt,a4paper]{article}
\usepackage{titling,lipsum}
\usepackage{caption}
\usepackage{subcaption}
\usepackage{graphicx} 
\usepackage{url}      
\usepackage{amssymb}
\usepackage{amsthm}
\usepackage{amsfonts}%
\usepackage[fleqn]{amsmath}
\usepackage{float}
\newtheorem{theorem}{Theorem}

\newtheorem{lemma}{Lemma}

\theoremstyle{remark}
\newtheorem{remark}{Remark}

\begin{document}
\date{}
\title{Numerical solution of nonlinear parabolic systems by block monotone iterations}
\author{Mohamed Al-Sultani
\\
Institute of Fundamental Sciences, Massey University,\\
 Palmerston North, New Zealand\\
E-mail: M.Al-Sultani@massey.ac.nz}\maketitle
\begin{abstract}
    This paper deals with  investigating numerical methods for solving  coupled system of nonlinear parabolic  problems. We utilize block monotone iterative methods based on   Jacobi and Gauss--Seidel methods to solve difference schemes which approximate the coupled system of nonlinear parabolic problems, where reaction functions are quasi-monotone nondecreasing or nonincreasing. In the view of upper and lower solutions method, two monotone upper and lower sequences of solutions are constructed, where the monotone property ensures the theorem on existence of solutions to problems with quasi-monotone  nondecreasing and nonincreasing  reaction functions. Construction of initial upper and lower solutions is presented. The sequences of solutions generated by the block Gauss--Seidel method converge not slower than by the block Jacobi method.
 \end{abstract}
\section{Introduction}
Several problems in the chemical, physical and engineering sciences are characterized by coupled systems of
nonlinear parabolic equations \cite{P92}. 
 In this paper, we construct  block monotone iterative methods for solving  the coupled system of nonlinear parabolic  equations
\begin{align}\label{bp1}
&u_{\alpha,t} -L_{\alpha}u_{\alpha}(x,y,t)+f_{\alpha}(x,y,t,u)=0,\quad (x,y,t)\in  Q_{T}=\omega \times (0,T], \\
&\alpha=1,2,\quad\omega=\{(x,y):0<x<l_{1},\quad 0<y<l_{2}\},\nonumber\\
 & u(x,y,t)=g(x,y,t),\quad (x,y,t)\in \partial Q_{T}, \nonumber \\
 &  u(x,y,0)=\psi(x,y),\quad (x,y)\in \overline{\omega}, \nonumber
\end{align}
where $u=(u_{1},u_{2})$, $g=(g_{1},g_{2})$, $f=(f_{1}, f_{2})$, $\psi=(\psi_{1},\psi_{2})$,  $\partial Q_{T}= \partial \omega\times (0,T]$, and $\partial \omega$
is the boundary of $\omega$. The differential operators $L_{\alpha}$, $\alpha=1,2$,  are defined by
\begin{align*}
 &L _{\alpha}u_{\alpha}(x,y,t)= \varepsilon_{\alpha}(u_{\alpha,xx}+u_{\alpha,yy})-v_{\alpha,1}(x,y,t)u_{\alpha,x}-v_{\alpha,2}(x,y,t)u_{\alpha,y},\\
  & \alpha=1,2,
\end{align*}
where
$\varepsilon_{\alpha}$, $\alpha=1,2$, are positive constants diffusion coefficients. It is assumed that the functions $f_{\alpha}$, $g_{\alpha}$, $v_{\alpha}$, $\alpha=1,2$,
are smooth in their respective domains.

The aim of this paper is to construct and  investigate  block monotone iterative methods
based on Jacobi and Gauss--Seidel methods  for solving coupled systems of nonlinear parabolic equations  with
quasi-monotone nondecreasing or quasi-monotone nonincreasing   reaction functions $f_{\alpha}$, $\alpha=1,2$,
which satisfy the inequalities
\begin{equation*}
-\frac{\partial f_{\alpha}}{\partial u_{\alpha^{\prime}}}\geq0,\quad (x,y,t)\in \overline{Q}_{T},\quad \alpha^{\prime}\neq \alpha,\quad \alpha, \alpha^{\prime}=1,2,
\end{equation*}
when $f_{\alpha}$, $\alpha=1,2$,  are  quasi-monotone nondecreasing, and
\begin{equation*}
-\frac{\partial f_{\alpha}}{\partial u_{\alpha^{\prime}}}\leq0,\quad (x,y,t)\in \overline{Q}_{T},\quad \alpha^{\prime}\neq \alpha,\quad \alpha, \alpha^{\prime}=1,2,
\end{equation*}
when $f_{\alpha}$, $\alpha=1,2$, are  quasi-monotone nonincreasing.
\section{Properties of solutions to system (\ref{bp1})}
 We introduce the following notation:
\begin{equation}\label{bp2}
  f_{\alpha}(x,y,t,u_{\alpha},u_{\alpha^{\prime}})=\left\{ \begin{array}{ll}
f_{1}(x,y,t,u_{1},u_{2}),\quad \alpha=1,\\
f_{2}(x,y,t,u_{1},u_{2}),\quad \alpha=2,
\end{array}\right.
\alpha\neq \alpha^{\prime}.
\end{equation}
 Two vector functions $\widetilde{u}(x,y,t)=(\widetilde{u}_{1},\widetilde{u}_{2})$ and
$\widehat{u}(x,y,t)=(\widehat{u}_{1},\widehat{u}_{2})$, are called ordered  upper and lower solutions to (\ref{bp1}), if they satisfy the inequalities
\begin{subequations}\label{bp3}
\begin{equation}\label{bp3a}
\widehat{u}(x,y,t)\leq \widetilde{u}(x,y,t),\quad(x,y,t)\in \overline{Q}_{T},
\end{equation}
\begin{equation}\label{bp3b}
\widehat{u}_{\alpha,t}-\mbox{L}_{\alpha}\widehat{u}_{\alpha}+f_{\alpha}(x,y,t,\widehat{u})\leq0 \leq \widetilde{u}_{\alpha,t} -\mbox{L}_{\alpha}\widetilde{u}_{\alpha}+f_{\alpha}(x,y,t,\widetilde{u}),        
\end{equation}
\begin{equation}\label{bp3c}
(x,y,t)\in Q_{T},\quad \widehat{u}(x,y,t)\leq g(x,y,t) \leq \widetilde{u}(x,y,t),\quad(x,y,t)\in \partial Q_{T},
\end{equation}
\begin{equation*}
\widehat{u}(x,y,0)\leq \psi(x,y) \leq \widetilde{u}(x,y,0),\quad (x,y)\in \overline{\omega},\quad \alpha=1,2,
\end{equation*}
\end{subequations}
when the reaction functions $f_{\alpha}$, $\alpha=1,2$, are quasi-monotone nondecreasing,
and if they satisfy the inequalities
\begin{subequations}\label{bp4}
\begin{equation}\label{bpa}
\widehat{u}(x,y,t)\leq \widetilde{u}(x,y,t),\quad(x,y,t)\in \overline{Q}_{T},
\end{equation}
\begin{equation}\label{bpb}
\widehat{u}_{\alpha,t}-\mbox{L}_{\alpha}\widehat{u}_{\alpha}+f_{\alpha}(x,y,t,\widehat{u}_{\alpha},\widetilde{u}_{\alpha^{\prime}})\leq0 \leq \widetilde{u}_{\alpha,t} -\mbox{L}_{\alpha}\widetilde{u}_{\alpha}+f_{\alpha}(x,y,t,\widetilde{u}_{\alpha},\widehat{u}_{\alpha^{\prime}}),        
\end{equation}
\begin{equation*}\label{bpc}
(x,y,t)\in Q_{T},\quad \widehat{u}(x,y,t)\leq g(x,y,t) \leq \widetilde{u}(x,y,t),\quad(x,y,t)\in \partial Q_{T},
\end{equation*}
\begin{equation*}
\widehat{u}(x,y,0)\leq \psi(x,y) \leq \widetilde{u}(x,y,0),\quad (x,y)\in \overline{\omega},\quad \alpha=1,2,
\end{equation*}
\end{subequations}
when the reaction functions $f_{\alpha}$, $\alpha=1,2$, are quasi-monotone   nonincreasing.

For a given ordered upper $\widetilde{u}$ and lower $\widehat{u}$ solutions, a sector $\langle \widehat{u}, \widetilde{u}\rangle$ is defined as follows:
\begin{equation*}
  \langle \widehat{u}, \widetilde{u}\rangle=\left\{u(x,y,t):\quad \widehat{u}(x,y,t)\leq u(x,y,t)\leq \widetilde{u}(x,y,t),\quad (x,y,t)\in \overline{Q}_{T} \right\}.
\end{equation*}
In the sector $\langle \widehat{u}, \widetilde{u}\rangle$, the vector  function $f(x,y,t,u)$ is assumed to satisfy the constraints
\begin{equation}\label{bp5}
  0\leq\frac{\partial f_{\alpha}(x,y,t,u)}{\partial u_{\alpha}}\leq c_{\alpha}(x,y,t),\quad u\in \langle \widehat{u},\widetilde{u}\rangle,\quad (x,y,t)\in \overline{Q}_{T},\quad \alpha=1,2,
\end{equation}
where $c_{\alpha}(x,y,t)$, $\alpha=1,2$, are nonnegative   bounded functions.

The vector function $f(x,y,t,u)$ is called quasi-monotone nondecreasing in the sector $\langle \widehat{u}, \widetilde{u}\rangle$ if it satisfies the conditions
\begin{equation}\label{bp6}
-\frac{\partial f_{\alpha}(x,y,t,u)}{\partial u_{\alpha^{\prime}}}\geq0,\quad u\in \langle \widehat{u},\widetilde{u}\rangle, \quad (x,y,t)\in \overline{Q}_{T},\quad \alpha^{\prime}\neq \alpha,\quad \alpha, \alpha^{\prime}=1,2,
\end{equation}
and $f(x,y,t,u)$ is called quasi-monotone nonincreasing if it satisfies the conditions
\begin{equation}\label{bp7}
-\frac{\partial f_{\alpha}(x,y,t,u)}{\partial u_{\alpha^{\prime}}}\leq0,\quad u\in \langle \widehat{u},\widetilde{u}\rangle, \quad (x,y,t)\in \overline{Q}_{T},\quad \alpha^{\prime}\neq \alpha,\quad \alpha, \alpha^{\prime}=1,2.
\end{equation}
\begin{theorem}
Let $\widetilde{u}$ and $\widehat{u}$ be ordered upper and lower solutions of problem (\ref{bp1}),
  $f$ in (\ref{bp1}) be quasi-monotone nondecreasing  (\ref{bp6}) or quasi-monotone nonincreasing   (\ref{bp7}) in the sector $\langle \widehat{u}, \widetilde{u}\rangle$ and satisfy (\ref{bp5}). Then problem (\ref{bp1}) has a unique solution in the sector $\langle \widehat{u}, \widetilde{u}\rangle$.
\end{theorem}
The proof of the theorem can be found in  Theorems 8.3.1 and 8.3.2,  \cite{P92}.
\section{The nonlinear difference scheme}
\subsection{The statement of  the nonlinear difference\\ scheme}
On  $\overline{\omega}$ and [0,T], we introduce a rectangular mesh $\overline{\Omega}^{h}=\overline{\Omega}^{hx}\times\overline{\Omega}^{hy}$ and $ \overline{\Omega}^{\tau}$, such that
\begin{align}\label{bp8}
&\overline{\Omega}^{hx}=\{x_{i},\quad  i=0,1,\ldots, N_{x};\quad x_{0}=0, \quad x_{N_{x}}=l_{1}; \ \  hx=x_{i+1}-x_{i}\},  \\
&\overline{\Omega}^{hy}=\{y_{j},\quad  j=0,1,\ldots, N_{y};\quad y_{0}=0, \quad y_{N_{y}}=l_{2}; \ \ hy=y_{j+1}-y_{j}\},\nonumber \\
& \overline{\Omega}^{\tau}=\{t_{m},\quad m=0,1,\ldots,N_{\tau};\quad t_{0}=0,\quad t_{N_{\tau}}=T;\ \ \tau=t_{m}-t_{m-1} \}.\nonumber
\end{align}
For a mesh function $U(p,t_{m})=(U_{1}(p,t_{m}),U_{2}(p,t_{m}))$, $(p,t_{m}) \in \overline{\Omega}^{h}\times\overline{\Omega}^{\tau}$, $p=(x_{i},y_{j})$,
we use   the implicit  difference  scheme
\begin{eqnarray}\label{bp9}
&&\left(\mathcal{L}_{\alpha}^{h}(p,t_{m})+\tau^{-1}\right)U_{\alpha}(p,t_{m}) +f_{\alpha}(p,t_{m},U)-\tau^{-1} U_{\alpha}(p,t_{m-1})=0,\\
&&\quad (p,t_{m})\in \Omega^{h\tau}=\Omega^{h}\times \Omega^{\tau},\quad U(p,t_{m})=g(p,t_{m}),\quad (p,t_{m})\in \partial \Omega^{h\tau},\nonumber\\
&&\quad U(p,0)=\psi(p),\quad p\in\overline{\Omega}^{h},\nonumber
\end{eqnarray}
\begin{align*}
 \mathcal{L}_{\alpha}^{h}(p,t_{m})U_{\alpha}(p,t_{m})=&-\varepsilon_{\alpha}\left( D^{2}_{x}U_{\alpha}(p,t_{m})+D^{2}_{y}U_{\alpha}(p,t_{m})\right)\nonumber \\
  &+v_{\alpha,1}(p,t_{m})D^{1}_{x}U_{\alpha}(p,t_{m}) \nonumber\\
  &+v_{\alpha,2}(p,t_{m})D^{1}_{y}U_{\alpha}(p,t_{m}),\quad \alpha=1,2,
\end{align*}
where $\partial\Omega^{h}$ is the boundary of $\Omega^{h}$.
It is assumed that the functions  $v_{\alpha,1}(p,t_{m})$ and $v_{\alpha,2}(p,t_{m})$ , $(p,t_{m})\in\overline{\Omega}^{h\tau}$, $\alpha=1,2$,
 are nonnegative, $D^{2}_{x}U_{\alpha}(p,t_{m})$, $D^{2}_{y}U_{\alpha}(p,t_{m})$ and $D^{1}_{x}U_{\alpha}(p,t_{m})$, $D^{1}_{y}U_{\alpha}(p,t_{m})$, $\alpha=1,2$,  are, respectively, the central difference and backward difference approximations to the second and first derivatives:
\begin{align*}
&D^{2}_{x}U_{\alpha}(x_{i},y_{j},t_{m}) =\frac{U_{\alpha,i-1,j,m}-2U_{\alpha,ij,m}+U_{\alpha,i+1,j,m}}{h_{x}^{2}},\\
&D^{2}_{y}U_{\alpha}(x_{i},y_{j},t_{m}) =\frac{U_{\alpha,i,j-1,m}-2U_{\alpha,ij,m}+U_{\alpha,i,j+1,m}}{h_{y}^{2}}, \\
&D^{1}_{x}U_{\alpha}(x_{i},y_{j},t_{m}) =\frac{U_{\alpha,ij,m}-U_{\alpha,i-1,j,m}}{h_{x}}, \\
&D^{1}_{y}U_{\alpha}(x_{i},y_{j},t_{m}) =\frac{U_{\alpha,ij,m}-U_{\alpha,i,j-1,m}}{h_{y}},\quad \alpha=1,2,
\end{align*}
where $U_{\alpha,ij,m}\equiv U_{\alpha}(x_{i},y_{j},t_{m})$.
\begin{remark}\label{bp10}
An approximation of the first derivatives $u_{x}$ and $u_{y}$ depends on the signs of $v_{\alpha,1}(x,y,t)$ and $v_{\alpha,2}(x,y,t)$ , $\alpha=1,2$.
When $v_{\alpha,1}(x,y,t)$ and $v_{\alpha,2}(x,y,t)$, $\alpha =1,2$, are nonpositive, then $u_{x}$ and $u_{y}$ are approximated by the  forward difference formula. The first derivatives  $u_{x}$ and $u_{y}$ are approximated by  using both forward or backward difference formulae when $v_{\alpha,1}(x,y,t)$ and $v_{\alpha,2}(x,y,t)$,  $\alpha=1,2$, have variable signs.
\end{remark}
On each time level $t_{m}$, $m\geq1$, we introduce the linear problems
\begin{align}\label{bp11}
&\Big( \mathcal{L}_{\alpha}^{h}(p,t_{m})+\left(\tau^{-1}+ k_{\alpha}(p,t_{m})\right)I\Big)W_{\alpha}(p,t_{m})=\varphi_{\alpha}(p,t_{m}),\quad p\in \Omega^{h},\nonumber\\
&  \alpha=1,2,\quad U(p,t_{m})=g(p,t_{m}),\quad p\in \partial \Omega^{h},
\end{align}
where $I$ is the identity operator and  $k_{\alpha}(p,t_{m})$, $\alpha=1,2$, are nonnegative bounded mesh  functions. We now formulate the maximum
principle for the difference operators $ \mathcal{L}_{\alpha}^{h}(p,t_{m})+(\tau^{-1}+ k_{\alpha}(p,t_{m}))I$, $\alpha=1,2$,
and give an estimate of the solution to (\ref{bp11}).
\begin{lemma}\label{bp12}
 \begin{itemize}
\item[(i)] If $W_{\alpha}(p,t_{m})$, $\alpha=1,2$,  satisfy  the conditions
  \begin{align*}
 &\left( \mathcal{L}_{\alpha}^{h}(p,t_{m})+(\tau^{-1}+k_{\alpha}(p,t_{m}))I\right) W_{\alpha}(p,t_{m})\geq0 \ (\leq 0),\quad p\in \Omega^{h}, \nonumber\\
 & W_{\alpha}(p,t_{m}))\geq0 \ (\leq 0),\quad p\in \partial \Omega^{h},
   \end{align*}
 then $W_{\alpha}(p,t_{m})\geq0 \ (\leq0),\quad p \in \overline{\Omega}^{h}$.
\item[(ii)]
 The following estimates of the solution to (\ref{bp11}) hold
  \begin{equation}\label{bp13}
   \|W_{\alpha}(\cdot,t_{m}) \|_{\overline{\Omega}^{h}}\leq \max \left\{\|g_{\alpha} (\cdot,t_{m})\|_{\partial \Omega^{h}},\left\| \frac{\varphi_{\alpha}(\cdot,t_{m})}{ k_{\alpha}(\cdot,t_{m})+\tau^{-1}}\right\|_{\Omega^{h}} \right\},\quad \alpha=1,2,
  \end{equation}
  where
  \begin{align*}
  &  \|g_{\alpha}(\cdot,t_{m})\|_{\partial\Omega^{h}}=\max_{p\in\partial\Omega^{h}}|g_{\alpha}(p,t_{m})|,\\
   &\left\| \frac{\varphi_{\alpha}(\cdot,t_{m})}{ k_{\alpha}(\cdot,t_{m})+\tau^{-1}}\right\|_{\Omega^{h}}
   =\max_{p\in\Omega^{h}}\left|  \frac{\varphi_{\alpha}(p,t_{m})}{k_{\alpha}(p,t_{m})+\tau^{-1}}\right|.
  \end{align*}
\end{itemize}
\end{lemma}
The proof of the lemma can be found in \cite{Ab79}, \cite{sm2001}.
\begin{remark}\label{bp14}
In this remark we discuss the mean-value theorem for vector-valued functions. Introduce the following notation:
\begin{equation}\label{bp15}
\mathcal{F}_{\alpha}(x,y,t,u_{\alpha},u_{\alpha^{\prime}})=\left\{ \begin{array}{ll}
\mathcal{F}_{1}(x,y,t,u_{1},u_{2}),\quad \alpha=1,\\
\mathcal{F}_{2}(x,y,t,u_{1},u_{2}),\quad \alpha=2.
\end{array}\right.
\end{equation}
Assume that $\mathcal{F}_{\alpha}(x,y,t,u_{\alpha},u_{\alpha^{\prime}})$, $\alpha=1,2$, are smooth functions, then we have
\begin{equation}\label{bp16}
   \mathcal{F}_{\alpha}(x,y,t,u_{\alpha},u_{\alpha^{\prime}})-\mathcal{F}_{\alpha}(x,y,t,w_{\alpha},u_{\alpha^{\prime}})=\frac{\partial \mathcal{F}_{\alpha}(h_{\alpha}, u_{\alpha^{\prime}})}{\partial u_{\alpha}}[u_{\alpha}-w_{\alpha}],
  \end{equation}
   \begin{equation*}
   \mathcal{F}_{\alpha}(x,y,t,u_{\alpha},u_{\alpha^{\prime}})-\mathcal{F}_{\alpha}(x,y,t,u_{\alpha},w_{\alpha^{\prime}})=\frac{\partial \mathcal{F}_{\alpha}(u_{\alpha},h_{\alpha^{\prime}})}{\partial u_{\alpha^{\prime}}}[u_{\alpha^{\prime}}-w_{\alpha^{\prime}}],
  \end{equation*}
  where $h_{\alpha}(x,y,t)$ lies between $u_{\alpha}(x,y,t)$ and $w_{\alpha}(x,y,t)$, and  $h_{\alpha^{\prime}}(x,y,t)$ lies between $u_{\alpha^{\prime}}(x,y,t)$ and $w_{\alpha^{\prime}}(x,y,t)$, $\alpha=1,2$.
\end{remark}
\subsection{Quasi-monotone nondecreasing reaction functions}
On each time level $t_{m}\in \Omega^{\tau}$, $m\geq1$,  the vector mesh functions
\begin{equation*}
\widetilde{U}(p,t_{m})=(\widetilde{U}_{1}(p,t_{m}),\widetilde{U}_{2}(p,t_{m})),\quad \widehat{U}(p,t_{m})=(\widehat{U}_{1}(p,t_{m}),\widehat{U}_{2}(p,t_{m})),
\end{equation*}
\begin{equation*}
 p\in \overline{\Omega}^{h},
\end{equation*}
are called ordered upper and lower solutions of (\ref{bp9}),  if they satisfy the inequalities
\begin{subequations}\label{bp19}
  \begin{equation}\label{bp19a}
\widetilde{U}(p,t_{m})\geq\widehat{U}(p,t_{m}),\quad p\in \overline{\Omega}^{h},\quad m\geq1,
\end{equation}
\begin{equation}\label{bp19b}
\left(\mathcal{L}_{\alpha}^{h}(p,t_{m})+\tau^{-1}\right)\widetilde{U}_{\alpha}(p,t_{m}) +f_{\alpha}(p,t_{m},\widetilde{U})-\tau^{-1} \widetilde{U}_{\alpha}(p,t_{m-1})\geq0,
\end{equation}
\begin{equation*}
  \left(\mathcal{L}_{\alpha}^{h}(p,t_{m})+\tau^{-1}\right)\widehat{U}_{\alpha}(p,t_{m}) +f_{\alpha}(p,t_{m},\widehat{U})-\tau^{-1} \widehat{U}_{\alpha}(p,t_{m-1})\leq0,\  p\in \Omega^{h},
\end{equation*}
\begin{equation}\label{bp19c}
  \alpha=1,2,\ m\geq1,\quad \widehat{U}(p,t_{m})\leq g(p,t_{m})\leq \widetilde{U}(p,t_{m}),\ p\in \partial \Omega^{h},
  \end{equation}
  \begin{equation*}
 \widehat{U}(p,0)\leq \psi(p)\leq \widetilde{U}(p,0),\quad p\in\overline{\Omega}^{h}.\nonumber
\end{equation*}
\end{subequations}
For a given pair of ordered upper and lower solutions $\widetilde{U}(p,t_{m})$ and $\widehat{U}(p,t_{m})$, we define the sector
\begin{equation}\label{bp41}
  \langle \widehat{U}(t_{m}),\widetilde{U}(t_{m})\rangle=\left\{U(p,t_{m}):\widehat{U}(p,t_{m})\leq U(p,t_{m})\leq \widetilde{U}(p,t_{m}),\quad p\in \overline{\Omega}^{h} \right\}.
\end{equation}
In the sector $ \langle \widehat{U}(t_{m}),\widetilde{U}(t_{m})\rangle$, the vector   function $f(p,t_{m},U)$ is  assumed to satisfy the constraints
\begin{equation}\label{bp20}
 \frac{\partial f_{\alpha}(p,t_{m},U)}{\partial u_{\alpha}}\leq c_{\alpha}(p,t_{m}),\quad U\in \langle \widehat{U}(t_{m}),\widetilde{U}(t_{m})\rangle,\quad p\in\overline{\Omega}^{h},\quad \alpha=1,2,
\end{equation}
\begin{equation}\label{bp21}
-\frac{\partial f_{\alpha}(p,t_{m},U)}{\partial u_{\alpha^{\prime}}}\geq0,\quad U\in \langle \widehat{U}(t_{m}),\widetilde{U}(t_{m})\rangle, \ p\in \overline{\Omega}^{h},\ \alpha^{\prime}\neq \alpha,\quad \alpha, \alpha^{\prime}=1,2,
\end{equation}
where $c_{\alpha}(p,t_{m})$, $\alpha=1,2$, are nonnegative  bounded functions.
Reaction functions,  which satisfy (\ref{bp21}), are called
quasi-monotone nondecreasing.
\par
We introduce the notation
\begin{equation}\label{bp22}
\Gamma_{\alpha}(p,t_{m},U)=c_{\alpha}(p,t_{m})U_{\alpha}(p,t_{m})-f_{\alpha}(p,t_{m},U),\quad p\in \overline{\Omega}^{h} ,\quad \alpha=1,2,
\end{equation}
where $c_{\alpha}(p,t_{m})$, $\alpha=1,2$, are defined in (\ref{bp20}), and give a monotone property of $\Gamma_{\alpha}$, $\alpha=1,2$.
\begin{lemma}\label{bp23}
Suppose that  $U=(U_{1},U_{2})$ and $V=(V_{1},V_{2})$, are any functions in $\langle\widehat{U}(t_{m}),\widetilde{U}(t_{m})\rangle$, where
 $U\geq V$, and
assume that (\ref{bp20}),  (\ref{bp21}) are satisfied. Then
\begin{equation}\label{bp24}
  \Gamma_{\alpha}(U)\geq \Gamma_{\alpha}(V),\quad \alpha=1,2,
\end{equation}
where $(p,t_{m})$ is suppressed in (\ref{bp24}).
\end{lemma}
\begin{proof}
  From (\ref{bp22}), we have
  \begin{eqnarray}\label{bp25}
  \Gamma_{\alpha}(U)-\Gamma _{\alpha}(V)&=&c_{\alpha}(p,t_{m})(U_{\alpha}(p,t_{m})-V_{\alpha}(p,t_{m}) \\ \nonumber
 & &- \left[f_{\alpha}(p,t_{m},U_{2})-f_{\alpha}(p,t_{m},V_{1},U_{2})\right]\\ \nonumber
 & &- \left[f_{\alpha}(p,t_{m},V_{1},U_{2})-f_{\alpha}(p,t_{m},V_{1},V_{2})\right]. \nonumber
\end{eqnarray}
For $\alpha=1$ in (\ref{bp25}), using the mean-value theorem (\ref{bp16}), we obtain
\begin{eqnarray*}
  \Gamma_{1}(U)-\Gamma _{1}(V)&=& \left(c_{1}(p,t_{m})-\frac{\partial f_{1}(Q_{1},U_{2})}{\partial u_{1}}\right)(U_{1}-V_{1})\\ &&-\frac{\partial f_{1}(V_{1},Q_{2})}{\partial u_{2}}(U_{2}-V_{2}),
\end{eqnarray*}
where
\begin{equation*}
V_{\alpha}(p,t_{m})\leq Q_{\alpha}(p,t_{m})\leq U_{\alpha}(p,t_{m}),\quad p\in \overline{\Omega}^{h},\quad \alpha=1,2,\quad m\geq1.
\end{equation*}
From here, (\ref{bp20}), (\ref{bp21}) and taking into account that $U_{\alpha}\geq V_{\alpha}$, $\alpha=1,2$, we conclude (\ref{bp24})
for $\alpha=1$. Similarly, we can prove (\ref{bp24}) for $\alpha=2$.
\end{proof}
\subsection{Quasi-monotone nonincreasing reaction functions}
On each time level $t_{m}\in \Omega^{\tau}$, $m\geq1$,  the vector mesh functions
\begin{equation*}
\widetilde{U}(p,t_{m})=(\widetilde{U}_{1}(p,t_{m}),\widetilde{U}_{2}(p,t_{m})),\quad \widehat{U}(p,t_{m})=(\widehat{U}_{1}(p,t_{m}),\widehat{U}_{2}(p,t_{m})),
\end{equation*}
\begin{equation*}
 p\in \overline{\Omega}^{h},
\end{equation*}
are called ordered upper and lower solutions of (\ref{bp9}), if they satisfy the inequalities
\begin{subequations}\label{bp40}
  \begin{equation}\label{bp40a}
\widetilde{U}(p,t_{m})\geq\widehat{U}(p,t_{m}),\quad p\in \overline{\Omega}^{h},\quad m\geq1,
\end{equation}
\begin{equation}\label{bp40b}
 \left(\mathcal{L}_{\alpha}^{h}(p,t_{m})+\tau^{-1}\right)\widetilde{U}_{\alpha}(p,t_{m}) +f_{\alpha}(p,t_{m},\widetilde{U}_{\alpha},\widehat{U}_{\alpha^{\prime}})-\tau^{-1} \widetilde{U}_{\alpha}(p,t_{m-1})\geq0,
 \end{equation}
 \begin{equation*}
 \left(\mathcal{L}_{\alpha}^{h}(p,t_{m})+\tau^{-1}\right)\widehat{U}_{\alpha}(p,t_{m}) +f_{\alpha}(p,t_{m},\widehat{U}_{\alpha},\widetilde{U}_{\alpha^{\prime}})-\tau^{-1} \widehat{U}_{\alpha}(p,t_{m-1})\leq0,
\end{equation*}
\begin{equation*}
p\in \Omega^{h}, \quad \alpha^{\prime}\neq \alpha, \quad  \alpha, \alpha^{\prime}=1,2,\quad m\geq1,
\end{equation*}
\begin{equation}\label{bp40c}
\widehat{U}(p,t_{m})\leq g(p,t_{m})\leq \widetilde{U}(p,t_{m}),\quad p\in \partial \Omega^{h},
\end{equation}
\begin{equation*}
  \widehat{U}(p,0)\leq \psi(p)\leq \widetilde{U}(p,0),\quad p\in\overline{\Omega}^{h}.
\end{equation*}
\end{subequations}
The upper $\widetilde{U}(p,t_{m})$ and lower $\widehat{U}(p,t_{m})$ solutions are dependent of each other and calculated simultaneously.

We assume that in the sector $\langle \widehat{U},\widetilde{U}\rangle$ defined in (\ref{bp41}),
the vector function $f(p,t_{m},U)$ in (\ref{bp9}), satisfies the constraints (\ref{bp20}) and
\begin{equation}\label{bp42}
-\frac{\partial f_{\alpha}(p,t_{m},U)}{\partial u_{\alpha^{\prime}}}\leq0,\quad U\in \langle \widehat{U}(t_{m}),\widetilde{U}(t_{m})\rangle, \ p\in \overline{\Omega}^{h},\ \alpha^{\prime}\neq \alpha,\quad \alpha, \alpha^{\prime}=1,2.
\end{equation}
Reaction functions, which satisfy (\ref{bp42}), are called quasi-monotone nonincreasing.
 We give a monotone property of $\Gamma_{\alpha}$, $\alpha=1,2$, in the case of quasi-monotone  nonincreasing
 reaction functions, where $\Gamma_{\alpha}$, $\alpha=1,2$, are defined in   (\ref{bp22}).
\begin{lemma}\label{bp43}
Suppose that $U=(U_{1},U_{2})$ and $V=(V_{1},V_{2})$, are any functions in $\langle\widehat{U}(t_{m},\widetilde{U}(t_{m})\rangle$, where
 $U\geq V$, and
assume that (\ref{bp20}) and (\ref{bp42}) are satisfied. Then
\begin{equation}\label{bp44}
  \Gamma_{\alpha}(U_{1},V_{2})\geq \Gamma_{\alpha}(V_{1},U_{2}),\quad \alpha=1,2,
\end{equation}
where $(p,t_{m})$ is suppressed in (\ref{bp44}).
\end{lemma}
\begin{proof}
  From (\ref{bp22}), we have
  \begin{eqnarray}\label{bp45}
   \Gamma_{\alpha}(U_{1},V_{2})- \Gamma_{\alpha}(V_{1},U_{2})&=&c_{\alpha}(p,t_{m})(U_{\alpha}(p,t_{m})-V_{\alpha}(p,t_{m}))\\ \nonumber
 &&- \left[f_{\alpha}(p,t_{m},U_{1},V_{2})-f_{\alpha}(p,t_{m},V_{1},V_{2})\right]\\ \nonumber
 &&+ \left[f_{\alpha}(p,t_{m},V_{1},U_{2})-f_{\alpha}(p,t_{m},V_{1},V_{2})\right]. \nonumber
\end{eqnarray}
For $\alpha=1$ in (\ref{bp45}), using the mean-value theorem (\ref{bp16}), we obtain
\begin{align*}
  \Gamma_{1}(U_{1},V_{2})- \Gamma_{1}(V_{1},U_{2})=& \left(c_{1}(p,t_{m})-\frac{\partial f_{1}(Q_{1},V_{2})}{\partial u_{1}}\right)(U_{1}-V_{1})\\ &+\frac{\partial f_{1}(V_{1},Q_{2})}{\partial u_{2}}(U_{2}-V_{2}),
\end{align*}
where
\begin{equation*}
V_{\alpha}(p,t_{m})\leq Q_{\alpha}(p,t_{m})\leq U(p,t_{m}),\quad (p,t_{m})\in \overline{\Omega}^{h \tau},\quad \alpha=1,2.
\end{equation*}
From here, (\ref{bp20}), (\ref{bp42}) and taking into account that $U_{\alpha}\geq V_{\alpha}$, $\alpha=1,2$, we conclude that
$$
  \Gamma_{1}(U_{1},V_{2})- \Gamma_{1}(V_{1},U_{2})\geq0.
 $$
 Similarly, we can prove that
 $$
 \Gamma_{2}(U_{1},V_{2})- \Gamma_{2}(V_{1},U_{2})\geq0.
 $$
\end{proof}
\section{The case of quasi-monotone nondecreasing reaction functions}
\subsection{The statement of the block nonlinear difference scheme}
Write down the difference scheme (\ref{bp9}) at an interior mesh point $(x_{i},y_{j})\in \Omega^{h}$ in the form
\begin{align}\label{bp17}
&d_{\alpha,ij,m}U_{\alpha,ij,m}-l_{\alpha,ij,m}U_{\alpha,i-1,j,m}-r_{\alpha,ij,m}U_{\alpha,i+1,j,m}-b_{\alpha,ij,m}U_{\alpha,i,j-1,m}\nonumber \\ &-t_{\alpha,ij,m}U_{\alpha,i,j+1,m}+f_{\alpha,ij,m}(U_{1,ij,m},U_{2,ij,m})-\tau^{-1}U_{\alpha,ij,m-1}\nonumber\\
& +G_{\alpha,ij,m}^{*}=0,\quad i=1,2,\ldots,N_{x}-1,\quad j=1,2,\ldots,N_{y}-1,  \\
&U_{\alpha,ij,m}=g_{\alpha,ij,m},\quad i=0,N_{x},\quad j=0,N_{y}, \nonumber\\
&U_{\alpha,ij,0}=\psi_{\alpha,ij},\quad i=0,1,\ldots,N_{x},\quad j=0,1,\ldots,N_{y},     \nonumber
  \end{align}
\begin{equation*}
l_{\alpha,ij,m}=\frac{\varepsilon_{\alpha}}{h_{x}^{2}}+\frac{v_{\alpha}(x_{i},y_{j},t_{m})}{h_{x}},\quad r_{\alpha,ij}=\frac{\varepsilon_{\alpha}}{h_{x}^{2}},
\end{equation*}
\begin{equation*}
b_{\alpha,ij,m}=\frac{\varepsilon_{\alpha}}{h_{y}^{2}}+\frac{v_{\alpha}(x_{i},y_{j},t_{m})}{h_{y}},\quad t_{\alpha,ij}=\frac{\varepsilon_{\alpha}}{h_{y}^{2}},
\end{equation*}
\begin{equation*}
d_{\alpha,ij,m}=\tau^{-1}+l_{\alpha,ij,m}+r_{\alpha,ij,m}+b_{\alpha,ij,m}+t_{\alpha,ij,m},\quad \alpha=1,2,
\end{equation*}
where $G_{\alpha,ij,m}^{*}$ is  associated with the boundary function $g_{\alpha}(x_{i},y_{j},t_{m})$. On each time level $m$, $m\geq1$, we define column vectors and diagonal matrices by
\begin{equation*}
  U_{\alpha,i,m}=(U_{\alpha,i,1,m}, \ldots, U_{\alpha,i,N_{y}-1,m})^{T},\ \ G_{\alpha,i,m}^{*}=(G_{\alpha,i,1,m}^{*}, \ldots, G_{\alpha,i,N_{y}-1,m}^{*})^{T},
  \end{equation*}
  \begin{equation*}
    g_{\alpha,i,m}=(g_{\alpha,i,0,m}, g_{\alpha,i,N_{y},m})^{T},\quad i=0,N_{x},
  \end{equation*}
  \begin{equation*}
    \psi_{\alpha,i}=(\psi_{\alpha,i,0}, \ldots, \psi_{\alpha,i,N_{y}})^{T},\quad i=0,1,\ldots,N_{x},
  \end{equation*}
  \begin{align*}
 & F_{\alpha,i,m}(U_{1,i},U_{2,i})=\\
 &(f_{\alpha,i,1,m}(U_{1,i,1,m},U_{2,i,1,m}), \ldots, f_{\alpha,i,N_{y}-1,m}(U_{1,i,N_{y}-1,m},U_{2,i,N_{y}-1,m}))^{T},
 \end{align*}
\begin{equation*}
 L_{\alpha,i,m}=\mbox{diag}(l_{\alpha,i,1,m},\ldots, l_{\alpha,i,N_{y}-1,m}),
 \end{equation*}
 \begin{equation*}
 R_{\alpha,i,m}=\mbox{diag}(r_{\alpha,i,1,m},\ldots, r_{\alpha,i,N_{y}-1,m}),\quad \alpha=1,2,
\end{equation*}
where $L_{\alpha,1,m}U_{\alpha,0,m}$ is included in $G_{\alpha,1,m}^{*}$, and  $R_{\alpha,N_{x}-1,m}U_{\alpha,N_{x},m}$ is included in $G_{\alpha,N_{x},m}^{*}$.
Then the difference scheme (\ref{bp9}) may be written in the form
\begin{align}\label{bp18}
& A_{\alpha,i,m}U_{\alpha,i,m}-(L_{\alpha,i,m}U_{\alpha,i-1,m}+R_{\alpha,i,m}U_{\alpha,i+1,m})=\\
&+F_{\alpha,i,m}(U_{i,m})-\tau^{-1}U_{\alpha,i,m-1}+G_{\alpha,i,m}^{*}=\mathbf{0},\nonumber\\
& i=1,2,\ldots,N_{x}-1,\quad j=1,2,\ldots,N_{y}-1,\quad \alpha=1,2,\quad m\geq1,\nonumber \\
& U_{\alpha,i,m}=g_{\alpha,i,m},\quad i=0,N_{x},\quad U_{\alpha,i,0} =\psi_{\alpha,i},\quad i=0,1,\ldots,N_{x},\nonumber
\end{align}
\begin{equation*}
U_{i,m}=(U_{1,i,m},U_{2,i,m}),
\end{equation*}
with the tridiagonal matrix $A_{\alpha,i,m}$ in the form
\begin{equation*}
A_{\alpha,i,m}= \left[\begin{array}{ccccccc}
d_{\alpha,i,1,m}       &     -t_{\alpha,i,1,m}   &           &                 0          \\   [4 ex]
-b_{\alpha,i,2,m}      &     d_{\alpha,i,2,m}      & -t_{\alpha,i,2,m}     &                                 \\    [.5 ex]
        \quad\quad    \ddots   &          \ \ \ \     \ddots      & \ \ \ \   \ddots                      \\    [2 ex]
              &   - b_{\alpha,i,N_{y}-2,m}      &  \  d_{\alpha,i,N_{y}-2,m}        & -t_{\alpha,i,N_{y}-2,m}             \\    [4 ex]
    0          &              &        -b_{\alpha,i,N_{y}-1,m}      &        d_{\alpha,i,N_{y}-1,m}                     \\   
\end{array}\right]
.
\end{equation*}
Matrices $L_{\alpha,i,m}$ and $R_{\alpha,i,m}$ contain the coupling coefficients of a mesh point, respectively, to the mesh point of the left line and the mesh point of the right line.

We introduce the notation for the residuals of the nonlinear difference scheme (\ref{bp18}) in the  form
\begin{eqnarray}\label{bp18a}
&&\mathcal{G}_{\alpha,i,m}(U_{\alpha,i,m}, U_{\alpha,i, m-1}, U_{\alpha^{\prime},i, m}) =\\
&&A_{\alpha,i,m}U_{\alpha,i,m}-(L_{\alpha,i,m}U_{\alpha,i-1,m}+R_{\alpha,i,m}U_{\alpha,i+1,m}) \nonumber \\
&&+F_{\alpha,i,m}(U_{\alpha,i,m},  U_{\alpha^{\prime},i,m})-\tau^{-1}U_{\alpha,i,m-1}+G_{\alpha,i,m}^{*},\quad  i=1,2,\ldots,N_{x}-1,\nonumber \\
&& \alpha^{\prime}\neq\alpha,\quad \alpha, \alpha^{\prime}=1,2, \nonumber
\end{eqnarray}
where
\begin{align*}
&F_{\alpha,i,m}(U_{\alpha,i,m},U_{\alpha^{\prime},i,m})=\left\{ \begin{array}{ll}
F_{1,i,m}(U_{1,i,m},U_{2,i,m}),\quad \alpha=1,\\
F_{2,i,m}(U_{1,i,m},U_{2,i,m}),\quad \alpha=2,
\end{array}\right.\\
 &i=0,1,\ldots,N_{x}.\nonumber
\end{align*}
\subsection{Block monotone Jacobi and Gauss-Seidel methods}
We now present the block monotone Jacobi and block monotone Gauss--Seidel  methods for the nonlinear
difference scheme (\ref{bp9}) when the reaction functions are quasi-monotone nondecreasing
   based on the method of upper and lower solutions.  We define functions $c_{\alpha,m}$, $\alpha=1,2$, $m\geq1$, in the following form
\begin{equation}\label{bp26a}
  c_{\alpha, m}=\max_{(x_{i},y_{j})\in\overline{\Omega}^{h}}c_{\alpha,ij,m},\quad \alpha=1,2,\quad m\geq1,
\end{equation}
where $c_{\alpha,ij,m}$, $ij\in \overline{\Omega}^{h}$, $\alpha=1,2$, are defined in (\ref{bp20}).
On each time level $t_{m}$, $m\geq1$, the upper $\{\overline{U}^{(n)}_{\alpha,i,m}\}$ and lower $\{\underline{U}^{(n)}_{\alpha,i,m}\}$, $\alpha=1,2$,
sequences of solutions are calculated by the following block  Jacobi and block Gauss-Seidel methods
\begin{align}\label{bp27}
&A_{\alpha,i,m}Z_{\alpha,i,m}^{(n)}-\eta L_{\alpha,i,m}Z^{(n)}_{\alpha,i-1,m}+c_{\alpha,m}Z_{\alpha,i,m}^{(n)}=\\
&-\mathcal{G}_{\alpha,i,m}\left(U_{\alpha,i,m}^{(n-1)},U_{\alpha, i,m-1}, U_{\alpha^{\prime},i,m}^{(n-1)}\right ),\quad  i=1,2,\ldots,  N_{x}-1, \ \ \alpha^{\prime}\neq \alpha,\nonumber\\
&\alpha, \alpha^{\prime}=1,2,\quad m\geq1,\nonumber
\end{align}
\begin{equation*}
Z_{\alpha,i,m}^{(n)}=\left\{ \begin{array}{ll}
g_{\alpha,i,m}-U_{\alpha,i,m}^{(0)}, \quad n=1  , \\
\mathbf{0} ,\quad\quad\quad\quad\quad \quad\ \  n\geq2,
\end{array}\right.
\quad i=0,N_{x},
 \end{equation*}
 \begin{equation*}
U_{\alpha,i,0}=\psi_{\alpha,i},\quad i=0,1,\ldots,N_{x},\quad U_{\alpha,i,m}=U^{(n_{m})}_{\alpha,i,m},
 \end{equation*}
where $U_{i,m}^{(n-1)}=(U_{1,i,m}^{(n-1)},U_{2,i,m}^{(n-1)})$,  $\mathcal{G}_{\alpha,i,m}\left(U_{\alpha,i,m}^{(n-1)},U_{\alpha, i,m-1}, U_{\alpha^{\prime},i,m}^{(n-1)}\right )$,
$\alpha^{\prime}\neq\alpha$,  $\alpha, \alpha^{\prime}=1,2$,
 are defined in (\ref{bp18a}),
   $\mathbf{0}$ is  zero column vector
with the  $N_{x}-1$ components, and $U_{\alpha,i,m}$, $i=0,1,\ldots,N_{x}$, $\alpha=1,2$, are the approximate  solutions on time level
 $m\geq1$, where $n_{m}$ is a number of iterations on time level   $m\geq1$.
For  $\eta=0$ and $\eta=1$, we have, respectively,  the block Jacobi and block  Gauss--Seidel methods.
\begin{remark}\label{bp28}
 Similar to Remark \ref{bp14}, we discuss the mean-value theorem for mesh vector-functions.
Assume that $F_{\alpha}(x,y,t,u_{\alpha}, u_{\alpha^{\prime}})$, $i=0,1,\ldots,N_{x}$, $\alpha\neq \alpha^{\prime}$, $\alpha, \alpha^{\prime}=1,2$, are smooth functions. In
the notation of $ F_{\alpha,i,m}(U_{\alpha,i,m},U_{\alpha^{\prime},i,m}) $ in (\ref{bp18a}),  we have
 \begin{align}\label{bp29}
  & F_{\alpha,i,m}(U_{\alpha,i,m},U_{\alpha^{\prime},i,m})-F_{\alpha,i,m}(V_{\alpha,i,m},U_{\alpha^{\prime},i,m})=\\
  &\frac{\partial F_{\alpha,i,m}(Y_{\alpha,i,m},U_{\alpha^{\prime},i,m})}{\partial  u_{\alpha^{\prime}}}[U_{\alpha,i,m}-V_{\alpha,i,m}],\nonumber
  \end{align}
  \begin{align*}
  & F_{\alpha,i,m}(U_{\alpha,i,m},U_{\alpha^{\prime},i,m})-F_{\alpha,i,m}(U_{\alpha,i,m},V_{\alpha^{\prime},i,m})=\\
  &\frac{\partial F_{\alpha,i,m}(U_{\alpha,i,m},Y_{\alpha^{\prime},i,m})}{\partial  u_{\alpha^{\prime}}}[U_{\alpha^{\prime},i,m}-V_{\alpha^{\prime},i,m}],
  \end{align*}
  where $Y_{\alpha,i,m}$ lie between $U_{\alpha,i,m}$ and $V_{\alpha,i,m}$, and  $Y_{\alpha^{\prime},i,m}$ lie between $U_{\alpha^{\prime},i,m}$ and $V_{\alpha^{\prime},i,m}$, $i=0,1,\ldots,N_{x}$, $ \alpha^{\prime}\neq\alpha $,  $\alpha, \alpha^{\prime}=1,2$, $m\geq1$.
  The partial derivatives $\frac{\partial F_{\alpha,i,m}}{\partial u_{\alpha}}$ and $\frac{\partial F_{\alpha,i,m}}{\partial u_{\alpha^{\prime}}}$, are the diagonal matrices
 \begin{equation*}
 \frac{\partial F_{\alpha,i,m}}{\partial u_{\alpha}}=  \mbox{diag} \left(\frac{\partial f_{\alpha,i,1,m}}{\partial u_{\alpha}}, \ldots, \frac{\partial f_{\alpha,i,N_{y}-1,m}}{\partial u_{\alpha}}\right),
 \end{equation*}
\begin{equation*}
 \frac{\partial F_{\alpha,i,m}}{\partial u_{\alpha^{\prime}}}= \mbox{diag} \left(\frac{\partial f_{\alpha,i,1,m}}{\partial u_{\alpha^{\prime}}}, \ldots, \frac{\partial f_{\alpha,i,N_{y}-1,m}}{\partial u_{\alpha^{\prime}}}\right),
 \end{equation*}
  where $\frac{\partial f_{\alpha,ij,m}}{\partial u_{\alpha}}$ and $\frac{\partial f_{\alpha,ij,m}}{\partial u_{\alpha^{\prime}}} $, $j=1,\ldots, N_{y}-1$, are calculated, respectively, at $Y_{\alpha,i,m}$ and $Y_{\alpha^{\prime},i,m}$,
 $i=1,2,\ldots, N_{x}-1$.
\end{remark}
\begin{theorem}\label{bp30}
Let  $f(p,t_{m},U)$ in (\ref{bp9})  satisfy  (\ref{bp20}) and (\ref{bp21}),
 where    $\widetilde{U}(p,t_{m})=(\widetilde{U}_{1}(p,t_{m}),\widetilde{U}_{2}(p,t_{m}))$ and
 $\widehat{U}(p,t_{m})=(\widehat{U}_{1}(p,t_{m}),\widehat{U}_{2}(p,t_{m}))$ are ordered upper and lower solutions (\ref{bp19}) of (\ref{bp9})
 . Then the upper $\{\overline{U}_{\alpha,i,m}^{(n)}\}$ and lower $\{\underline{U}_{\alpha,i,m}^{(n)}\}$, $i=0,1,\ldots,N_{x}$,
 $\alpha=1,2$, sequences  generated by (\ref{bp27}), with $\overline{U}^{(0)}(p,t_{m})=\widetilde{U}(p,t_{m})$ and
  $\underline{U}^{(0)}(p,t_{m})=\widehat{U}(p,t_{m})$,
 converge monotonically, such that,
\begin{equation}\label{bp31}
   \underline{U}_{\alpha,i,m}^{(n-1)}\leq \underline{U}_{\alpha,i,m}^{(n)}\leq  \overline{U}_{\alpha,i,m}^{(n)}\leq \overline{U}_{\alpha,i,m}^{(n-1)},\ \ i=0,1, \ldots, N_{x},\ \ \alpha=1,2, \ \ m\geq1.
  \end{equation}
\end{theorem}
\begin{proof}
We consider the case of Gauss-Seidel method $\eta=1$, and the case of the  Jacobi method can be proved by a similar manner.
On first time level $m=1$,
since $\overline{U}^{(0)}$ is an upper solution (\ref{bp19}) with respect to  $U_{\alpha}(p,0)=\psi_{\alpha}(p)$, from (\ref{bp27}), we have
\begin{equation}\label{bp32}
\left( A_{\alpha,i,1}+c_{\alpha,1}I\right)\overline{Z}_{\alpha,i,1}^{(1)}\leq L_{\alpha,i,1}\overline{Z}^{(1)}_{\alpha,i-1,1},\quad i=1,2,\ldots,N_{x}-1,\quad \alpha=1,2,
\end{equation}
where $I$ is the identity matrix. For $i=1$ in (\ref{bp32}), taking into account that $
L_{\alpha,i,1}\geq\mathbf{0}$, $i=1,2,\ldots,N_{x}-1$, and $\overline{Z}^{(1)}_{\alpha,0,1}\leq\mathbf{0}$, it follows that    $\left( A_{\alpha,1,1}+c_{\alpha,1}I\right)\overline{Z}_{\alpha,1,1}^{(1)}\leq\mathbf{0}$.
Taking into account that $d_{\alpha,ij}>0$, $b_{\alpha,ij}$, $t_{\alpha,ij}\geq0$, $\alpha=1,2$, in (\ref{bp17})
and $A_{\alpha,i,1}$ are strictly diagonal dominant matrix, we conclude that  $A_{\alpha,i,1}$, $i=1,2,\ldots,N_{x}-1$, $\alpha=1,2$,
are $M$-matrices  and $A_{\alpha,i,1}^{-1}\geq\emph{O}$ (Corollary 3.20, \cite{Varga2000}),   which leads to $(A_{\alpha,i,1}+c_{\alpha,1}I)^{-1}\geq\emph{O}$, where $\emph{O}$ is the ($N_{y}-1)\times (N_{y}-1$) null  matrix. From here, we obtain
\begin{equation*}
\overline{Z}_{\alpha,1,1}^{(1)}\leq \mathbf{0},\quad \alpha=1,2.
\end{equation*}
Taking into account that $\overline{Z}_{\alpha,1,1}^{(1)}\leq \mathbf{0}$, for $i=2$ in (\ref{bp32}), in a similar manner, we conclude that
 \begin{equation*}
\overline{Z}_{\alpha,2,1}^{(1)}\leq \mathbf{0},\quad \alpha=1,2.
\end{equation*}
By induction on $i$, we can prove that
\begin{equation}\label{bp33}
\overline{Z}_{\alpha,i,1}^{(1)}\leq \mathbf{0},\quad i=0,1,\ldots,N_{x},\quad \alpha=1,2.
\end{equation}
 Similarly, for the lower solution $\underline{U}^{(0)}=\widehat{U}$, we have
 \begin{equation}\label{bp34}
   \underline{Z}_{\alpha,i,1}^{(1)}\geq \mathbf{0},\quad i=0,1,\ldots,N_{x},\quad \alpha=1,2.
 \end{equation}
 We now prove that $\overline{U}^{(1)}_{\alpha,i,1}$ and  $\underline{U}^{(1)}_{\alpha,i,1}$, are ordered upper and lower solutions (\ref{bp19})
 with respect to the  column vector  $U_{\alpha,i,0}=\psi_{\alpha,i}$, where the column vector $\psi_{\alpha,i}$ is associated with the initial
 function $\psi(x,y)$ from (\ref{bp1}). Let $W^{(1)}_{\alpha,i,1}=\overline{U}^{(1)}_{\alpha,i,1}-\underline{U}^{(1)}_{\alpha,i,1}$, $i=0,1,\ldots,N_{x}$, $\alpha=1,2$, from (\ref{bp27})
 for $\alpha=1$, we have
 \begin{align}\label{bp35}
&\left( A_{1,i,1}+c_{1,1}I\right)W_{1,i,1}^{(1)}-L_{1,i,1}W_{1,i-1,1}^{(1)}= c_{1,1}W_{1,i,1}^{(0)}+R_{1,i,1}W_{1,i+1,1}^{(0)}\nonumber \\
  & -\left[F_{1,i,1}(\overline{U}_{1,i,1}^{(0)},\overline{U}_{2,i,1}^{(0)})-F_{1,i,1}(\underline{U}_{1,i,1}^{(0)},\overline{U}_{2,i,1}^{(0)})\right] \nonumber \\
  -& \left[F_{1,i,1}(\underline{U}_{1,i,1}^{(0)},\overline{U}_{2,i,1}^{(0)})-F_{1,i,1}(\underline{U}_{1,i,1}^{(0)},\underline{U}_{2,i,1}^{(0)})\right] \nonumber,\\
  & i=1,2,\ldots,N_{x}-1,\quad W^{(1)}_{1,i,1}=\mathbf{0},\quad i=0,N_{x}.
  \end{align}
  By the mean-value theorem (\ref{bp29}), we have
\begin{eqnarray*}
&&F_{1,i,1}(\overline{U}_{1,i,1}^{(0)},\overline{U}_{2,i,1}^{(0)})-F_{1,i,1}(\underline{U}_{1,i,1}^{(0)},\overline{U}_{2,i,1}^{(0)})=\\
&&\frac{\partial F_{1,i,1}(Q_{1,i,1}^{(0)},\overline{U}_{2,i,1}^{(0)})}{\partial u_{1}}\left[\overline{U}_{1,i,1}^{(0)}-\underline{U}_{1,i,1}^{(0)}\right],
\end{eqnarray*}
\begin{eqnarray*}
&&F_{1,i,1}(\underline{U}_{1,i,1}^{(0)},\overline{U}_{2,i,1}^{(0)})-F_{1,i,1}(\underline{U}_{1,i,1}^{(0)},\underline{U}_{2,i,1}^{(0)})=\\
&&\frac{\partial F_{1,i,1}(\underline{U}_{1,i,1}^{(0)},Q_{2,i,1}^{(0)})}{\partial u_{2}}\left[\overline{U}_{2,i,1}^{(0)}-\underline{U}_{2,i,1}^{(0)}\right],
\end{eqnarray*}
where $\underline{U}_{\alpha,i,1}^{(0)}\leq Q_{\alpha,i,1}^{(0)}\leq \overline{U}_{\alpha,i,1}^{(0)}$, $i=1,2,\ldots,N_{x}-1$, $\alpha=1,2$, and
\begin{equation*}
\frac{\partial F_{1,i,1}(Q_{1,i,1}^{(0)},\overline{U}_{2,i,1}^{(0)})}{\partial u_{1}}= \quad\quad\quad\quad\quad\quad\quad\quad\quad\quad\quad\quad\quad\quad\quad\quad\quad\quad\quad\quad\quad\quad\
\end{equation*}
\begin{equation*}
 \mbox{diag} \left(\frac{\partial f_{1,i,1,1}}{\partial u_{1}}(Q_{1,i,1,1}^{(0)},\overline{U}_{2,i,1,1}^{(0)}),\ldots,\frac{\partial f_{1,i,N_{y}-1,1}}{\partial u_{1}}(Q_{1,i,N_{y}-1,1}^{(0)},\overline{U}_{2,i,N_{y}-1,1}^{(0)}) \right),
\end{equation*}
\begin{equation*}
\frac{\partial F_{1,i,1}(\overline{U}_{1,i,1}^{(0)},Q_{2,i,1}^{(0)})}{\partial u_{2}}= \quad\quad\quad\quad\quad\quad\quad\quad\quad\quad\quad\quad\quad\quad\quad\quad\quad\quad\quad\quad\quad\quad\
\end{equation*}
\begin{equation*}
 \mbox{diag} \left(\frac{\partial f_{1,i,1,1}}{\partial u_{2}}(\overline{U}_{1,i,1,1}^{(0)},Q_{2,i,1,1}^{(0)}),\ldots,\frac{\partial f_{1,i,N_{y}-1,1}}{\partial u_{2}}(\overline{U}_{1,i,N_{y}-1,1}^{(0)},Q_{2,i,N_{y}-1,1}^{(0)})\right).
\end{equation*}
From here, we conclude that $\frac{\partial F_{1,i,1}}{\partial u_{1}}$, $\frac{\partial F_{1,i,1}}{\partial u_{2}}$
 satisfy (\ref{bp20}) and (\ref{bp21}). From here and (\ref{bp35}), we have
 \begin{eqnarray}\label{bp36}
 \left(A_{1,i,1}+c_{1,1}I\right)W_{1,i,1}^{(1)}-L_{1,i,1}W_{1,i-1,1}^{(1)}&=&\left(c_{1,1}-\frac{\partial F_{1,i,1}}{\partial u_{1}}\right)W_{1,i,1}^{(0)}\\
 &&- \frac{\partial F_{1,i,1}}{\partial u_{2}}W_{2,i,1}^{(0)}+ R_{1,i,1}W_{1,i+1,1}^{(0)}, \nonumber
 \end{eqnarray}
 \begin{equation*}
  i=1,2,\ldots,N_{x}-1,\quad W_{1,i,1}^{(1)}=\mathbf{0},\quad i=0,N_{x}.
 \end{equation*}
  From here, (\ref{bp20}), (\ref{bp21}), taking into account that $W_{\alpha,i,1}^{(0)}\geq \mathbf{0}$, $i=0,1,\ldots,N_{x}$, $\alpha=1,2$, and $R_{1,i,1}\geq \emph{O}$, we obtain
\begin{equation}\label{bp37}
  \left(A_{1,i,1}+c_{1,1}I\right)W_{1,i,1}^{(1)}\geq L_{1,i,1}W_{1,i-1,1}^{(1)},\quad i=1,2,\ldots,N_{x}-1,
\end{equation}
\begin{equation*}
  W_{1,i,1}^{(1)}=\mathbf{0},\quad i=0,N_{x}.
\end{equation*}
Taking into account that $(A_{1,i,1}+c_{1,1}I)^{-1}\geq \emph{O}$ (Corollary 3.20, \cite{Varga2000}), $i=1,2,\ldots,N_{x}-1$,
for $i=1$ in (\ref{bp37}) and $W_{1,0,1}^{(1)}=\mathbf{0}$, we conclude that $W_{1,1,1}^{(1)}\geq \mathbf{0}$. For $i=2$ in (\ref{bp37}), using $L_{1,2,1}\geq\emph{O}$ and $W_{1,1,1}^{(1)}\geq\mathbf{0}$, we obtain $W_{1,2,1}^{(1)}\geq \mathbf{0}$. Thus, by induction on $i$, we prove that
\begin{equation*}
W_{1,i,1}^{(1)}\geq \mathbf{0},\quad i=0,1,\ldots,N_{x}.
\end{equation*}
By a similar argument, we can prove that
\begin{equation*}
W_{2,i,1}^{(1)}\geq \mathbf{0},\quad i=0,1,\ldots,N_{x}.
\end{equation*}
Thus, we prove  (\ref{bp19a}).  We now prove (\ref{bp19b}). From (\ref{bp27}) for $\alpha=1$ and using the mean-value theorem (\ref{bp29}), we conclude that
\begin{eqnarray}\label{bp38}
&&\mathcal{G}_{1,i,1}\left(\overline{U}_{1,i,1}^{(1)}, \psi_{1,i}, \overline{U}_{2,i,1}^{(1)}\right)=\\
&&-\left(c_{1,1}-\frac{\partial F_{1,i,1}(\overline{E}^{(1)}_{1,i,1},\overline{U}^{(0)}_{2,i,1})}{\partial u_{1}}\right) \overline{Z}_{1,i,1}^{(1)}+\frac{\partial F_{1,i,1}(\overline{U}^{(0)}_{1,i,1},\overline{E}^{(1)}_{2,i,1})}{\partial u_{2}}\overline{Z}_{2,i,1}^{(1)}\nonumber \\
&&-R_{1,i,1}\overline{Z}_{1,i+1,1}^{(1)},\quad  i=1,2, \ldots, N_{x}-1, \nonumber
\end{eqnarray}
where
\begin{equation*}
\overline{U}^{(1)}_{\alpha,i,1}\leq \overline{E}^{(1)}_{\alpha,i,1}\leq \overline{U}^{(0)}_{\alpha,i,1},\quad i=0,1,\ldots,N_{x},\quad \alpha=1,2.
\end{equation*}
 From (\ref{bp33}), (\ref{bp34}), taking into account that $\underline{U}^{(1)}_{\alpha,i,1}\leq \overline{U}^{(1)}_{\alpha,i,1}$, $\alpha=1,2$, we conclude that $\frac{\partial F_{1,i,1}}{\partial u_{1}}$ and $\frac{\partial F_{1,i,1}}{\partial u_{2}}$ satisfy
 (\ref{bp20}) and (\ref{bp21}). From (\ref{bp20}), (\ref{bp21}), (\ref{bp33}) and taking into account that $R_{1,i,1}\geq \emph{O}$, $i=1,2,\ldots,N_{x}-1$, we conclude that
 \begin{equation}\label{bp39}
\mathcal{G}_{1,i,1}\left(\overline{U}_{1,i,1}^{(1)}, \psi_{1,i}, \overline{U}_{2,i,1}^{(1)}\right)\geq \mathbf{0},\quad i=1,2,\ldots,N_{x}-1,
 \end{equation}
 Similarly, we obtain
 \begin{equation*}
\mathcal{G}_{2,i,1}\left(\overline{U}_{2,i,1}^{(1)}, \psi_{2,i},\overline{U}_{1,i,1}^{(1)} \right)\geq \mathbf{0},\quad i=1,2,\ldots,N_{x}-1,
 \end{equation*}
 which means that $\overline{U}_{\alpha,i,1}^{(1)}$, $i=0,1, \ldots, N_{x}$, $\alpha=1,2$, are  upper solution  (\ref{bp19b}) on $m=1$. By a similar manner, we can prove that
 \begin{equation*}
\mathcal{G}_{\alpha,i,1}\left(\underline{U}_{\alpha,i,1}^{(1)}, \psi_{\alpha,i}, \underline{U}_{\alpha^{\prime},i,1}^{(1)}\right)\leq \mathbf{0},\quad i=1,2,\ldots,N_{x}-1,\quad \alpha=1,2,
 \end{equation*}
 which means that $\underline{U}_{\alpha,i,1}^{(1)}$, $i=0,1, \ldots, N_{x}$,  $\alpha=1,2$, are lower solutions (\ref{bp19b}) on $m=1$. By induction on $n$, we  prove (\ref{bp31})
 on the  first time level $m=1$.

On the second time level $m=2$, taking into account that $\overline{U}_{\alpha,i,2}^{(0)}=\widetilde{U}_{\alpha,i,2}$, $i=0,1, \ldots, N_{x}$,  $\alpha=1,2$, from
(\ref{bp9}), we obtain
\begin{align*}
&\mathcal{G}_{\alpha,i,2}\left(\widetilde{U}_{\alpha,i,2},\overline{U}_{\alpha,i,1}, \widetilde{U}_{\alpha^{\prime},i,2}\right )= \\&A_{\alpha,i,2}\widetilde{U}_{\alpha,i,2}-L_{\alpha,i,2}\widetilde{U}_{\alpha,i-1,2}
-R_{\alpha,i,2}\widetilde{U}_{\alpha,i+1,2}+F_{\alpha,i,2}(\widetilde{U}_{i,2})-\tau^{-1}\overline{U}_{\alpha,i,1} \\&+G_{\alpha,i,2}^{*},\quad i=1,2\ldots,N_{x}-1,\quad
\alpha^{\prime}\neq \alpha,\quad  \alpha, \alpha^{\prime}=1,2,
\end{align*}
where $\overline{U}_{\alpha,i,1}$, $i=1,2,\ldots,N_{x}-1$, $\alpha=1,2$,  are the approximate  solutions on the first time level $m=1$,
which defined in (\ref{bp27}).
From here, taking into account that from (\ref{bp31}), we have   $\overline{U}_{\alpha,i,1}\leq \widetilde{U}_{\alpha,i,1}$, $i=0,1,\ldots,N_{x}$, $\alpha=1,2$,
it follows that
\begin{eqnarray*}
&&  \mathcal{G}_{\alpha,i,2}\left(\widetilde{U}_{\alpha,i,2},\overline{U}_{\alpha,i,1}, \widetilde{U}_{\alpha^{\prime},i,2}\right )\geq \mathcal{G}_{\alpha,i,2}\left(\widetilde{U}_{\alpha,i,2},\widetilde{U}_{\alpha,i,1},\widetilde{U}_{\alpha^{\prime},i,2}\right )\geq\mathbf{0},\\
&& i=1,2, \ldots,N_{x}-1,\quad \alpha^{\prime}\neq \alpha,\quad  \alpha, \alpha^{\prime}=1,2,
\end{eqnarray*}
which means that  $\overline{U}_{\alpha,i,2}^{(1)}=\widetilde{U}_{\alpha,i,2}$, $i=0,1,\ldots,N_{x}$, $\alpha=1,2$, are  upper
solutions with respect to $\overline{U}_{\alpha,i,1}$, $i=0,1,\ldots,N_{x}$, $\alpha=1,2$. Similarly, we can obtain that
\begin{equation*}
  \mathcal{G}_{\alpha,i,2}\left(\widehat{U}_{\alpha,i,2},\underline{U}_{\alpha,i,1}, \widehat{U}_{\alpha^{\prime},i,2}\right )\leq\mathbf{0},\ \  i=1,2,\ldots,N_{x}-1,\ \
  \alpha^{\prime}\neq \alpha,\ \  \alpha, \alpha^{\prime}=1,2,
\end{equation*}
which means that  $\underline{U}_{\alpha,i,2}^{(1)}=\widehat{U}_{\alpha,i,2}$, $i=0,1,\ldots,N_{x}$, $\alpha=1,2$, are  lower
solutions with respect to $\underline{U}_{\alpha,i,1}$, $i=0,1,\ldots,N_{x}$, $\alpha=1,2$.

From here and  (\ref{bp27}), on the  second time level $m=2$, we obtain
\begin{equation}\label{bp47}
\left( A_{\alpha,i,2}+c_{\alpha,2}I\right)\overline{Z}_{\alpha,i,2}^{(1)}\leq L_{\alpha,i,2}\overline{Z}^{(1)}_{\alpha,i-1,2},\quad i=1,2,\ldots,N_{x}-1,\quad \alpha=1,2.
\end{equation}
Taking into account that $d_{\alpha,ij}>0$, $b_{\alpha,ij}$, $t_{\alpha,ij}\geq0$, $(x_{i},y_{j})\in \Omega^{h}$, $\alpha=1,2$,
in (\ref{bp17}) and $A_{\alpha,i,2}$, $i=1,2,\ldots,N_{x}-1$, $\alpha=1,2$, are strictly diagonal dominant matrix,
 we conclude that  $A_{\alpha,i,2}+c_{\alpha,2}I$, $i=1,2,\ldots,N_{x}-1$, $\alpha=1,2$,
 are $M$-matrices  and $(A_{\alpha,i,2}+c_{\alpha,2}I)^{-1}\geq\emph{O}$, $i=1,2,\ldots,N_{x}-1$, $\alpha=1,2$,
 (Corollary 3.20, \cite{Varga2000}). From here, for $i=1$ in (\ref{bp47}),  taking into account that
$L_{\alpha,i,2}\geq\emph{O}$, $i=1,2,\ldots,N_{x}-1$, and   $\overline{Z}^{(1)}_{\alpha,0,2}\leq\mathbf{0}$ from (\ref{bp27}), we obtain that
\begin{equation*}
\overline{Z}_{\alpha,1,2}^{(1)}\leq \mathbf{0},\quad \alpha=1,2.
\end{equation*}
From here, for $i=2$ in (\ref{bp47}), we conclude that
 \begin{equation*}
\overline{Z}_{\alpha,2,2}^{(1)}\leq \mathbf{0},\quad \alpha=1,2.
\end{equation*}
By induction on $i$, we can prove that
\begin{equation}\label{bp48}
\overline{Z}_{\alpha,i,2}^{(1)}\leq \mathbf{0},\quad i=0,1,\ldots,N_{x},\quad \alpha=1,2.
\end{equation}
 Similarly, for the lower case, we can prove that
 \begin{equation}\label{bp49}
   \underline{Z}_{\alpha,i,2}^{(1)}\geq \mathbf{0},\quad i=0,1,\ldots,N_{x},\quad \alpha=1,2.
 \end{equation}
The proof that $\overline{U}^{(1)}_{\alpha,i,2}$ and  $\underline{U}^{(1)}_{\alpha,i,2}$, $\alpha=1,2$,
are ordered upper and lower solutions (\ref{bp19})
 repeats  the proof on the first
 time level $m=1$. By induction on $n$, we can prove  $(\ref{bp31})$ for $m=1$.
By induction on $m$, we can prove (\ref{bp31}) for  $m\geq1$.

\end{proof}
 \subsection{Existence and uniqueness of  a solution to the nonlinear difference   scheme (\ref{bp18})}
In the following theorem, we prove the existence of a solution to  (\ref{bp18}) based on Theorem \ref{bp30}.
\begin{theorem}\label{bp55}
 Let  $f(p,t_{m},U)$ satisfy (\ref{bp20}), where $\widetilde{U}_{\alpha,i,m}$ and $\widehat{U}_{\alpha,i,m}$,
  $i=0,1\ldots,N_{x}$, $\alpha=1,2$,
   $m\geq1$, be ordered upper and lower solutions (\ref{bp19}) to (\ref{bp18}).
  Then a solution of the nonlinear  difference  scheme (\ref{bp18})
  exists in $\langle\widehat{U}(t_{m}),\widetilde{U}(t_{m})\rangle$, $m\geq1$.
  \end{theorem}
\begin{proof}
We consider the upper case of the  Gauss--Seidel method $(\eta=1)$ in (\ref{bp27}).
 On the first time level $t_{1}$, from (\ref{bp31}), we conclude that
  $\lim \overline{U}^{(n)}_{\alpha,i,1}=\overline{V}_{\alpha,i,1}$, $i=0,1,\ldots,N_{x}$, $\alpha=1,2$,
  as $n\rightarrow \infty$ exists, and
  \begin{eqnarray}\label{bp56}
&&\overline{V}_{\alpha,i,1}\leq   \overline{U}^{(n)}_{\alpha,i,1}\leq
\overline{U}^{(n-1)}_{\alpha,i,1}\leq \widetilde{U}_{\alpha,i,1} ,\quad \lim_{n \rightarrow \infty} \overline{Z}^{(n)}_{\alpha,i,1}=\mathbf{0},\quad i=0,1,\ldots,N_{x},\nonumber\\
&& \alpha=1,2,
  \end{eqnarray}
  where $\overline{U}^{(0)}_{\alpha,i,1}=\widetilde{U}_{\alpha,i,1}$. Similar to (\ref{bp38}), we have
  \begin{eqnarray}\label{bp57}
&&\mathcal{G}_{\alpha,i,1}\left(\overline{U}_{\alpha,i,1}^{(n)}, \psi_{\alpha,i}, \overline{U}_{\alpha^{\prime},i,1}^{(n)}\right)=
-\left(c_{\alpha,1}-\frac{\partial F_{\alpha,i,1}(\overline{E}^{(n)}_{\alpha,i,1},\overline{U}^{(n-1)}_{\alpha^{\prime},i,1})}{\partial u_{1}}\right) \overline{Z}_{\alpha,i,1}^{(n)}\nonumber\\ &&+\frac{\partial F_{\alpha,i,1}(\overline{U}^{(n-1)}_{\alpha,i,1},\overline{E}^{(1)}_{\alpha^{\prime},i,1})}{\partial u_{2}}\overline{Z}_{\alpha^{\prime},i,1}^{(n)}-R_{\alpha,i,1}\overline{Z}_{\alpha,i+1,1}^{(n)}, \\
&& i=1,2, \ldots, N_{x}-1,\quad \alpha^{\prime}\neq \alpha,\quad  \alpha, \alpha^{\prime}=1,2,\nonumber
\end{eqnarray}
where
\begin{equation*}
\overline{U}^{(n)}_{\alpha,i,1}\leq \overline{E}^{(n)}_{\alpha,i,1}\leq \overline{U}^{(n-1)}_{\alpha,i,1},\quad i=0,1,\ldots,N_{x},\quad \alpha=1,2.
\end{equation*}
From here and  (\ref{bp56}), we conclude that $\overline{V}_{\alpha,i,1}$, $i=0,1,\ldots,N_{x}$,  $\alpha=1,2$,  solve (\ref{bp18}) on
the first time level $m=1$.

By the assumption of the theorem that $\widetilde{U}_{\alpha,i,2}$, $i=0,1,\ldots,N_{x}$, $\alpha=1,2$,
are upper solutions and from (\ref{bp56}), it follows that $\widetilde{U}_{\alpha,i,2}$, $i=0,1,\ldots,N_{x}$, $\alpha=1,2$, are upper
solutions with respect to $\overline{V}_{\alpha,i,1}$, $i=0,1,\ldots,N_{x}$, $\alpha=1,2$. Indeed, from (\ref{bp56}),  we have
\begin{align*}
&\mathcal{G}_{\alpha,i,2}\left(\widetilde{U}_{\alpha,i,2},\overline{V}_{\alpha,i,1}, \widetilde{U}_{\alpha^{\prime},i,2}\right )= \\&A_{\alpha,i,2}\widetilde{U}_{\alpha,i,2}-L_{\alpha,i,2}\widetilde{U}_{\alpha,i-1,2}
-R_{\alpha,i,2}\widetilde{U}_{\alpha,i+1,2}+F_{\alpha,i,2}(\widetilde{U}_{i,2})-\tau^{-1}\overline{V}_{\alpha,i,1} \\&+G_{\alpha,i,2}^{*}\geq \mathcal{G}_{\alpha,i,2}\left(\widetilde{U}_{\alpha,i,2},\widetilde{U}_{\alpha,i,1}, \widetilde{U}_{\alpha^{\prime},i,2}\right )\geq \mathbf{0},\\
&i=1,2\ldots,N_{x}-1,\quad \alpha^{\prime}\neq \alpha,\quad  \alpha, \alpha^{\prime}=1,2.
\end{align*}
Using a similar argument as in (\ref{bp56}), we can prove that the limits
\begin{equation*}
  \lim_{n\rightarrow \infty} \overline{U}^{(n)}_{\alpha,i,2}=\overline{V}_{\alpha,i,2},\quad i=0,1,\ldots,N_{x},\quad \alpha=1,2,
\end{equation*}
exist and solve (\ref{bp18}) on the second time level $m=2$.

By induction on $m$, $m\geq1$, we can prove that
\begin{equation*}
 \lim_{n\rightarrow \infty} \overline{U}^{(n)}_{\alpha,i,m}= \overline{V}_{\alpha,i,m},\quad i=0,1,\ldots,N_{x},\quad \alpha=1,2,\quad m\geq1,
\end{equation*}
are solutions of the nonlinear difference scheme (\ref{bp18}). Similarly, we can prove that $\underline{V}_{\alpha,i,m}$ defined by
\begin{equation*}
\lim_{n\rightarrow \infty} \underline{U}^{(n)}_{\alpha,i,m}= \underline{V}_{\alpha,i,m},\quad i=0,1,\ldots,N_{x},\quad \alpha=1,2,\quad m\geq1,
\end{equation*}
are  solutions to the nonlinear difference scheme (\ref{bp18}).
\end{proof}
We now assume that the reaction functions $f_{\alpha}$, $\alpha=1,2$, satisfy the two-sided constrains
\begin{equation}\label{bp61}
  \underline{c}_{\alpha}(p,t_{m})\leq \frac{\partial f_{\alpha}(p,t_{m},U)}{\partial u_{\alpha}}\leq c_{\alpha}(p,t_{m}),\quad
  U\in\langle \widehat{U}(t_{m}), \widetilde{U}(t_{m})\rangle, \quad p\in\overline{\Omega}^{h},
\end{equation}
\begin{equation}\label{bp61a}
   0\leq - \frac{\partial f_{\alpha}(p,t_{m},U)}{\partial u_{\alpha^{\prime}}}\leq q_{\alpha}(p,t_{m}),\quad U\in\langle \widehat{U}(t_{m}), \widetilde{U}(t_{m})\rangle, \quad p\in\overline{\Omega}^{h},
\end{equation}
\begin{equation*}
\alpha^{\prime}\neq \alpha,\quad \alpha, \alpha^{\prime}=1,2,
  \quad m\geq1,
\end{equation*}
where $c_{\alpha}(p,t_{m})$ is defined in (\ref{bp20}), $q_{\alpha}(p,t_{m})$ and  $\underline{c}_{\alpha}(p,t_{m})$, $\alpha=1,2$,
are, respectively, nonnegative bounded and bounded functions.
It is assumed that the time step $\tau$ satisfies the assumptions
\begin{eqnarray}\label{bp62}
&&\tau <\max_{m\geq1} \frac{1}{\beta_{m}}, \\
&&\beta_{m}=\max \left(0,q_{m}-\underline{c}_{m}\right)=  \left\{ \begin{array}{ll}
0,\quad \mbox{if}\  q_{m}- \underline{c}_{m}\leq0,\\ \\
q_{m}- \underline{c}_{m},\quad \mbox{if}\ q_{m}- \underline{c}_{m}>0, \nonumber
\end{array}\right. \\
 && \underline{c}_{m}=\min_{\alpha=1,2}\left[\min_{p\in\overline{\Omega}^{h}}\underline{c}_{\alpha}(p,t_{m})\right],\quad q_{m}=\max_{\alpha=1,2}\|q_{\alpha}(\cdot,t_{m})\|_{\overline{\Omega}^{h}},\nonumber
\end{eqnarray}
the notation of the discrete norm from (\ref{bp13}) is in use.
 When  $\beta_{m}=0$, $m\geq1$, then there is no restriction on  $\tau$.
\begin{theorem}\label{bp63}
 Suppose that functions $f_{\alpha}(p,t_{m},U)$, $\alpha=1,2$, satisfy (\ref{bp61})
 and (\ref{bp61a}),
 where $\widetilde{U}(p,t_{m})$ and $\widehat{U}(p,t_{m})$ are ordered
 upper and lower solutions (\ref{bp19}) of (\ref{bp9}). Let assumptions in
  (\ref{bp62}) on time step $\tau$ be satisfied. Then the  nonlinear
 difference scheme (\ref{bp9}) has a unique solution.
\end{theorem}
\begin{proof}
  To prove the uniqueness of a solution to the  nonlinear difference scheme (\ref{bp9}), it suffices to prove that
  \begin{equation*}
    \overline{V}_{\alpha}(p,t_{m})=\underline{V}_{\alpha}(p,t_{m}),\quad p\in \overline{\Omega}^{h},\quad \alpha=1,2,\quad m\geq1,
  \end{equation*}
  where $\overline{V}_{\alpha}(p,t_{m})$ and $\underline{V}_{\alpha}(p,t_{m})$ are the  solutions  to
  the  nonlinear difference scheme (\ref{bp9}), which are   defined in Theorem \ref{bp55}. From (\ref{bp31}) and Theorem \ref{bp55}, we obtain
  \begin{eqnarray}\label{bp64}
 &&  \underline{U}_{\alpha}^{(n)}(p,t_{m})\leq \underline{V}_{\alpha}(p,t_{m})\leq \overline{V}_{\alpha}(p,t_{m})\leq \overline{U}^{(n)}_{\alpha}(p,t_{m}),\quad p\in\overline{\Omega}^{h},\ \alpha=1,2,\nonumber\\
 &&m\geq1.
  \end{eqnarray}
  Letting $W_{\alpha}(p,t_{m})=\overline{V}_{\alpha}(p,t_{m})-\underline{V}_{\alpha}(p,t_{m})$, from (\ref{bp9}), we have
  \begin{eqnarray*}
&&\left(\mathcal{L}_{\alpha}^{h}(p,t_{m})+\tau^{-1}\right)W_{\alpha}(p,t_{m}) +f_{\alpha}(p,t_{m},\overline{V})- f_{\alpha}(p,t_{m},\underline{V})
\\
&&-\tau^{-1}W_{\alpha}(p,t_{m-1})=0,\quad p\in \Omega^{h},\quad W_{\alpha}(p,t_{m})=0,\quad p\in \partial \Omega^{h},\quad m\geq1.
\end{eqnarray*}
Using the mean-value theorem (\ref{bp16}), we obtain
\begin{eqnarray}\label{bp64a}
&&\left(\mathcal{L}_{\alpha}^{h}(p,t_{m})+\left(\tau^{-1}+\frac{\partial f_{\alpha}(p,t_{m},H_{\alpha})}{\partial u_{\alpha}}\right)\right)W_{\alpha}(p,t_{m})
=\\
&&-\frac{\partial f_{\alpha}(p,t_{m},H_{\alpha^{\prime}})}{\partial u_{\alpha^{\prime}}}W_{\alpha^{\prime}}(p,t_{m})+\tau^{-1}W_{\alpha}(p,t_{m-1}),
\quad p\in \Omega^{h}, \nonumber\\
&&W_{\alpha}(p,t_{m})=0,\quad (p,t_{m})\in \partial \Omega^{h\tau},\quad \underline{V}_{\alpha}(p,t_{m})\leq H_{\alpha}(p,t_{m})\leq \overline{V}_{\alpha}(p,t_{m}),\nonumber\\
&& \alpha^{\prime}\neq \alpha,\quad \alpha, \alpha^{\prime}=1,2.\nonumber
\end{eqnarray}
 From here and (\ref{bp64}), it follows that the partial derivatives satisfy (\ref{bp61}) and (\ref{bp61a}).
 If $\underline{c}_{m}\geq0$, from (\ref{bp64a}) for $m=1$,  using (\ref{bp13}),  (\ref{bp61}), (\ref{bp61a}) and
 taking into account that $W_{\alpha}(p,0)=0$, we conclude that
 \begin{equation*}
   W(t_{1})\leq\frac{\tau q_{1}}{1+\tau \underline{c}_{1}}W(t_{1}),
 \end{equation*}
where
\begin{eqnarray*}
&&  W(t_{m})=\max_{\alpha=1,2}W_{\alpha}(t_{m}),\quad W_{\alpha}(t_{m})=\|W_{\alpha}(\cdot,t_{m})\|_{\overline{\Omega}^{h}},\quad \alpha=1,2, \\
&& \|W_{\alpha}(\cdot,t_{m})\|_{\overline{\Omega}^{h}}=\max_{p\in\Omega^{h}}|W_{\alpha}(p,t_{m})|.
\end{eqnarray*}
From here, by the assumption on $\tau$ in (\ref{bp62}) and taking into account that $W(t_{m})\geq0$, we conclude that $W(t_{1})=0$.

On the second time level $m=2$, from (\ref{bp64a}) and taking into account that $W(t_{1})=0$, by a similar manner, we obtain that $W(t_{2})=0$.
By induction on $m$, we prove that $W(t_{m})=0$, $m\geq1$.
Thus, we prove the theorem when $\underline{c}_{m}\geq0$.

If $\underline{c}_{m}<0$, from (\ref{bp64a}) for $m=1$,  using (\ref{bp13}), (\ref{bp61}) and (\ref{bp61a}),  we conclude that
 \begin{equation*}
   W(t_{1})\leq\frac{\tau q_{1}}{1-\tau | \underline{c}_{1}|}W(t_{1}).
 \end{equation*}
 From here, by the assumption on $\tau$ in (\ref{bp62}) and taking into account that $W(t_{m})\geq0$, we conclude that $W(t_{1})=0$.

On the second time level $m=2$, from (\ref{bp64a}) and taking into account that $W(t_{1})=0$, by a similar manner, we obtain that $W(t_{2})=0$.
By induction on $m$, we prove that $W(t_{m})=0$, $m\geq1$.
Thus, we prove the theorem.
\end{proof}
\subsection{Convergence analysis }
A stopping test for the block monotone iterative methods (\ref{bp27}) is chosen in the following form
\begin{equation}\label{bp65}
 \max_{\alpha=1,2} \left\|\mathcal{G}_{\alpha}\left(U_{\alpha}^{(n)}(\cdot,t_{m}), U_{\alpha}(\cdot,t_{m-1}), U_{\alpha^{\prime}}^{(n)}(\cdot,t_{m})\right)\right\|_{\Omega^{h}}\leq \delta,
 \end{equation}
 \begin{eqnarray*}
  &&  \left\|\mathcal{G}_{\alpha}\left(U_{\alpha}^{(n)}(\cdot,t_{m}), U_{\alpha}(\cdot,t_{m-1}), U_{\alpha^{\prime}}^{(n)}(\cdot,t_{m})\right)\right\|_{\overline{\Omega}^{h}}= \\
&&\max_{p\in \overline{\Omega}^{h}} \left|\mathcal{G}_{\alpha}\left(U_{\alpha}^{(n)}(p,t_{m}), U_{\alpha}(p,t_{m-1}), U_{\alpha^{\prime}}^{(n)}(p,t_{m})\right)\right|,
 \end{eqnarray*}
where $ \mathcal{G}_{\alpha}\left(U_{\alpha}^{(n)}(p,t_{m}), U_{\alpha}(p,t_{m-1}), U_{\alpha^{\prime}}^{(n)}(p,t_{m})\right) $, $\alpha^{\prime}\neq \alpha$, $\alpha, \alpha^{\prime}=1,2$, are
defined in (\ref{bp18a}),  $U^{(n)}_{\alpha}(p,t_{m})$, $p\in\Omega^{h}$, $\alpha=1,2$,  are generated by (\ref{bp27}), and  $\delta$ is a
prescribed accuracy. On each time level $t_{m}$,  $m\geq1$,
we set up $U_{\alpha}(p,t_{m})=U^{(n_{m})}_{\alpha}(p,t_{m})$, $ p \in \Omega^{h}$, $\alpha=1,2$, $\alpha=1,2$,
 such that $m_{n}$
is the minimal subject to (\ref{bp65}).

Instead of (\ref{bp61}), we now assume that
\begin{eqnarray}\label{bp65a}
&&  q\leq \frac{\partial f_{\alpha}(x,y,t,u)}{\partial u_{\alpha}}\leq c_{\alpha}(x,y,t),\quad (x,y,t)\in Q_{T},\quad -\infty<u<\infty, \nonumber \\
&&\alpha=1,2,\quad  q=\max_{m\geq1} q_{m},
\end{eqnarray}
where $q_{m}$ is defined in (\ref{bp62}).
\begin{remark}\label{bp65b}
The assumption $\frac{\partial f_{\alpha}(p,t_{m},U)}{\partial u_{\alpha}}\geq q_{m}>0$, in (\ref{bp65a}) can always be
obtain by a change of variables. Indeed, we introduce the following functions
$z_{\alpha}(x,y,t)=\exp ^{-\lambda t} u_{\alpha}(x,y,t)$, $\alpha=1,2$, where $\lambda$ is a constant. Now, $z=(z_{1},z_{2})$ satisfy (\ref{bp1}) with
\begin{equation*}
  f_{\alpha}^{*}=\lambda z_{\alpha}+\exp^{-\lambda t} f_{\alpha}(x,y,t,\exp^{\lambda t}z),
\end{equation*}
instead of $f_{\alpha}$, $\alpha=1,2$, and we have
\begin{equation*}
  \frac{\partial f^{*}_{\alpha}}{\partial  z_{\alpha}}=\lambda+\frac{\partial f_{\alpha}}{\partial u_{\alpha}},\quad
  \frac{\partial f^{*}_{\alpha}}{\partial  z_{\alpha^{\prime}}}=\frac{\partial f_{\alpha}}{\partial u_{\alpha^{\prime}}},\quad \alpha^{\prime}\neq \alpha,\quad
  \alpha, \alpha^{\prime}=1,2.
\end{equation*}
Thus, if $\lambda\geq \max_{m\geq1}\left(q_{m}, |\underline{c}_{m}|\right)$, where $q_{m}$ and $\underline{c}_{m}$ are defined in (\ref{bp62}),
 then from this, (\ref{bp61}) and (\ref{bp61a}), we conclude that $ \frac{\partial f^{*}_{\alpha}}{\partial  z_{\alpha}}$ satisfies (\ref{bp65a}).
\end{remark}
\begin{theorem}\label{bp66}
Let $\widetilde{U}(p,t_{m})$ and $\widehat{U}(p,t_{m})$ be ordered
 upper and lower solutions (\ref{bp19}) of (\ref{bp9}). Suppose that functions $f_{\alpha}(p,t_{m},U)$, $\alpha=1,2$, satisfy (\ref{bp61a})
 and (\ref{bp65a}).
  Then for the sequence of solutions $\{U^{(n)}\}$ generated
 by (\ref{bp27}), (\ref{bp65}), we have the following estimate
 \begin{equation}\label{bp67}
\max_{m\geq1} \max_{\alpha=1,2}\|U_{\alpha}(\cdot, t_{m})-U_{\alpha}^{*}(\cdot, t_{m})\|_{\overline{\Omega}^{h}}\leq T \delta.
  \end{equation}
where $U_{\alpha}(p,t_{m})=U_{\alpha}^{(n_{m})}(p,t_{m})$, $n_{m}$ is a minimal
  number of iterations subject to (\ref{bp65}), and
  $U_{\alpha}^{*}(p,t_{m})$, $\alpha=1,2$, $m\geq1$, are the unique solutions to the  nonlinear difference scheme (\ref{bp9}).
  \end{theorem}
\begin{proof}
We consider the case of an upper sequence. On a time level $t_{m}$, $m\geq1$, from  (\ref{bp9})
for $\overline{U}_{\alpha}(p,t_{m})$ and $U^{*}_{\alpha}(p,t_{m})$ , we have
\begin{eqnarray*}
&&\left(\mathcal{L}_{\alpha}(p,t_{m})^{h}+\tau^{-1}\right)\overline{U}_{\alpha}(p,t_{m}) +f_{\alpha}(p,t_{m},\overline{U})-\tau^{-1} \overline{U}_{\alpha}(p,t_{m-1})=\\
&& \mathcal{G}_{\alpha}\left(\overline{U}_{\alpha}(p,t_{m}),\overline{U}_{\alpha}(p,t_{m-1}), \overline{U}_{\alpha^{\prime}}(p,t_{m})\right),\quad  p\in \Omega^{h},\ \ \alpha^{\prime}\neq \alpha, \\
&&\alpha,\alpha^{\prime}=1,2,\quad \overline{U}(p,t_{m})=g(p,t_{m}),\quad  p\in \partial \Omega^{h},\quad m\geq1,\\
&& \overline{U}(p,0)=\psi(p),\quad p\in\overline{\Omega}^{h},\nonumber
\end{eqnarray*}
\begin{eqnarray*}
&&\left(\mathcal{L}_{\alpha}^{h}(p,t_{m})+\tau^{-1}\right)U^{*}_{\alpha}(p,t_{m}) +f_{\alpha}(p,t_{m},U^{*})-\tau^{-1} U_{\alpha}^{*}(p,t_{m-1})=0,\\
&&p\in \Omega^{h},\quad \alpha=1,2,\quad U^{*}(p,t_{m})=g(p,t_{m}),\quad p\in \partial \Omega^{h},\\
&& U^{*}(p,0)=\psi(p),\quad p\in\overline{\Omega}^{h},\quad m\geq1.\nonumber
\end{eqnarray*}
Letting $W_{\alpha}(p,t_{m})=U_{\alpha}(p,t_{m})-U_{\alpha}^{*}(p,t_{m})$, $p\in \overline{\Omega}^{h}$, $\alpha=1,2$, $m\geq1$, from here and using the mean-value theorem, we obtain
\begin{eqnarray*}
&&\left(\mathcal{L}_{\alpha}^{h}(p,t_{m})+\left(\tau^{-1}+\frac{\partial f_{\alpha}(p,t_{m},K)}{\partial u_{\alpha}}\right)I\right)W_{\alpha}(p,t_{m})=
 \\
&&-\frac{\partial f_{\alpha}(p,t_{m},K)}{\partial u_{\alpha^{\prime}}}W_{\alpha^{\prime}}(p,t_{m}) +\mathcal{G}_{\alpha}\left(\overline{U}_{\alpha}(p,t_{m}),\overline{U}_{\alpha}(p,t_{m-1}), \overline{U}_{\alpha^{\prime}}(p,t_{m})\right) \\
&& +\tau^{-1} W_{\alpha}(p,t_{m-1}),\quad p\in \Omega^{h},\quad  W_{\alpha}(p,t_{m})=0,\quad p\in \partial \Omega^{h},\\
&&W_{\alpha}(p,0)=0,\quad p\in\overline{\Omega}^{h},\quad  \alpha^{\prime}\neq \alpha,\quad  \alpha,  \alpha^{\prime}=1,2,\quad  m\geq1,
\end{eqnarray*}
where
\begin{equation*}
U_{\alpha}^{*}(p,t_{m})\leq   K_{\alpha}(p,t_{m})\leq  \overline{U}_{\alpha}(p,t_{m}),\quad
  \alpha=1,2,\quad m\geq1.
\end{equation*}
The partial derivatives satisfy (\ref{bp61a}) and (\ref{bp65a}). From here, (\ref{bp61a}) and  (\ref{bp65a}), by using (\ref{bp13}), we obtain that
\begin{eqnarray*}
&&W_{\alpha}(t_{m}) \leq\frac{1}{\tau^{-1}+q} \left( q  W_{\alpha^{\prime}}(t_{m})+ \delta +\tau^{-1} W_{\alpha}(t_{m-1})   
\right),\\
 && W_{\alpha}(t_{m})=\max \|W_{\alpha}(\cdot,t_{m})\|_{\overline{\Omega}^{h}},\quad \alpha\neq \alpha^{\prime},\quad \alpha, \alpha^{\prime}=1,2.
\end{eqnarray*}
where the notation of the norm from (\ref{bp13}) is in use. From here, in the notation  $W_{m}=\max_{\alpha=1,2}W_{\alpha}(t_{m})$, we have
\begin{equation*}
W_{m} \leq\frac{1}{\tau^{-1}+q} \left( q  W_{m}+ \delta +\tau^{-1} W_{m-1}   
\right).
\end{equation*}
Taking into account that
\begin{equation*}
1-  \frac{q}{\tau^{-1}+q}>0,
\end{equation*}
it follows that
\begin{equation*}
W_{m} \leq \tau \delta + W_{m-1}.   
\end{equation*}
From here, taking into account that $W_{0}=0$, by induction on $m$,  we obtain that
\begin{equation*}
W_{m} \leq \delta \sum_{\rho=1}^{m} \tau .
\end{equation*}
Since $ \sum_{\rho=1}^{m} \tau\leq T $, we prove the theorem.
\end{proof}
\begin{theorem}\label{bp68}
  Let the assumptions in Theorem \ref{bp66} be satisfied.  Then for the sequence of solutions $\{U^{(n)}\}$ generated
 by (\ref{bp27}), (\ref{bp65}),  the following estimate holds
\begin{eqnarray}\label{bp69}
&& \max_{m\geq1}  \max_{\alpha=1,2}\|U_{\alpha}(\cdot, t_{m})-u_{\alpha}^{*}(\cdot, t_{m})\|_{\overline{\Omega}^{h}}\leq T\left(\delta +\max_{m\geq1} E_{m}\right), \\
&& E_{m}=\max_{\alpha=1,2}\|E_{\alpha}(\cdot,t_{m})\|_{\overline{\Omega}^{h}},\quad m\geq1, \nonumber
\end{eqnarray}
where the notation of the norm from (\ref{bp13}) is in use, $U_{\alpha}(p,t_{m})=U_{\alpha}^{(n_{m})}(p,t_{m})$,
$\alpha=1,2$, $m\geq1$, $n_{m}$ is the minimal number of iterations subject to the stopping test (\ref{bp65}),    $u_{\alpha}^{*}(x,y,t)$, $\alpha=1,2$, are the exact solutions to (\ref{bp1}), and   $E_{\alpha}(p,t_{m})$,
$\alpha=1,2$, $m\geq1$,
are the truncation errors of the exact solutions  $u_{\alpha}^{*}(x,y,t)$, $\alpha=1,2$,
 on the nonlinear
difference scheme (\ref{bp9}).
\end{theorem}
\begin{proof}
We denote $V(p,t_{m})=u^{*}(p,t_{m})- U^{*}(p,t_{m})
$, where the mesh vector function  $U^{*}(p,t_{m})$ is the unique solution of the nonlinear difference scheme (\ref{bp9}).
From (\ref{bp9}), by using the mean-value theorem, we obtain that
\begin{eqnarray*}
&&\left(\mathcal{L}_{\alpha}^{h}(p,t_{m})+\left(\tau^{-1}+\frac{\partial f_{\alpha}(p,t_{m},Y)}{\partial u_{\alpha}}\right)I\right)V_{\alpha}(p,t_{m})
-\tau^{-1} V_{\alpha}(p,t_{m-1}) \\
&&+\frac{\partial f_{\alpha}(p,t_{m},Y)}{\partial u_{\alpha^{\prime}}}V_{\alpha^{\prime}}(p,t_{m})=E_{\alpha}(p,t_{m}),\quad p\in \Omega^{h},\quad \alpha^{\prime}\neq \alpha,  \nonumber\\
&&\alpha, \alpha^{\prime}=1,2,\quad  V(p,t_{m})=0,\quad p\in \partial \Omega^{h},\quad V(p,0)=0,\quad p\in\overline{\Omega}^{h},\\
&& m\geq1,\nonumber
\end{eqnarray*}
where $Y_{\alpha}(p,t_{m})$, $\alpha=1,2$ lie between $u^{*}_{\alpha}(p,t_{m}) $ and $U^{*}_{\alpha}(p,t_{m})$, $\alpha=1,2$. From here, (\ref{bp61a}), (\ref{bp65a}), by using (\ref{bp13}), it follows that
\begin{eqnarray*}
&&\|V_{\alpha}(\cdot, t_{m})\|_{\overline{\Omega}^{h}}\leq \\
&&\frac{1}{\tau^{-1}+q}\left(q \|V_{\alpha^{\prime}}(\cdot,t_{m})\|_{\Omega^{h}}
+ \tau^{-1}\|V_{\alpha}(\cdot,t_{m-1})\|_{\Omega^{h}}+\|E_{\alpha}(\cdot,t_{m})\|_{\Omega^{h}} \right).
\end{eqnarray*}
Letting $V_{m}=\max_{\alpha=1,2}\|V_{\alpha}(\cdot, t_{m})\|_{\overline{\Omega}^{h}}$, $m\geq1$, we have
\begin{equation*}
V_{m}\leq
\frac{1}{\tau^{-1}+q}\left(q V_{m}
+ \tau^{-1}V_{m-1}+E_{m} \right).
\end{equation*}
Thus, taking into account that
\begin{equation*}
  1-\frac{q}{\tau^{-1}+q}>0,
\end{equation*}
we conclude
\begin{equation}\label{bp70}
  V_{m}\leq V_{m-1}+ \tau E_{m} .
\end{equation}
Since $V_{0}=0$, for $m=1$ in (\ref{bp70}), we have
\begin{equation*}
  V_{1}\leq  \tau E_{1} .
\end{equation*}
For $m=2$ in (\ref{bp70}), we obtain
\begin{equation*}
  V_{2}\leq  \tau ( E_{1} + E_{2}).
\end{equation*}
By induction on $m$, we can prove that
\begin{equation*}
  V_{m}\leq  \tau \sum_{\rho=1}^{m}E_{\rho}=\left(\sum_{\rho=1}^{m}\tau\right) \max_{\rho\geq1}E_{\rho}.
\end{equation*}
Since $\sum_{\rho=1}^{m}\tau\leq T$, where $T$ is the final time, we have
\begin{equation}\label{bp71}
  V_{m}\leq  T \max_{\rho\geq1}E_{\rho}.
\end{equation}
We estimate the left hand side in (\ref{bp69}) as follows
\begin{eqnarray*}
 \|U_{\alpha}^{( n_{m})}(\cdot,t_{m})\pm U_{\alpha}^{*}(\cdot, t_{m}) -u_{\alpha}^{*}( \cdot, t_{m})\|_{\overline{\Omega}^{h}} &\leq& \|U_{\alpha}^{(n_{m})}(\cdot, t_{m})-U_{\alpha}^{*}(\cdot, t_{m})\|_{\overline{\Omega}^{h}} \\
&&+\|U_{\alpha}^{*}(\cdot, t_{m})-u_{\alpha}^{*}(\cdot, t_{m})\|_{\overline{\Omega}^{h}},
\end{eqnarray*}
where $U_{\alpha}^{*}(p,t_{m})$, $\alpha=1,2$, are the exact solutions of (\ref{bp9}).
From here, (\ref{bp67}) and  (\ref{bp71}), we prove (\ref{bp69}).
\end{proof}
\begin{remark}\label{bp72}
  The truncation errors $E_{\alpha}(p,t_{m})$, $\alpha=1,2$, $m\geq1$, for the  nonlinear difference scheme (\ref{bp9})  are given in the form
  \begin{equation*}
\max_{m\geq1} E_{m}= O(\tau+h^{\kappa}),
\end{equation*}
where $E_{m}$ is defined in (\ref{bp69}),  $\tau$ and $h$ are, respectively, the time and space steps, $\kappa=1$ in the case
 of one-sided difference approximations of $u_{\alpha,x}$, $u_{\alpha,y}$, $\alpha=1,2$, and
$\kappa=2$ in the case of central difference approximations of these derivatives.
  \end{remark}
\subsection{Construction of upper and lower solutions}
To start  the monotone iterative methods (\ref{bp27}), on each time level $t_{m}$, $m\geq1$,  initial iterations are needed.
In this section, we discuss the construction of initial iterations   $\widetilde{U}_{\alpha}(p, t_{m})$ and $\widehat{U}_{\alpha}(p, t_{m})$, $\alpha=1,2$.
\subsubsection{Bounded $f_{u}$}
Assume that  the functions $f_{\alpha}$, $g_{\alpha}$ and $\psi_{\alpha}$, $\alpha=1,2$, in (\ref{bp1}) satisfy the conditions
\begin{align}\label{bp73}
&f_{\alpha}(x,y,t,\mathbf{0})\leq 0,\quad  f_{\alpha}(x,y,t,u)\geq -M_{\alpha},\quad u_{\alpha}(x,y,t)\geq0,    \\
&(x,y,t)\in \overline{Q}_{T},\quad g_{\alpha}(x,y,t)\geq0,\quad (x,y,t) \partial Q_{T},\quad \psi_{\alpha}(x,y)\geq0,\nonumber \\
&(x,y)\in \overline{\omega},\quad\alpha=1,2, \nonumber
\end{align}
where $ M_{\alpha}$, $\alpha=1,2$, are positive  constants. 
 From (\ref{bp19b}) and (\ref{bp73}), we obtain that the functions
 \begin{equation}\label{bp73a}
 \widehat{U}_{\alpha}(p,t_{m})=\left\{ \begin{array}{ll}
\psi_{\alpha}(p), \quad m=0,\\
0,\quad\quad\quad m\geq1,
\end{array}\right.
 p\in\overline{\Omega}^{h},\quad \alpha=1,2,
 \end{equation}
 are lower solutions of (\ref{bp9}).

 We introduce the linear problems
\begin{align} \label{bp74}
&\left(\mathcal{L}_{\alpha}^{h}(p,t_{m})+\tau^{-1}\right)\widetilde{U}_{\alpha}(p,t_{m})  =\tau^{-1}\widetilde{U}_{\alpha}(p,t_{m-1})+M_{\alpha},\quad p\in \Omega^{h}, \nonumber\\
& \widetilde{U}_{\alpha}(p,t_{m})=g_{\alpha}(p,t_{m}),\quad p\in \partial \Omega^{h},\quad \widetilde{U}_{\alpha}(p,0)= \psi_{\alpha}(p),\quad p\in \overline{\Omega}^{h}, \\
& \alpha=1,2,\quad m\geq1. \nonumber
\end{align}
\begin{theorem}\label{bp75}
  Let assumptions in (\ref{bp73}) be satisfied. Then $\widehat{U}$ and $\widetilde{U}$ from, respectively, (\ref{bp73a}) and (\ref{bp74})
  , are ordered lower and upper
  solutions to (\ref{bp9}), such that
  \begin{equation}\label{bp76}
    0\leq \widehat{U}_{\alpha}(p,t_{m})\leq \widetilde{U}_{\alpha}(p,t_{m}),\quad p\in \overline{\Omega}^{h},\quad \alpha=1,2,\quad m\geq1.
  \end{equation}
\end{theorem}
\begin{proof}
  From (\ref{bp73}) and (\ref{bp74}), by the maximum principle  (\ref{bp13}), we conclude  (\ref{bp76}) for $m=1$.
  By induction on $m$, we can prove (\ref{bp76}) for $m\geq1$.

We now show that $\widetilde{U}_{\alpha}(p,t_{m})$, $\alpha=1,2$, are   upper solutions
    (\ref{bp19}) to (\ref{bp9}). We present the left hand side of (\ref{bp9}) in the form
     \begin{align}\label{bp76a}
  &  \mathcal{G}_{\alpha}\left(\widetilde{U}_{\alpha}(p,t_{m}), \widetilde{U}_{\alpha}(p,t_{m-1}), \widetilde{U}_{\alpha^{\prime}}(p,t_{m})\right)=\\
  &\left(\mathcal{L}_{\alpha}^{h}(p,t_{m})+\tau^{-1}\right)\widetilde{U}_{\alpha}(p,t_{m})+f_{\alpha}(p,t_{m},\widetilde{U})
  -\tau^{-1}\widetilde{U}_{\alpha}(p,t_{m-1}), \nonumber\\
 & p\in \Omega^{h},\quad \alpha^{\prime}\neq \alpha,\quad \alpha ,\alpha^{\prime}=1,2, \quad m\geq1. \nonumber
  \end{align}
Using  (\ref{bp74}), for $m\geq1$, we obtain that
   \begin{eqnarray*}
   &&   \mathcal{G}_{\alpha}\left(\widetilde{U}_{\alpha}(p,t_{m}), \widetilde{U}_{\alpha}(p,t_{m-1}), \widetilde{U}_{\alpha^{\prime}}(p,t_{m})\right)= M_{\alpha}+f_{\alpha}(p,t_{m},\widetilde{U}),\quad p\in \Omega^{h},\\
   &&  \alpha^{\prime}\neq \alpha,\quad \alpha ,\alpha^{\prime}=1,2, .
   \end{eqnarray*}
   From here and (\ref{bp73}), we conclude that
    \begin{eqnarray*}
   &&   \mathcal{G}_{\alpha}\left(\widetilde{U}_{\alpha}(p,t_{m}), \widetilde{U}_{\alpha}(p,t_{m-1}), \widetilde{U}_{\alpha^{\prime}}(p,t_{m})\right)\geq 0,\quad p\in \Omega^{h},\quad  \alpha^{\prime}\neq \alpha,\\
   && \alpha ,\alpha^{\prime}=1,2, \quad m\geq1.
   \end{eqnarray*}
 Since $\widetilde{U}_{\alpha}(p,t_{m})$, $\alpha=1,2$, satisfy the boundary and initial conditions, we prove that $\widetilde{U}_{\alpha}(p,t_{m})$, $\alpha=1,2$, are upper solutions to (\ref{bp9}). From here and (\ref{bp76}), we conclude that $\widehat{U}$ and $\widetilde{U}$ from, respectively, (\ref{bp73a}) and (\ref{bp74}), are ordered lower and upper solutions to (\ref{bp9}).
 \end{proof}
\subsubsection{Constant upper and lower solutions}
Let  the functions $f_{\alpha}$, $g_{\alpha}$ and $\psi_{\alpha}$, $\alpha=1,2$, in (\ref{bp1}) satisfy the conditions
\begin{align}\label{bp77}
&f_{\alpha}(x,y,t,\mathbf{0})\leq 0,\ \ f_{\alpha}(x,y,t,K)\geq0,\ \  u_{\alpha}(x,y,t)\geq0,\ \ (x,y,t)\in \overline{Q}_{T}, \nonumber\\
& 0\leq g_{\alpha}(x,y,t)\leq K_{\alpha},\ \ (x,y,t)\in \partial Q_{T},\ \ 0\leq\psi_{\alpha}(x,y)\leq K_{\alpha}, \\
&(x,y)\in\overline{\omega},\ \ \alpha=1,2, \nonumber
\end{align}
where $K_{1}$, $K_{2}$ are positive
constants, and
 $K=(K_{1},K_{2})$.  The mesh functions $\widehat{U}_{\alpha}(p,t_{m})$, $\alpha=1,2$, from (\ref{bp73a})
 are lower solutions to (\ref{bp9}).

In the following lemma, we prove that the mesh functions
\begin{equation}\label{bp79}
\widetilde{U}_{\alpha}(p,t_{m})= \left\{ \begin{array}{ll}
\psi_{\alpha}(p),\quad   m=0,\\
K_{\alpha},\quad  m\geq1,
\end{array}\right.
p\in \overline{\Omega}^{h},\quad \alpha=1,2,
\end{equation}
are  upper solutions to (\ref{bp9}).
\begin{theorem}\label{bp80}
Suppose that the assumptions in  (\ref{bp77}) are satisfied. Then the mesh functions $\widehat{U}_{\alpha}(p,t_{m})$ and $\widetilde{U}_{\alpha}(p,t_{m})$  from,
respectively, (\ref{bp73a}) and (\ref{bp79}), are ordered lower and upper solutions to (\ref{bp9}) and satisfy (\ref{bp76}).
\end{theorem}
\begin{proof}
It is clear  from (\ref{bp73a}) and (\ref{bp79}), that $0\leq\widehat{U}_{\alpha}(p,t_{m})\leq \widetilde{U}_{\alpha}(p,t_{m})$,
$p\in\overline{\Omega}^{h}$, $\alpha=1,2$, $m\geq1$. We now show that $\widetilde{U}_{\alpha}(p,t_{m})$, $\alpha=1,2$, are upper solutions (\ref{bp19}) to (\ref{bp9}).

Using (\ref{bp79}), we write  the left hand side of (\ref{bp9}) for $m=1$  in the form
     \begin{align*}
  &  \mathcal{G}_{\alpha}\left(\widetilde{U}_{\alpha}(p,t_{1}), \psi_{\alpha}(p), \widetilde{U}_{\alpha^{\prime}}(p,t_{1})\right)=\mathcal{L}_{\alpha}^{h}(p,t_{1})K_{\alpha}+f_{\alpha}(p,t_{1},K) \\
 &+\tau^{-1}(K_{\alpha}-\psi_{\alpha}(p)),\quad  p\in \Omega^{h},\quad  \alpha^{\prime}\neq \alpha,\quad \alpha ,\alpha^{\prime}=1,2.
  \end{align*}
From here and  (\ref{bp77}) , we conclude that
\begin{equation*}
 \mathcal{G}_{\alpha}\left(\widetilde{U}_{\alpha}(p,t_{1}), \psi_{\alpha}(p), \widetilde{U}_{\alpha^{\prime}}(p,t_{1})\right)\geq0,\quad p\in \Omega^{h},\quad  \alpha^{\prime}\neq \alpha,\quad \alpha ,\alpha^{\prime}=1,2.
\end{equation*}
For $m\geq2$, from (\ref{bp77}) and (\ref{bp79}),  we have
\begin{eqnarray*}
&& \mathcal{G}_{\alpha}\left(\widetilde{U}_{\alpha}(p,t_{m}), \widetilde{U}_{\alpha}(p,t_{m-1}), \widetilde{U}_{\alpha^{\prime}}(p,t_{m})\right)\geq f(p,t_{m},K_{\alpha})\geq0,\quad p\in \Omega^{h},\\ &&  \alpha^{\prime}\neq \alpha,\quad \alpha ,\alpha^{\prime}=1,2, .
\end{eqnarray*}
Since $\widetilde{U}_{\alpha}(p,t_{0})$, $\alpha=1,2$, satisfy the  initial conditions
 and $\widetilde{U}_{\alpha}(p,t_{m})\geq g_{\alpha}(p,t_{m})$, $p\in\partial \Omega^{h}$, $\alpha=1,2$, at $m\geq1$, we prove that $\widetilde{U}_{\alpha}(p,t_{m})$, $\alpha=1,2$, are
upper solutions to (\ref{bp9}).  From here and (\ref{bp76}),
we conclude that $\widehat{U}$ and $\widetilde{U}$ from, respectively, (\ref{bp73a}) and (\ref{bp79}),
are ordered lower and upper solutions to (\ref{bp9}).
\end{proof}
\subsection{Applications}
\subsubsection{Gas-liquid interaction model}
 Consider the gas-liquid interaction model \cite{P92}, where a dissolved gas A and a dissolved reactant B
interact in a bounded diffusion medium $\omega$. The chemical reaction scheme is given by $A+k_{1}B\rightarrow k_{2}P$
and is called the second order reaction, where $k_{1}$ and $k_{2}$ are the reaction rates  and P is the product.
Denote by $z_{1}(x,y,t)$ and $z_{2}(x,y,t)$ the concentrations of the dissolved gas $A$ and the reactant $B$, respectively. Then the above reactant
scheme is governed by (\ref{bp1}) with $L_{\alpha}z_{\alpha}=\varepsilon_{\alpha}\triangle z_{\alpha}$, $f_{\alpha}=\sigma_{\alpha}z_{1}z_{2}$, $\alpha=1,2$, where
$\sigma_{1}$ is the reaction rate and
$\sigma_{2}=k_{1}\sigma_{1}$. By choosing a suitable positive constant $\varrho_{1}>0$ and letting $u_{1}=\varrho_{1}-z_{1}\geq0$, $u_{2}=z_{2}$, we have
\begin{equation}\label{bp81}
  f_{1}=-\sigma_{1}(\varrho_{1}-u_{1})u_{2},\quad f_{2}=\sigma_{2}(\varrho_{1}-u_{1})u_{2},
\end{equation}
and system (\ref{bp1}) is reduced to
\begin{eqnarray*}
&&u_{\alpha,t}-\varepsilon_{\alpha}\triangle u_{\alpha}+f_{\alpha}(u_{1},u_{2})=0,\quad (x,y,t)\in Q_{T},\quad \alpha=1,2, \\
&& u_{1}(x,y,t)=g_{1}^{*}(x,y,t)\geq0,\quad u_{2}(x,y,t)=g_{2}(x,y,t)\geq0,\\
&&(x,y,t)\in\partial Q_{T},\quad u_{\alpha}(x,y,0)=\psi_{\alpha}(x,y),\quad (x,y)\in\overline{\omega},\quad \alpha=1,2,
\end{eqnarray*}
where $g^{*}_{1}=\varrho_{1}-g_{1}\geq0$, $g_{2}\geq0$ on $\partial \omega$ and $\psi_{\alpha}\geq0$, $\alpha=1,2$,
in $\overline{\omega}^{h}$. It is clear from (\ref{bp81}) that $f_{\alpha}$, $\alpha=1,2$,
are quasi-monotone nondecreasing in the rectangle
\begin{equation*}
S_{\varrho}=[0,\varrho_{1}]\times [0,\varrho_{2}],
\end{equation*}
for any positive constant $\varrho_{2}$.

The nonlinear difference scheme (\ref{bp9}) is reduced to
\begin{eqnarray}\label{bp82}
 && (\mathcal{L}_{\alpha}^{h}(p,t_{m})+\tau^{-1}) U_{\alpha}(p,t_{m})+f_{\alpha}(U)-\tau^{-1}U_{\alpha}(p,t_{m-1})=0,\quad p \in \Omega^{h},\nonumber \\
&& \alpha=1,2,\quad  U_{1}(p,t_{m})=g^{*}_{1}(p,t_{m}),\quad  U_{2}(p,t_{m})=g_{2}(p,t_{m}),\quad p\in \partial \Omega^{h}, \nonumber \\
&& m\geq1,\quad U_{\alpha}(p,0)=\psi_{\alpha}(p),\quad p\in \overline{\Omega}^{h},
  \end{eqnarray}
where $f_{\alpha}$, $\alpha=1,2$, are defined in (\ref{bp81}). Since the reaction functions $f_{\alpha}$, $\alpha=1,2$,
 satisfy  the assumptions in  (\ref{bp77}),  with  $K_{\alpha}$, $\alpha=1,2$ are given by
 \begin{eqnarray}\label{bp82a}
  && K_{\alpha}=\varrho_{\alpha},\quad   \alpha=1,2,\\
  &&\varrho_{1}\geq\max_{m\geq1} \max_{p\in\partial \Omega^{h}} g^{*}_{1}(p,t_{m}),\quad    \varrho_{2}\geq\max_{m\geq1} \max_{p\in\partial \Omega^{h}} g_{2}(p,t_{m}),\quad m\geq1,\nonumber
 \end{eqnarray}
  it follows that the mesh functions $\widehat{U}_{\alpha}(p,t_{m})$ and $\widetilde{U}_{\alpha}(p,t_{m})$
 from, respectively, (\ref{bp73a}) and (\ref{bp79}) are ordered lower and upper solutions to (\ref{bp82}).

From (\ref{bp81}), in the sector $\langle \widehat{U}(t_{m}), \widetilde{U}(t_{m})\rangle=\langle 0, K_{\alpha}\rangle$, we have
\begin{eqnarray*}
&&\frac{\partial f_{1}}{\partial u_{1}}(U_{1},U_{2})=\sigma_{1}U_{2}(p,t_{m})\leq \sigma_{1}\varrho_{2},\quad p\in \overline{\Omega}^{h},\quad m\geq1,\\
&& \frac{\partial f_{2}}{\partial u_{2}}(U_{1},U_{2})=\sigma_{2}(\varrho_{1}-U_{1}(p,t_{m}))\leq \sigma_{2}\varrho_{1},\quad p\in \overline{\Omega}^{h},\quad m\geq1,\\
&&-\frac{\partial f_{1}}{\partial u_{2}}=\sigma_{1}(\varrho_{1}-U_{1}(p,t_{m}))\geq0,\quad p\in \overline{\Omega}^{h},\quad m\geq1,\\
&&-\frac{\partial f_{2}}{\partial u_{1}}=\sigma_{2}U_{2}(p,t_{m})\geq0,\quad p\in \overline{\Omega}^{h},\quad m\geq1,
\end{eqnarray*}
and the assumptions in (\ref{bp20}) and (\ref{bp21}) are satisfied  with
\begin{equation*}
c_{1}(p,t_{m})=\sigma_{1} \varrho_{2},\quad c_{2}(p,t_{m})= \sigma_{2}\varrho_{1},\quad p\in \overline{\Omega}^{h},\quad m\geq1.
\end{equation*}
From here and (\ref{bp82a}), we conclude that Theorem \ref{bp30}  holds  for the discrete gas-liquid interaction model (\ref{bp82}).
\subsubsection{The Volterra-Lotka competition model}
In the  Volterra-Lotka competition model \cite{P92} with the effect of dispersion between
two competing  species in an ecological systems,  the model is governed by (\ref{bp1}) with reaction
functions are given by
\begin{equation}\label{bp86}
f_{1}=-u_{1}(1-u_{1}+a_{1} u_{2}),\quad f_{2}=-u_{2}(1+a_{2}u_{1}-u_{2}),
\end{equation}
where $u_{1}$ and $u_{2}$ are the populations of two competing species, the parameters   $a_{\alpha}$, $\alpha=1,2$,
are positive constants which  describe the interaction of the two species. We assume that $a_{\alpha}$, $\alpha=1,2$, satisfy the inequality
\begin{equation}\label{bp86a}
  a_{1}<\frac{1}{a_{2}}.
\end{equation}
 System (\ref{bp1}) is reduced to
\begin{eqnarray*}
&&u_{\alpha,t}-\varepsilon_{\alpha}\triangle u_{\alpha}+f_{\alpha}(u_{1},u_{2})=0,\quad (x,y,t)\in Q_{T}, \\
&& u_{1}(x,y,t)=0,\quad u_{2}(x,y,t)=0,\quad (x,y,t)\in\partial Q_{T},\\
&&\quad u_{\alpha}(x,y,0)=\psi_{\alpha}(x,y),\quad (x,y)\in\overline{\omega},\quad \alpha=1,2.
\end{eqnarray*}
The nonlinear difference scheme (\ref{bp9}) is reduced to
\begin{eqnarray}\label{bp87}
 && (\mathcal{L}_{\alpha}^{h}(p,t_{m})+\tau^{-1}) U_{\alpha}(p,t_{m})+f_{\alpha}(U)-\tau^{-1}U_{\alpha}(p,t_{m-1})=0,\quad p \in \Omega^{h},\nonumber \\
&&  U_{1}(p,t_{m})=0,\quad U_{2}(p,t_{m})=0,\quad p\in \partial \Omega^{h},\quad m\geq1,  \nonumber \\
&& \quad U_{\alpha}(p,0)=\psi_{\alpha}(p),\quad p\in \overline{\Omega}^{h},\quad \alpha=1,2,
  \end{eqnarray}
where $f_{\alpha}$, $\alpha=1,2$, are defined in (\ref{bp86}).  We take $(M_{1},M_{2})$ and $(0,0)$ as
 ordered upper and lower solutions (\ref{bp19}) to (\ref{bp87}),
where $M_{\alpha}$, $\alpha=1,2$, are positive constants and chosen in the following forms
\begin{eqnarray}\label{bp87a}
&& M_{1}=a_{1}M_{2}+1,\\
&& M_{2}\geq \max\left\{\frac{a_{2}+1}{1-a_{1}a_{2}},   \max_{p\in\overline{\Omega}^{h}}\psi_{2}(p),  \frac{1}{a_{1}}\left( \max_{p\in\overline{\Omega}^{h}}\psi_{1}(p) -1\right)  \right\}. \nonumber
\end{eqnarray}
It is clear that $(M_{1}, M_{2})$ and $(0,0)$ satisfy (\ref{bp19a}) and (\ref{bp19c}). Now we prove (\ref{bp19b}). From  (\ref{bp76a}), it follows that
$M_{\alpha}$, $\alpha=1,2$, must satisfy the inequalities
\begin{eqnarray*}
&& \mathcal{G}_{1}\left(\widetilde{U}_{1}(p,t_{m}), \widetilde{U}_{1}(p,t_{m-1}), \widetilde{U}_{2}(p,t_{m})\right)=M_{1}(M_{1}-a_{1} M_{2}-1)\geq0,\nonumber\\
&& \mathcal{G}_{2}\left(\widetilde{U}_{1}(p,t_{m}), \widetilde{U}_{2}(p,t_{m-1}), \widetilde{U}_{2}(p,t_{m})\right)=M_{2}(M_{2}-a_{2}M_{1}-1)\geq0,\\
&&p\in \Omega^{h},\quad  m\geq1.
\end{eqnarray*}
From here, we conclude that $M_{\alpha}$, $\alpha=1,2$, must satisfy the inequalities
\begin{equation}\label{bp88a}
  a_{1}M_{2}+1\leq M_{1}\leq \frac{1}{a_{2}}(M_{2}-1).
\end{equation}

By using  (\ref{bp87a}), it is clear that the inequalities in   (\ref{bp88a}) are satisfied. Thus, we prove (\ref{bp19}).

In the sector $\langle \widehat{U}(t_{m}), \widetilde{U}(t_{m})\rangle=\langle \mathbf{0}, M \rangle$, $M=(M_{1},M_{2})$,  we have
\begin{eqnarray*}
&&\frac{\partial f_{1}}{\partial u_{1}}(U_{1},U_{2})=2 U_{1}(p,t_{m})-a_{1}U_{2}(p,t_{m})-1\leq 2 M_{1},\quad p\in \overline{\Omega}^{h},\\
&& \frac{\partial f_{2}}{\partial u_{2}}(U_{1},U_{2})=2 U_{2}(p,t_{m})-a_{2}U_{1}(p,t_{m})-1\leq 2 M_{2},\quad p\in \overline{\Omega}^{h},\\
&&-\frac{\partial f_{1}}{\partial u_{2}}=a_{1}U_{1}(p,t_{m})\geq0,\quad p\in \overline{\Omega}^{h},\\
&&-\frac{\partial f_{2}}{\partial u_{1}}=a_{2}U_{2}(p,t_{m})\geq0,\quad p\in \overline{\Omega}^{h},\quad m\geq1.
\end{eqnarray*}
From here,  the assumptions in (\ref{bp20}) and (\ref{bp21}) are satisfied  with
\begin{equation*}
c_{1}=2 M_{1},\quad c_{2}= 2 M_{2},
\end{equation*}
and we conclude that Theorem \ref{bp30}  holds  for the Volterra-Lotka competition model (\ref{bp87})
with $(\widetilde{U}_{1}, \widetilde{U}_{2})=(M_{1},M_{2})$ and $(\widehat{U}_{1}, \widehat{U}_{2})=(0,0)$.
\section{Comparison of the block monotone Jacobi and block monotone Gauss--Seidel methods}
The following theorem shows that the block monotone Gauss--Seidel  method (\ref{bp27}), $(\eta=1)$, converge not slower than the block monotone Jacobi method
(\ref{bp27}), $(\eta=0)$.
\begin{theorem}\label{bp89}
Let  $f(p,t_{m},U)$ in (\ref{bp9})  satisfy  (\ref{bp20}) and (\ref{bp21}),
 where    $\widetilde{U}(p,t_{m})=(\widetilde{U}_{1}(p,t_{m}),\widetilde{U}_{2}(p,t_{m}))$ and
 $\widehat{U}(p,t_{m})=(\widehat{U}_{1}(p,t_{m}),\widehat{U}_{2}(p,t_{m}))$
 are ordered upper and lower solutions (\ref{bp19}) of (\ref{bp9}).
 Suppose that $\{(\overline{U}^{(n)}_{\alpha,i,m})_{J},(\underline{U}^{(n)}_{\alpha,i,m})_{J}\}$ and $\{(\overline{U}^{(n)}_{\alpha,i,m})_{GS},(\underline{U}^{(n)}_{\alpha,i,m})_{GS}\}$, $i=0,1,\ldots,N_{x}$, $\alpha=1,2$, $m\geq1$,
 are, respectively, the sequences generated by the block monotone Jacobi method (\ref{bp27}), $(\eta=0)$
 and the block monotone Gauss--Seidel method (\ref{bp27}), $(\eta=1)$,
 where $(\overline{U}^{(0)})_{J}=(\overline{U}^{(0)})_{GS}=\widetilde{U}$ and
  $(\underline{U}^{(0)})_{J}=(\underline{U}^{(0)})_{GS}=\widehat{U}$, then
  \begin{eqnarray}\label{bp90}
&&(\underline{U}^{(n)}_{\alpha,i,m})_{J}\leq  (\underline{U}^{(n)}_{\alpha,i,m})_{GS}\leq (\overline{U}^{(n)}_{\alpha,i,m})_{GS} \leq (\overline{U}^{(n)}_{\alpha,i,m})_{J},\quad i=0,1,\ldots,N_{x},\nonumber \\
&&  \alpha=1,2,\quad m\geq1.
  \end{eqnarray}
\end{theorem}
\begin{proof}
   From (\ref{bp27}), we have
 \begin{eqnarray*}
 &&  A_{\alpha,i,m}(U_{\alpha,i,m}^{(n)})_{J}+c_{\alpha,m}(U^{(n)}_{\alpha,i,m})_{J} =c_{\alpha,m}(U^{(n-1)}_{\alpha,i,m})_{J}+ L_{\alpha,i,m}(U_{\alpha,i-1,m}^{(n-1)})_{J} \\
    &&+R_{\alpha,i,m}(U_{\alpha,i+1,m}^{(n-1)})_{J} - F_{\alpha,i,m}(U_{i,m}^{(n-1)})_{J}+
    \tau^{-1} (U_{\alpha,i,m-1})_{J}-G^{*}_{\alpha,i,m}, \nonumber \\
    && i=1,2,\ldots,N_{x}-1,\quad (U^{(n)}_{\alpha,i,m})_{J}=g_{\alpha,i,m},\quad i=0,N_{x},\quad m\geq1, \\
    && (U^{(n)}_{\alpha,i,0})_{J}=\psi_{\alpha,i},\quad i=0,1,\ldots,N_{x}.
 \end{eqnarray*}
   \begin{eqnarray*}
 &&  A_{\alpha,i,m}(U_{\alpha,i,m}^{(n)})_{GS}+c_{\alpha,m}(U^{(n)}_{\alpha,i,m})_{GS} =c_{\alpha,m}(U^{(n-1)}_{\alpha,i,m})_{GS}\\
    &&+ L_{\alpha,i,m}(U_{\alpha,i-1,m}^{(n)})_{GS} +R_{\alpha,i,m}(U_{\alpha,i+1,m}^{(n-1)})_{GS} - F_{\alpha,i,m}(U_{i,m}^{(n-1)})_{GS}\nonumber \\
    &&+ \tau^{-1} (U_{\alpha,i,m-1})_{GS} -G^{*}_{\alpha,i,m},\quad i=1,2,\ldots,N_{x}-1,\\
    &&(U^{(n)}_{\alpha,i,m})_{GS}=g_{\alpha,i,m},\quad i=0,N_{x},\quad m\geq1,\\
    && (U^{(n)}_{\alpha,i,0})_{GS}=\psi_{\alpha,i},\quad i=0,1,\ldots,N_{x}.
 \end{eqnarray*}
 From here, letting $\underline{W}^{(n)}_{\alpha,i,m}=\left(\underline{U}^{(n)}_{\alpha,i,m}\right)_{GS}-\left(\underline{U}^{(n)}_{\alpha,i,m}\right)_{J}$, $i=0,1,\ldots,N_{x}$, $\alpha=1,2$, $m\geq1$, we have
 \begin{eqnarray}\label{bp91}
   &&   A_{\alpha,i,m}\underline{W}^{(n)}_{\alpha,i,m}+c_{\alpha,m}\underline{W}^{(n)}_{\alpha,i,m} = c_{\alpha,m}\underline{W}^{(n-1)}_{\alpha,i,m} \\ 
       &&+L_{\alpha,i,m}\left((\underline{U}^{(n)}_{\alpha,i-1,m})_{GS}-(\underline{U}^{(n-1)}_{\alpha,i-1,m})_{J}\right)+R_{\alpha,i,m}\underline{W}^{(n-1)}_{\alpha,i+1,m} \nonumber\\
       &&-F_{\alpha,i,m}\left((\underline{U}^{(n-1)}_{i,m})_{GS}\right)+ F_{\alpha,i,m}\left((\underline{U}^{(n-1)}_{i,m})_{J}\right)\nonumber \\
       && +\tau^{-1} \left((U_{\alpha,i,m-1})_{GS}- (U_{\alpha,i,m-1})_{J}\right),\quad i=1,2,\ldots,N_{x}-1, \nonumber \\
    &&W^{(n)}_{\alpha,i,m}=0,\quad i=0,N_{x},\quad m\geq1,\quad W^{(n)}_{\alpha,i,0}=0,\quad i=0,1,\ldots,N_{x}. \nonumber
        \end{eqnarray}
By using  Theorem \ref{bp30}, we have $\left(\underline{U}^{(n)}_{\alpha,i,m}\right)_{GS}\geq \left(\underline{U}^{(n-1)}_{\alpha,i,m}\right)_{GS}$, $i=0,1,\ldots,N_{x}$, $\alpha=1,2$, $m\geq1$. From here and  (\ref{bp91}), we conclude that
 \begin{eqnarray}\label{bp92}
   &&   A_{\alpha,i,m}\underline{W}^{(n)}_{\alpha,i,m}+c_{\alpha,m}\underline{W}^{(n)}_{\alpha,i,m} \geq c_{\alpha,m}\underline{W}^{(n-1)}_{\alpha,i,m}
   +L_{\alpha,i,m}\underline{W}^{(n-1)}_{\alpha,i,m} \\
   &&+R_{\alpha,i,m}\underline{W}^{(n-1)}_{\alpha,i+1,m}-F_{\alpha,i,m}\left((\underline{U}^{(n-1)}_{i,m})_{GS}\right)+ F_{\alpha,i,m}\left((\underline{U}^{(n-1)}_{i,m})_{J}\right)\nonumber \\
       && +\tau^{-1} \left((U_{\alpha,i,m-1})_{GS}- (U_{\alpha,i,m-1})_{J}\right), \nonumber \\
 &&W^{(n)}_{\alpha,i,m}=0,\quad i=0,N_{x},\quad m\geq1,\quad W^{(n)}_{\alpha,i,0}=0,\quad i=0,1,\ldots,N_{x}. \nonumber
        \end{eqnarray}
Taking into account that $(A_{\alpha,i,m}+c_{\alpha,m}I)^{-1}\geq \emph{O}$,
$L_{\alpha,i,m}\geq\emph{O}$, $R_{\alpha,i,m}\geq\emph{O}$, $i=1,2,\ldots,N_{x}-1$,
$\alpha=1,2$, $m\geq1$, for $n=1$ in (\ref{bp92}), on the first time level $m=1$, in view
of $(\underline{U}^{(0)}_{\alpha,i,m})_{GS}=(\underline{U}^{(0)}_{\alpha,i,m})_{J}$
and $\underline{W}^{(0)}_{\alpha,i,m}=\mathbf{0}$,  we conclude that
\begin{equation*}
  \underline{W}^{(1)}_{\alpha,i,1}\geq\mathbf{0},\quad i=0,1,\ldots,N_{x},\quad \alpha=1,2.
\end{equation*}
For $n=2$ in (\ref{bp92}) and using  notation (\ref{bp22}), we obtain
 \begin{eqnarray*}
   &&   \left(A_{\alpha,i,1}+c_{\alpha,1}\right)\underline{W}^{(2)}_{\alpha,i,1} \geq
   L_{\alpha,i,1}\underline{W}^{(1)}_{\alpha,i,1}+R_{\alpha,i,1}\underline{W}^{(1)}_{\alpha,i+1,1} \\
   &&+\Gamma_{\alpha,i,1}\left((\underline{U}^{(1)}_{i,1})_{GS}\right)-
   \Gamma_{\alpha,i,1}\left((\underline{U}^{(1)}_{i,1})_{J}\right)
       , \nonumber \\
 &&W^{(2)}_{\alpha,i,1}=0,\quad i=0,N_{x},\quad W^{(2)}_{\alpha,i,0}=0,\quad i=0,1,\ldots,N_{x}. \nonumber
        \end{eqnarray*}
Taking into account that $(A_{\alpha,i,1}+c_{\alpha,1}I)^{-1}\geq \emph{O}$,
$L_{\alpha,i,1}\geq\emph{O}$, $R_{\alpha,i,1}\geq\emph{O}$, $i=1,2,\ldots,N_{x}-1$,
$\alpha=1,2$, and $ \underline{W}^{(1)}_{\alpha,i,1}\geq\mathbf{0}$, by using (\ref{bp24}), we have
\begin{equation*}
  \underline{W}^{(2)}_{\alpha,i,1}\geq\mathbf{0},\quad i=0,1,\ldots,N_{x},\quad \alpha=1,2.
\end{equation*}
By induction on $n$, we prove that
\begin{equation*}
  \underline{W}^{(n)}_{\alpha,i,1}\geq\mathbf{0},\quad i=0,1,\ldots,N_{x},\quad \alpha=1,2.
\end{equation*}
On the second time level $m=2$, taking into account that $(A_{\alpha,i,2}+c_{\alpha,2}I)^{-1}\geq \emph{O}$,
$L_{\alpha,i,2}\geq\emph{O}$, $R_{\alpha,i,2}\geq\emph{O}$, $i=1,2,\ldots,N_{x}-1$,
$\alpha=1,2$,  $  \underline{W}^{(0)}_{\alpha,i,2}=\mathbf{0}$ and $\underline{W}_{\alpha,i,1}\geq\mathbf{0}$,  from (\ref{bp92}), we have
\begin{equation*}
  \underline{W}^{(1)}_{\alpha,i,2}\geq\mathbf{0},\quad i=0,1,\ldots,N_{x},\quad \alpha=1,2.
\end{equation*}
For $n=2$ in (\ref{bp92}) and using  notation (\ref{bp22}), we obtain
 \begin{eqnarray*}
   &&   \left(A_{\alpha,i,2}+c_{\alpha,2}\right)\underline{W}^{(2)}_{\alpha,i,2} \geq
   L_{\alpha,i,2}\underline{W}^{(1)}_{\alpha,i,2}+R_{\alpha,i,2}\underline{W}^{(1)}_{\alpha,i+1,2} \\
   &&+\Gamma_{\alpha,i,2}\left((\underline{U}^{(1)}_{i,2})_{GS}\right)-
   \Gamma_{\alpha,i,2}\left((\underline{U}^{(1)}_{i,2})_{J}\right)
       , \nonumber \\
 &&W^{(2)}_{\alpha,i,2}=0,\quad i=0,N_{x},\quad W^{(2)}_{\alpha,i,0}=0,\quad i=0,1,\ldots,N_{x}. \nonumber
        \end{eqnarray*}
Taking into account that $(A_{\alpha,i,2}+c_{\alpha,2}I)^{-1}\geq \emph{O}$,
$L_{\alpha,i,2}\geq\emph{O}$, $R_{\alpha,i,2}\geq\emph{O}$, $i=1,2,\ldots,N_{x}-1$,
$\alpha=1,2$, and $ \underline{W}^{(1)}_{\alpha,i,2}\geq\mathbf{0}$, by using (\ref{bp24}), we have
\begin{equation*}
  \underline{W}^{(2)}_{\alpha,i,2}\geq\mathbf{0},\quad i=0,1,\ldots,N_{x},\quad \alpha=1,2.
\end{equation*}
By induction on $n$, we prove that
\begin{equation*}
  \underline{W}^{(n)}_{\alpha,i,2}\geq\mathbf{0},\quad i=0,1,\ldots,N_{x},\quad \alpha=1,2.
\end{equation*}
By induction on $m$, we prove that
\begin{equation*}
  \underline{W}^{(n)}_{\alpha,i,m}\geq\mathbf{0},\quad i=0,1,\ldots,N_{x},\quad \alpha=1,2,\quad m\geq1.
\end{equation*}
Thus, we prove (\ref{bp90}) for lower solutions. By following the same manner,
we can prove (\ref{bp90}) for upper solutions.
\end{proof}
\section{The case of quasi-monotone nonincreasing  reaction functions}
\subsection{The statement of the block nonlinear difference scheme}
We consider the same block nonlinear difference scheme discussed in section 4.1 which is given by (\ref{bp18}).
\subsection{Block monotone Jacobi and Gauss-Seidel methods}
We now  present the block monotone Jacobi and block monotone Gauss--Seidel  methods for the nonlinear
difference scheme (\ref{bp18})
   in the case of quasi-monotone nonincreasing reaction functions (\ref{bp42}).

For solving the nonlinear difference scheme (\ref{bp18}), on each time level $t_{m}$, $m\geq1$,
 we calculate either the  sequence $\{\overline{U}_{1,i,m}^{(n)}, \underline{U}_{2,i,m}^{(n)}\}$,
 or the sequence $\{\underline{U}_{1,i,m}^{(n)}, \overline{U}_{2,i,m}^{(n)}\}$, $i=0,1,\ldots,N_{x}$, $m\geq1$,
by the block  Jacobi and block  Gauss-Seidel methods. In the case
of $\{\overline{U}_{1,i,m}^{(n)}, \underline{U}_{2,i,m}^{(n)}\}$, we have
\begin{subequations}\label{bp94}
\begin{align}\label{bp94a}
&A_{1,i,m}\overline{Z}_{1,i,m}^{(n)}-\eta L_{1,i,m}\overline{Z}^{(n)}_{1,i-1,m}+c_{1,m}\overline{Z}_{1,i,m}^{(n)}=\\
&-\mathcal{G}_{1,i,m}\left(\overline{U}_{1,i,m}^{(n-1)},\overline{U}_{1, i,m-1}, \underline{U}_{2,i,m}^{(n-1)} \right ),\quad  i=1,2,\ldots,  N_{x}-1,\quad m\geq1,\nonumber\\
&A_{2,i,m}\underline{Z}_{2,i,m}^{(n)}-\eta L_{2,i,m}\underline{Z}^{(n)}_{2,i-1,m}+c_{2,m}\underline{Z}_{2,i,m}^{(n)}=\nonumber\\
&-\mathcal{G}_{2,i,m}\left(\underline{U}_{2,i,m}^{(n-1)} ,\underline{U}_{2, i,m-1},\overline{U}_{1,i,m}^{(n-1)} \right ),\quad  i=1,2,\ldots,  N_{x}-1,\quad m\geq1,\nonumber
\end{align}
and in the case of  $\{\underline{U}_{1,i,m}^{(n)}, \overline{U}_{2,i,m}^{(n)}\}$, we have
\begin{align}\label{bp95b}
&A_{1,i,m}\underline{Z}_{1,i,m}^{(n)}-\eta L_{1,i,m}\underline{Z}^{(n)}_{1,i-1,m}+c_{1,m}\underline{Z}_{1,i,m}^{(n)}=\\
&-\mathcal{G}_{1,i,m}\left(\underline{U}_{1,i,m}^{(n-1)} ,\underline{U}_{1, i,m-1}, \overline{U}_{2,i,m}^{(n-1)}\right ),\quad  i=1,2,\ldots,  N_{x}-1,\quad m\geq1,\nonumber\\
&A_{2,i,m}\overline{Z}_{2,i,m}^{(n)}-\eta L_{2,i,m}\overline{Z}^{(n)}_{2,i-1,m}+c_{2,m}\overline{Z}_{2,i,m}^{(n)}=\nonumber\\
&-\mathcal{G}_{2,i,m}\left(\overline{U}_{2,i,m}^{(n-1)},\overline{U}_{2, i,m-1}, \underline{U}_{1,i,m}^{(n-1)} \right ),\quad  i=1,2,\ldots,  N_{x}-1,\quad m\geq1,\nonumber
\end{align}
\begin{equation}\label{bp96c}
Z_{\alpha,i,m}^{(n)}=\left\{ \begin{array}{ll}
g_{\alpha,i,m}- U_{\alpha,i,m}^{(0)},\quad \quad n=1  ,\\
\mathbf{0},\quad\quad\quad\quad \quad \quad \quad \quad  n\geq2,
\end{array}\right.
\quad i=0,N_{x}, \quad \alpha=1,2,
 \end{equation}
 \begin{equation*}
U_{\alpha,i,m}=\psi_{\alpha,i},\quad i=0,1,\ldots,N_{x},\quad \alpha=1,2,
\end{equation*}
 \end{subequations}
 \begin{equation*}
   Z^{(n)}_{\alpha,i,m}=U^{(n)}_{\alpha,i,m}-U^{(n-1)}_{\alpha,i,m},\quad U_{\alpha,i,m}=U^{(n_{m})}_{\alpha,i,m},\quad m\geq1,
 \end{equation*}
where  $c_{\alpha,m}$,  $\alpha=1,2$,  $m\geq1$, are defined in   (\ref{bp26a}),
the residuals $\mathcal{G}_{\alpha,i,m}\Big(\overline{U}_{\alpha,i,m}^{(n-1)}, \\  \overline{U}_{\alpha,i,m-1},
 \underline{U}_{\alpha^{\prime},i,m}^{(n-1)}\Big)$,  $\mathcal{G}_{\alpha,i,m}\left(\underline{U}_{\alpha,i,m}^{(n-1)}, \underline{U}_{\alpha,i,m-1}, \overline{U}_{\alpha^{\prime},i,m}^{(n-1)}\right)$
 are defined in (\ref{bp18a}), $\mathbf{0}$ is zero column vector with $N_{x}-1$ components. The column vectors
   $U_{\alpha,i,m}$, $i=0,1,\ldots,N_{x}$, $\alpha=1,2$, are the approximate  solutions on time level
 $m\geq1$, where $n_{m}$ is a number of iterations on time level   $m\geq1$.
For  $\eta=0$ and $\eta=1$, we have, respectively,  the block Jacobi and block  Gauss--Seidel methods.
\begin{theorem}\label{bp97}
Let   $f(p,t_{m},U)$ in (\ref{bp9})  satisfy  (\ref{bp20}) and (\ref{bp42}),
 where    $\widetilde{U}(p,t_{m})=(\widetilde{U}_{1}(p,t_{m}),\widetilde{U}_{2}(p,t_{m}))$ and
 $\widehat{U}(p,t_{m})=(\widehat{U}_{1}(p,t_{m}),\widehat{U}_{2}(p,t_{m}))$
 are ordered upper and lower solutions (\ref{bp40}) of   (\ref{bp9}). Then the  sequences $\{\overline{U}_{1,i,m}^{(n)}, \underline{U}_{2,i,m}^{(n)}\}$ and
 $\{\underline{U}_{1,i,m}^{(n)}, \overline{U}_{2,i,m}^{(n-1)}\}$  generated by (\ref{bp94}), with $\overline{U}^{(0)}(p,t_{m})=\widetilde{U}(p,t_{m})$ and
    $\underline{U}^{(0)}(p,t_{m})=\widehat{U}(p,t_{m})$ are ordered upper and lower solutions and
 converge monotonically, such that,
\begin{equation}\label{bp98}
   \underline{U}_{\alpha,i,m}^{(n-1)}\leq \underline{U}_{\alpha,i,m}^{(n)}\leq  \overline{U}_{\alpha,i,m}^{(n)}\leq \overline{U}_{\alpha,i,m}^{(n-1)},\ i=0,1, \ldots, N_{x},\quad \alpha=1,2, \quad m\geq1.
  \end{equation}
\end{theorem}
\begin{proof}
 We consider the case of Gauss-Seidel method $\eta=1$, and the case of the  Jacobi method can be proved by a similar manner.
On first time level $m=1$,
since $\overline{U}^{(0)}$ and $\underline{U}^{(0)}$ are ordered  upper and lower  solution (\ref{bp40}) with respect to $U_{\alpha}(p,0)=\psi_{\alpha}(p)$, from (\ref{bp94a}) and (\ref{bp96c}), we have
\begin{eqnarray}\label{bp99}
&&(A_{1,i,1}+c_{1,1}I)\overline{Z}_{1,i,1}^{(1)}\leq  L_{1,i,1}\overline{Z}^{(1)}_{1,i-1,1},\quad i=1,2,\ldots,N_{x}-1,\nonumber\\
&& (A_{2,i,1}+c_{2,1}I)\underline{Z}_{2,i,1}^{(1)}\geq  L_{2,i,1}\underline{Z}^{(1)}_{2,i-1,1},\quad i=1,2,\ldots,N_{x}-1, \nonumber \\
&& \overline{Z}_{1,i,1}^{(1)}\leq\mathbf{0},\quad \underline{Z}_{2,i,1}^{(1)}\geq\mathbf{0},\quad i=0,N_{x},
\end{eqnarray}
where $I$ is the identity matrix. For $i=1$ in (\ref{bp99}), taking into account that $
L_{\alpha,i,1}\geq\emph{O}$, $i=1,2,\ldots,N_{x}-1$, and $\overline{Z}^{(1)}_{1,0,1}\leq\mathbf{0}$,
$\underline{Z}^{(1)}_{2,0,1}\geq\mathbf{0}$, we have   $\left( A_{1,1,1}+c_{1,1}I\right)\overline{Z}_{1,1,1}^{(1)}\leq\mathbf{0}$,
 $\left( A_{2,1,1}+c_{2,1}I\right)\underline{Z}_{2,1,1}^{(1)}\geq\mathbf{0}$. Taking into account that $d_{\alpha,ij}>0$, $b_{\alpha,ij}$, $t_{\alpha,ij}\geq0$, $\alpha=1,2$, in (\ref{bp17})
and $A_{\alpha,i,1}$ are strictly diagonal dominant matrix, we conclude that  $A_{\alpha,i,1}$, $i=1,2,\ldots,N_{x}-1$, $\alpha=1,2$,
are $M$-matrices  and $A_{\alpha,i,1}^{-1}\geq\emph{O}$ (Corollary 3.20, \cite{Varga2000}),   which leads to $(A_{\alpha,i,1}+c_{\alpha,1}I)^{-1}\geq\emph{O}$, where $\emph{O}$ is the ($N_{y}-1)\times (N_{y}-1$) null  matrix. From here, we obtain that
\begin{equation*}
 \overline{Z}_{1,1,1}^{(1)}\leq \mathbf{0},\quad  \underline{Z}_{2,1,1}^{(1)}\geq \mathbf{0}.
\end{equation*}
From here, for $i=2$ in (\ref{bp99}), in a similar manner, we conclude that
\begin{equation*}
 \overline{Z}_{1,2,1}^{(1)}\leq \mathbf{0},\quad  \underline{Z}_{2,2,1}^{(1)}\geq \mathbf{0}.
\end{equation*}
By induction on $i$, we can prove that
\begin{equation}\label{bp99a}
 \overline{Z}_{1,i,1}^{(1)}\leq \mathbf{0},\quad  \underline{Z}_{2,i,1}^{(1)}\geq \mathbf{0},\quad i=0,1,\ldots,N_{x}.
\end{equation}
From (\ref{bp95b}) and (\ref{bp96c}), by a similar manner,  we prove that
\begin{equation}\label{bp99b}
 \underline{Z}_{1,i,1}^{(1)}\geq \mathbf{0},\quad  \overline{Z}_{2,i,1}^{(1)}\leq \mathbf{0},\quad i=0,1,\ldots,N_{x}.
\end{equation}
We now prove that $\overline{U}^{(1)}_{\alpha,i,1}$ and  $\underline{U}^{(1)}_{\alpha,i,1}$, $i=0,1,\ldots,N_{x}$, $\alpha=1,2$, satisfy  (\ref{bp40a})
 with respect to the  column vector  $U_{\alpha,i,0}=\psi_{\alpha,i}$, $i=0,1,\ldots,N_{x}$. Let $W^{(1)}_{\alpha,i,1}=\overline{U}^{(1)}_{\alpha,i,1}-\underline{U}^{(1)}_{\alpha,i,1}$, $i=0,1,\ldots,N_{x}$, $\alpha=1,2$, from (\ref{bp94}), we have
 \begin{eqnarray*}
 && (A_{\alpha,i,1}+c_{\alpha,1}I)W_{\alpha,i,1}^{(1)}=L_{\alpha,i,1}W_{\alpha,i-1,1}^{(1)}+R_{\alpha,i,1}W_{\alpha,i+1,1}^{(0)}\\ && +c_{\alpha,1}\overline{U}_{\alpha,i,1}^{(0)}-F_{\alpha,i,1}(\overline{U}_{\alpha,i,1}^{(0)},\underline{U}_{\alpha^{\prime},i,1}^{(0)})\\
&&-\left[c_{\alpha,1}\underline{U}_{\alpha,i,1}^{(0)} - F_{\alpha,i,1}(\underline{U}_{\alpha,i,1}^{(0)},\overline{U}_{\alpha^{\prime},i,1}^{(0)})\right], \\
 && i=1,2, \ldots, N_{x}-1,\quad  W_{\alpha,i,1}^{(1)}=\mathbf{0},\quad i=0, N_{x},\\
 && W_{\alpha,i,0}=\mathbf{0},\quad i=0,1,\ldots,N_{x},\quad \alpha^{\prime}\neq\alpha,\quad \alpha, \alpha^{\prime}=1,2.
\end{eqnarray*}
Using  notation (\ref{bp22}) with  $(U_{1},U_{2})=(\overline{U}_{1,i,1}^{(0)},\underline{U}_{2,i,1}^{(0)})$ and
$(V_{1},V_{2})=(\underline{U}_{1,i,1}^{(0)}, \\ \overline{U}_{2,i,1}^{(0)})$, we present the above problem in the form
\begin{eqnarray*}
 && (A_{\alpha,i,1}+c_{\alpha,1}I)W_{\alpha,i,1}^{(1)}=L_{\alpha,i,1}W_{\alpha,i-1,1}^{(1)}+R_{\alpha,i,1}W_{\alpha,i+1,1}^{(0)}\\ && +\Gamma_{\alpha,i,1}(\overline{U}_{\alpha,i,1}^{(0)},\underline{U}_{\alpha^{\prime},i,1}^{(0)})- \Gamma_{\alpha,i,1}(\underline{U}_{\alpha,i,1}^{(0)},\overline{U}_{\alpha^{\prime},i,1}^{(0)}), \\
 && i=1,2, \ldots, N_{x}-1,\quad  W_{\alpha,i,1}^{(1)}=\mathbf{0},\quad i=0, N_{x},\\
 && W_{\alpha,i,0}=\mathbf{0},\quad i=0,1,\ldots,N_{x},\quad \alpha^{\prime}\neq\alpha,\quad \alpha, \alpha^{\prime}=1,2.
\end{eqnarray*}
From (\ref{bp44}), taking into account that $ R_{\alpha,i,1}\geq \emph{O} $,
$i=1,2,\ldots,N_{x}-1$,  and
$ W_{\alpha,i,1}^{(0)}\geq\mathbf{0}$, $i=0,1,\ldots,N_{x}$, $\alpha=1,2$,  we conclude that
\begin{eqnarray}\label{bp100}
 && (A_{\alpha,i,1}+c_{\alpha,1}I)W_{\alpha,i,1}^{(1)}\geq L_{\alpha,i,1}W_{\alpha,i-1,1}^{(1)},\quad  i=1,2, \ldots, N_{x}-1, \\
 && W_{\alpha,i,1}^{(1)}=\mathbf{0},\quad i=0, N_{x},\quad  W_{\alpha,i,0}=\mathbf{0},\quad i=0,1,\ldots,N_{x},\quad \alpha=1,2.\nonumber
\end{eqnarray}
For $i=1$ in (\ref{bp100}), taking into account that $L_{\alpha,i,1}\geq \emph{O}$, $i=1,2,\ldots,N_{x}-1$,
$W_{\alpha,0,1}^{(1)}=\mathbf{0}$, 
 and $(A_{\alpha,i,1}+c_{\alpha,1}I)^{-1}\geq\emph{O}$, $i=1,2,\ldots,N_{x}-1$, $\alpha=1,2$,
 (Corollary 3.20, \cite{Varga2000}), we have
 \begin{equation*}
   W_{\alpha,1,1}^{(1)}\geq\mathbf{0},\quad \alpha=1,2.
 \end{equation*}
 For $i=2$ in (\ref{bp100}), taking into account that $ W_{\alpha,1,1}^{(1)}\geq\mathbf{0}$, by a similar manner,  we obtain
  \begin{equation*}
   W_{\alpha,2,1}^{(1)}\geq\mathbf{0},\quad \alpha=1,2.
 \end{equation*}
 By induction on $i$, we can prove that
  \begin{equation*}
   W_{\alpha,i,1}^{(1)}\geq\mathbf{0},\quad i=0,1,\ldots,N_{x},\quad \alpha=1,2.
 \end{equation*}
 Now, by induction on $n$, we can prove that
 \begin{equation*}
  W_{\alpha,i,1}^{(n)}\geq\mathbf{0},\quad i=0,1,\ldots,N_{x},\quad \alpha=1,2.
 \end{equation*}
 Thus, we prove (\ref{bp40a}) on the first time level $m=1$. We now prove (\ref{bp40b}).
 From (\ref{bp94a}) and using (\ref{bp29}), we obtain
\begin{align}\label{bp101}
&\mathcal{G}_{1,i,1}\left(\overline{U}_{1,i,1}^{(1)}, \psi_{1,i}, \underline{U}_{2,i,1}^{(1)}\right)=-\left(c_{1,1}-\frac{\partial F_{1,i,1}(\overline{E}^{(1)}_{1,i,1},\underline{U}^{(1)}_{2,i,1})}{\partial u_{1}}\right) \overline{Z}_{1,i,1}^{(1)}\nonumber\\&+\frac{\partial F_{1,i,1}(\overline{U}^{(0)}_{1,i,1},\underline{E}^{(1)}_{2,i,1})}{\partial u_{2}}\underline{Z}_{2,i,1}^{(1)}
-R_{1,i,1}\overline{Z}_{1,i+1,1}^{(1)},\quad i=1,2, \ldots, N_{x}-1,
\end{align}
where
 \begin{equation*}
\overline{U}^{(1)}_{1,i,1}\leq\overline{E}^{(1)}_{1,i,1}\leq \overline{U}^{(0)}_{1,i,1},\quad \underline{U}^{(0)}_{2,i,1}\leq\underline{E}^{(1)}_{2,i,1}\leq \underline{U}^{(1)}_{2,i,1} , \quad i=0,1,\ldots,N_{x}.
\end{equation*}
From  (\ref{bp99a}), (\ref{bp99b}) and taking into account that $W^{(1)}_{\alpha,i,1}\geq0$, $i=0,1,\ldots,N_{x}$, $\alpha=1,2$, it follows
that the partial derivatives in (\ref{bp101}) satisfy (\ref{bp20}) and (\ref{bp42}). From (\ref{bp20}), (\ref{bp42}), (\ref{bp99a}), (\ref{bp99b}), (\ref{bp101})
and taking into account that $R_{1,i,1}\geq\emph{O}$, $ i=1,2, \ldots, N_{x}-1$, we conclude that
\begin{equation}\label{bp102}
  \mathcal{G}_{1,i,1}\left(\overline{U}_{1,i,1}^{(1)}, \psi_{1,i}, \underline{U}_{2,i,1}^{(1)}\right)\geq \mathbf{0},\quad i=1,2,\ldots,N_{x}.
\end{equation}
Similarly, we conclude that
 \begin{equation}\label{bp103}
  \mathcal{G}_{2,i,1}\left( \underline{U}_{2,i,1}^{(1)}, \psi_{1,i},\overline{U}_{1,i,1}^{(1)}\right)\leq\mathbf{0},\quad i=1,2,\ldots,N_{x}.
\end{equation}
By a similar argument, from (\ref{bp95b}), we prove that
\begin{eqnarray}\label{bp104}
 && \mathcal{G}_{1,i,1}\left(\underline{U}_{1,i,1}^{(1)}, \psi_{1,i}, \overline{U}_{2,i,1}^{(1)}\right)\leq\mathbf{0},\quad
 \mathcal{G}_{2,i,1}\left(\overline{U}_{2,i,1}^{(1)}, \psi_{1,i}, \underline{U}_{1,i,1}^{(1)} \right)\geq\mathbf{0},\\
 &&\quad i=1,2,\ldots,N_{x}-1.\nonumber
\end{eqnarray}
Thus, from (\ref{bp102})--(\ref{bp104}), it follows  (\ref{bp40b}) on the first time level $m=1$.
By induction on $n$, we can prove (\ref{bp98}) on the first time level $m=1$.

On the second time level $m=2$, from (\ref{bp94a}) and (\ref{bp98}), we have $\overline{U}_{1,i,1}\leq
\widetilde{U}_{1,i,1} $, $i=0,1,\ldots,N_{x}$. Thus, it follows that 
\begin{eqnarray*}
&&  \mathcal{G}_{1,i,2}\left(\widetilde{U}_{1,i,2} ,\overline{U}_{1,i,1}, \widehat{U}_{2,i,2}\right)\geq \mathcal{G}_{1,i,2}\left(\widetilde{U}_{1, i,2},\widetilde{U}_{1,i,1}, \widehat{U}_{2,i,2}\right )\geq\mathbf{0},\\
&& \mathcal{G}_{2,i,2}\left(\widehat{U}_{2,i,2} ,\underline{U}_{1,i,1}, \widetilde{U}_{1,i,2}\right)\leq \mathcal{G}_{2,i,2}\left(\widehat{U}_{2, i,2},\widehat{U}_{1,i,1}, \widetilde{U}_{1,i,2}\right )\leq\mathbf{0} \\
&&  i=1,2\ldots,N_{x},
\end{eqnarray*}
which means that  $\widetilde{U}_{1,i,2}$ and $\widehat{U}_{2,i,2}$,
  $i=0,1,\ldots,N_{x}$, are, respectively, upper and lower solutions  with respect to $\overline{U}_{1,i,1}$ and $\underline{U}_{1,i,1}$,
   $i=0,1,\ldots,N_{x}$, $\alpha=1,2$.

   Similarly, we can obtain that
\begin{eqnarray*}
 && \mathcal{G}_{1,i,2}\left(\widehat{U}_{1,i,2},\underline{U}_{1,i,1}, \widetilde{U}_{2,i,2}\right )\leq\mathbf{0},\quad
  \mathcal{G}_{2,i,2}\left(\widetilde{U}_{2,i,2},\overline{U}_{2,i,1}, \widehat{U}_{1,i,2}\right )\geq\mathbf{0}, \\
 &&  i=1,2\ldots,N_{x}-1,
\end{eqnarray*}
which means that  $\widehat{U}_{1,i,2}$ and $\widetilde{U}_{2,i,2}$,
  $i=0,1,\ldots,N_{x}$, are, respectively, lower and upper solutions  with respect to $\underline{U}_{1,i,1}$ and $\overline{U}_{2,i,1}$,
   $i=0,1,\ldots,N_{x}$.

From  (\ref{bp94a}) and  (\ref{bp96c}), we have
\begin{eqnarray}\label{bp105}
&&(A_{1,i,2}+c_{1,2}I)\overline{Z}_{1,i,2}^{(1)}\leq  L_{1,i,2}\overline{Z}^{(1)}_{1,i-1,2},\quad i=1,2,\ldots,N_{x}-1,\nonumber\\
&& (A_{2,i,2}+c_{2,2}I)\underline{Z}_{2,i,2}^{(1)}\geq  L_{2,i,2}\underline{Z}^{(1)}_{2,i-1,2},\quad i=1,2,\ldots,N_{x}-1, \nonumber \\
&& \overline{Z}_{1,i,2}^{(1)}\leq\mathbf{0},\quad \underline{Z}_{2,i,2}^{(1)}\geq\mathbf{0},\quad i=0,N_{x},
\end{eqnarray}
where $I$ is the identity matrix. For $i=1$ in (\ref{bp105}), taking into account that $
L_{\alpha,i,2}\geq\emph{O}$, $i=1,\ldots,N_{x}-1$, and $\overline{Z}^{(1)}_{1,0,2}\leq\mathbf{0}$,
$\underline{Z}^{(1)}_{2,0,2}\geq\mathbf{0}$, we have   $\left( A_{1,1,2}+c_{1,2}I\right)\overline{Z}_{1,1,2}^{(1)}\leq\mathbf{0}$,
 $\left( A_{2,1,2}+c_{2,2}I\right)\underline{Z}_{2,1,2}^{(1)}\geq\mathbf{0}$. Taking into account that $d_{\alpha,ij}>0$, $b_{\alpha,ij}$, $t_{\alpha,ij}\geq0$, $\alpha=1,2$, in (\ref{bp17})
and $A_{\alpha,i,2}$ are strictly diagonal dominant matrix, we conclude that  $A_{\alpha,i,2}$, $i=1,2,\ldots,N_{x}-1$, $\alpha=1,2$,
are $M$-matrices  and $A_{\alpha,i,2}^{-1}\geq\emph{O}$ (Corollary 3.20, \cite{Varga2000}),   which leads to $(A_{\alpha,i,2}+c_{\alpha,2}I)^{-1}\geq\emph{O}$, where $\emph{O}$ is the ($N_{y}-1)\times (N_{y}-1$) null  matrix. From here, we obtain that
\begin{equation*}
 \overline{Z}_{1,1,2}^{(1)}\leq \mathbf{0},\quad  \underline{Z}_{2,1,2}^{(1)}\geq \mathbf{0}.
\end{equation*}
From here, for $i=2$ in (\ref{bp105}), in a similar manner, we conclude that
\begin{equation*}
 \overline{Z}_{1,2,2}^{(1)}\leq \mathbf{0},\quad  \underline{Z}_{2,2,2}^{(1)}\geq \mathbf{0}.
\end{equation*}
By induction on $i$, we can prove that
\begin{equation*}
 \overline{Z}_{1,i,2}^{(1)}\leq \mathbf{0},\quad  \underline{Z}_{2,i,2}^{(1)}\geq \mathbf{0},\quad i=0,1,\ldots,N_{x}.
\end{equation*}
By a similar argument, for  $\{\underline{U}^{(n)}_{1,i,2},\overline{U}^{(n)}_{2,i,2}\} $, from (\ref{bp95b}) and (\ref{bp96c}),  we can prove that
\begin{equation*}
 \underline{Z}_{1,i,2}^{(1)}\geq \mathbf{0},\quad  \overline{Z}_{2,i,2}^{(1)}\leq \mathbf{0},\quad i=0,1,\ldots,N_{x}.
\end{equation*}
The proof that $\overline{U}^{(1)}_{\alpha,i,2}$ and  $\underline{U}^{(1)}_{\alpha,i,2}$, $\alpha=1,2$,
are ordered upper and lower solutions (\ref{bp40})
 repeats  the proof on the first
 time level $m=1$. By induction on $n$, we can prove (\ref{bp98}) on the second time level $m=2$. By induction on $m$, we can prove (\ref{bp98}) for  $m\geq1$.
\end{proof}
\subsection{Existence and uniqueness of  a solution to the nonlinear difference   scheme (\ref{bp18})}
In the following theorem, we prove the existence of a solution to  (\ref{bp18}) based on Theorem \ref{bp97}.
\begin{theorem}\label{bp106}
 Let  $f(p,t_{m},U)$ satisfy (\ref{bp20}), where $\widetilde{U}_{\alpha,i,m}$ and $\widehat{U}_{\alpha,i,m}$,
  $i=0,1\ldots,N_{x}$, $\alpha=1,2$,
   $m\geq1$, be ordered upper and lower solutions (\ref{bp40}) to (\ref{bp18}).
  Then a solution of the nonlinear implicit difference  scheme (\ref{bp18})
  exists in $\langle\widehat{U}(t_{m}),\widetilde{U}(t_{m})\rangle$, $m\geq1$.
  \end{theorem}
  \begin{proof}
  We consider  the  Gauss--Seidel method $(\eta=1)$ in (\ref{bp94}).
 On the first time level $t_{1}$, from (\ref{bp98}), we conclude that
  $\lim \overline{U}^{(n)}_{\alpha,i,1}=\overline{V}_{\alpha,i,1}$,
  $\lim \underline{U}^{(n)}_{\alpha,i,1}=\underline{V}_{\alpha,i,1}$, $i=0,1,\ldots,N_{x}$, $\alpha=1,2$
  as $n\rightarrow \infty$ exist, and
  \begin{eqnarray}\label{bp107}
&&\widehat{U}_{\alpha,i,1} \leq   \underline{U}^{(n-1)}_{\alpha,i,1}\leq
\underline{U}^{(n)}_{\alpha,i,1}\leq\underline{V}_{\alpha,i,1},\quad \overline{V}_{\alpha,i,1}\leq   \overline{U}^{(n)}_{\alpha,i,1}\leq\overline{U}^{(n-1)}_{\alpha,i,1}\leq \widetilde{U}_{\alpha,i,1},\nonumber\\
&& \lim_{n \rightarrow \infty} \overline{Z}^{(n)}_{\alpha,i,1}=\mathbf{0},\quad \lim_{n \rightarrow \infty} \underline{Z}^{(n)}_{\alpha,i,1}=\mathbf{0},\quad i=0,1,\ldots,N_{x},\quad \alpha=1,2,
  \end{eqnarray}
  where $\overline{U}^{(0)}_{\alpha,i,1}=\widetilde{U}_{\alpha,i,1}$,
  $\underline{U}^{(0)}_{\alpha,i,1}=\widehat{U}_{\alpha,i,1}$. Similar to (\ref{bp101}), we have
  \begin{align}\label{bp108}
&\mathcal{G}_{1,i,1}\left(\overline{U}_{1,i,1}^{(n)}, \psi_{1,i}, \underline{U}_{2,i,1}^{(n)}\right)=-\left(c_{1,1}-\frac{\partial F_{1,i,1}(\overline{E}^{(n)}_{1,i,1},\underline{U}^{(n)}_{2,i,1})}{\partial u_{1}}\right) \overline{Z}_{1,i,1}^{(n)}\nonumber\\&+\frac{\partial F_{1,i,1}(\overline{U}^{(n-1)}_{1,i,1},\underline{E}^{(n)}_{2,i,1})}{\partial u_{2}}\underline{Z}_{2,i,1}^{(n)}
-R_{1,i,1}\overline{Z}_{1,i+1,1}^{(n)},\quad i=1,2, \ldots, N_{x}-1, \nonumber\\
&\overline{U}^{(n)}_{1,i,1}\leq\overline{E}^{(n)}_{1,i,1}\leq \overline{U}^{(n-1)}_{1,i,1},\quad \underline{U}^{(n-1)}_{2,i,1}\leq\underline{E}^{(n)}_{2,i,1}\leq \underline{U}^{(n)}_{2,i,1} , \quad i=0,1,\ldots,N_{x}.
\end{align}
By taking the limit of both side of (\ref{bp108}) and using (\ref{bp107}), we conclude that
\begin{equation}\label{bp109}
  \mathcal{G}_{1,i,1}\left(\overline{V}_{1,i,1}, \psi_{1,i}, \underline{V}_{2,i,1}\right)=\mathbf{0},\quad i=1,2,\ldots,N_{x}-1.
\end{equation}
Similarly, we have
\begin{equation}\label{bp110}
  \mathcal{G}_{2,i,1}\left(\underline{V}_{2,i,1}, \psi_{1,i},\overline{V}_{1,i,1} \right)=\mathbf{0},\quad i=1,2,\ldots,N_{x}-1.
\end{equation}
In a similar manner, we can prove that
\begin{eqnarray}\label{bp111}
&&  \mathcal{G}_{1,i,1}\left(\underline{V}_{1,i,1}, \psi_{1,i}, \overline{V}_{2,i,1}\right)=\mathbf{0},\quad \mathcal{G}_{2,i,1}\left(\overline{V}_{2,i,1}, \psi_{1,i}, \underline{V}_{1,i,1}\right)=\mathbf{0},\nonumber  \\
&& i=1,2,\ldots,N_{x}-1.
\end{eqnarray}
From (\ref{bp109})--(\ref{bp111}), we conclude that $\overline{V}_{1,i,1}$,  $\underline{V}_{2,i,1}$ and
$\underline{V}_{1,i,1}$,  $\overline{V}_{2,i,1}$,  $i=0,1,\ldots, \\ N_{x}$,  solve
(\ref{bp18}).

By the assumption of the theorem that $\widetilde{U}_{\alpha,i,2}$, $ \widehat{U}_{\alpha,i,2} $ $i=0,1,\ldots,N_{x}$, $\alpha=1,2$,
are ordered  upper and lower  solutions and from (\ref{bp107}), it follows that $\widetilde{U}_{\alpha,i,2}$ and  $\widehat{U}_{\alpha,i,2}$,
 $i=0,1,\ldots,N_{x}$, $\alpha=1,2$, are upper and lower
solutions with respect to, respectively,  $\overline{V}_{\alpha,i,1}$ and
 $\underline{V}_{\alpha,i,1}$ $i=0,1,\ldots,N_{x}$, $\alpha=1,2$. Indeed from (\ref{bp94a}) and  (\ref{bp107}), we have
 \begin{align*}
&\mathcal{G}_{1,i,2}\left(\widetilde{U}_{1,i,2} ,\overline{V}_{1,i,1}, \widehat{U}_{2,i,2}\right )= \\&A_{1,i,2}\widetilde{U}_{1,i,2}-L_{1,i,2}\widetilde{U}_{1,i-1,2}
-R_{1,i,2}\widetilde{U}_{1,i+1,2}+F_{1,i,2}(\widetilde{U}_{1,i,2}, \widehat{U}_{2,i,2}) \\&-\tau^{-1}\overline{V}_{1,i,1} +G_{1,i,2}^{*}\geq \mathcal{G}_{1,i,2}\left(\widetilde{U}_{1,i,2} ,\widetilde{U}_{1,i,1}, \widehat{U}_{2,i,2}\right )\geq \mathbf{0},\\
&i=1,2\ldots,N_{x}-1,
\end{align*}
\begin{align*}
&\mathcal{G}_{2,i,2}\left(\widehat{U}_{2,i,2} ,\underline{V}_{2,i,1}, \widetilde{U}_{1,i,2}\right )= \\&A_{2,i,2}\widehat{U}_{2,i,2}-L_{2,i,2}\widehat{U}_{2,i-1,2}
-R_{2,i,2}\widehat{U}_{2,i+1,2}+F_{2,i,2}(\widehat{U}_{2,i,2}, \widetilde{U}_{1,i,2}) \\&-\tau^{-1}\underline{V}_{2,i,1} +G_{2,i,2}^{*}\leq \mathcal{G}_{2,i,2}\left(\widehat{U}_{2,i,2} ,\widehat{U}_{2,i,1}, \widetilde{U}_{1,i,2}\right )
\leq \mathbf{0},\\
&i=1,2\ldots,N_{x}-1.
\end{align*}
By a similar manner, from (\ref{bp95b}) and (\ref{bp107}),  we can prove that
\begin{align*}
&\mathcal{G}_{1,i,2}\left(\widehat{U}_{1,i,2} ,\underline{V}_{1,i,1}, \widetilde{U}_{2,i,2}\right )\leq\mathbf{0},\quad
\mathcal{G}_{2,i,2}\left(\widetilde{U}_{2,i,2} ,\overline{V}_{2,i,1}, \widehat{U}_{1,i,2}\right )
\geq \mathbf{0},\\
&i=1,2\ldots,N_{x}-1.
\end{align*}

Using a similar argument as in (\ref{bp107}), we can prove that the limits
\begin{equation*}
  \lim_{n\rightarrow \infty} \overline{U}^{(n)}_{\alpha,i,2}=\overline{V}_{\alpha,i,2},\quad  \lim_{n\rightarrow \infty} \underline{U}^{(n)}_{\alpha,i,2}=\underline{V}_{\alpha,i,2},\quad i=0,1,\ldots,N_{x},\quad \alpha=1,2,
\end{equation*}
 exist and solve (\ref{bp18}) on the second time level $m=2$.

By induction on $m$, $m\geq1$, we can prove that
\begin{eqnarray*}
&& \lim_{n\rightarrow \infty} \overline{U}^{(n)}_{\alpha,i,m}=\overline{V}_{\alpha,i,m} ,\quad
  \lim_{n\rightarrow \infty} \underline{U}^{(n)}_{\alpha,i,m}=\underline{V}_{\alpha,i,m},\quad i=0,1,\ldots,N_{x}.\\
 &&\alpha=1,2,\quad m\geq1.
\end{eqnarray*}
Thus, $(\overline{V}_{1,i,m}, \underline{V}_{2,i,m}) $ and $ (\underline{V}_{1,i,m}, \overline{V}_{2,i,m}) $, $i=0,1,\ldots,N_{x}  $, $m\geq1$,
are solutions of the nonlinear difference scheme (\ref{bp18}).
   \end{proof}
We now assume that the reaction functions $f_{\alpha}$, $\alpha=1,2$, satisfy (\ref{bp61}) and the two-sided constrains
\begin{equation}\label{bp112}
-q_{\alpha}(p,t_{m})\leq - \frac{\partial f_{\alpha}(p,t_{m},U)}{\partial u_{\alpha^{\prime}}}\leq 0,\quad U\in\langle \widehat{U}(t_{m}), \widetilde{U}(t_{m})\rangle, \quad p\in\overline{\Omega}^{h},
\end{equation}
\begin{equation*}
\alpha^{\prime}\neq \alpha,\quad \alpha, \alpha^{\prime}=1,2,
  \quad m\geq1,
\end{equation*}
where  $q_{\alpha}(p,t_{m})$, $\alpha=1,2$,
are  nonnegative  bounded functions. It is assumed that the time step $\tau$ satisfies the assumptions in  (\ref{bp62}).
\begin{theorem}\label{bp113}
 Suppose that functions $f_{\alpha}(p,t_{m},U)$, $\alpha=1,2$, satisfy (\ref{bp61})
 and (\ref{bp112}),
 where $\widetilde{U}(p,t_{m})$ and $\widehat{U}(p,t_{m})$ are ordered
 upper and lower solutions (\ref{bp40}) of (\ref{bp9}). Let assumption
  (\ref{bp62}) on time step $\tau$ be satisfied. Then the  nonlinear
 difference scheme (\ref{bp9}) has a unique solution.
\end{theorem}
\begin{proof}
 To prove the uniqueness of a solution to the  nonlinear difference scheme (\ref{bp9}), it suffices to prove that
  \begin{equation*}
    \overline{V}_{\alpha}(p,t_{m})=\underline{V}_{\alpha}(p,t_{m}),\quad p\in \overline{\Omega}^{h},\quad \alpha=1,2,\quad m\geq1,
  \end{equation*}
where $(\overline{V}_{1}(p,t_{m}), \underline{V}_{2}(p,t_{m})) $ and $ (\underline{V}_{1}(p,t_{m}), \overline{V}_{2}(p,t_{m})) $,
$p\in \overline{\Omega}^{h}$, $m\geq1$, are the  solutions  to
  the  nonlinear difference scheme (\ref{bp9}), which are   defined in the proof of  Theorem \ref{bp106}. From (\ref{bp98}) and Theorem \ref{bp106}, we obtain
   \begin{eqnarray}\label{bp114}
 &&  \underline{U}_{\alpha}^{(n)}(p,t_{m})\leq \underline{V}_{\alpha}(p,t_{m})\leq \overline{V}_{\alpha}(p,t_{m})\leq \overline{U}^{(n)}_{\alpha}(p,t_{m}),\ \ p\in\overline{\Omega}^{h},\ \ \alpha=1,2,\nonumber\\
 &&m\geq1.
  \end{eqnarray}
  Letting $W_{\alpha}(p,t_{m})=\overline{V}_{\alpha}(p,t_{m})-\underline{V}_{\alpha}(p,t_{m})$, from (\ref{bp9}), we have
  \begin{eqnarray*}
&&\left(\mathcal{L}_{\alpha}^{h}(p,t_{m})+\tau^{-1}\right)W_{\alpha}(p,t_{m}) +\left[f_{\alpha}(\overline{V}_{\alpha},\underline{V}_{\alpha^{\prime}})- f_{\alpha}(\underline{V}_{\alpha}, \underline{V}_{\alpha^{\prime}})\right]\\
&&+
\left[f_{\alpha}(\underline{V}_{\alpha},\underline{V}_{\alpha^{\prime}})- f_{\alpha}(\underline{V}_{\alpha}, \overline{V}_{\alpha^{\prime}})\right]
-\tau^{-1}W_{\alpha}(p,t_{m-1})=0,\quad p\in \Omega^{h},\\
&& W_{\alpha}(p,t_{m})=0,\quad p\in \partial \Omega^{h},\quad W_{\alpha}(p,0)=0,\quad p\in \overline{\Omega}^{h},\quad \alpha^{\prime}\neq \alpha,\\
&& \alpha, \alpha^{\prime}=1,2,\quad m\geq1.
\end{eqnarray*}
Using the mean-value theorem (\ref{bp16}), we obtain
\begin{eqnarray}\label{bp115}
&&\left(\mathcal{L}_{\alpha}^{h}(p,t_{m})+\left(\tau^{-1}+\frac{\partial f_{\alpha}(p,t_{m},H_{\alpha})}{\partial u_{\alpha}}\right)\right)W_{\alpha}(p,t_{m})
=\\
&&\frac{\partial f_{\alpha}(p,t_{m},H_{\alpha^{\prime}})}{\partial u_{\alpha^{\prime}}}W_{\alpha^{\prime}}(p,t_{m})+\tau^{-1}W_{\alpha}(p,t_{m-1}),
\quad p\in \Omega^{h}, \nonumber\\
&& W_{\alpha}(p,t_{m})=0,\quad p\in \partial \Omega^{h},\quad m\geq1,\quad W_{\alpha}(p,0)=0,\quad p\in \overline{\Omega}^{h},\nonumber\\
&& \underline{V}_{\alpha}(p,t_{m})\leq H_{\alpha}(p,t_{m})\leq \overline{V}_{\alpha}(p,t_{m}),\quad \alpha^{\prime}\neq \alpha,\quad \alpha, \alpha^{\prime}=1,2.\nonumber
\end{eqnarray}
 From here and (\ref{bp114}), it follows that the partial derivatives satisfy (\ref{bp61}) and (\ref{bp112}).
 If $\underline{c}_{1}\geq0$ in (\ref{bp62}), from (\ref{bp115}) for $m=1$,  using (\ref{bp13}),  (\ref{bp61}), (\ref{bp112}) and
 taking into account that $W_{\alpha}(p,0)=0$, we conclude that
 \begin{equation*}
   W(t_{1})\leq\frac{\tau q_{1}}{1+\tau \underline{c}_{1}}W(t_{1}),
 \end{equation*}
where
\begin{eqnarray*}
&&  W(t_{m})=\max_{\alpha=1,2}W_{\alpha}(t_{m}),\quad W_{\alpha}(t_{m})=\|W_{\alpha}(\cdot,t_{m})\|_{\overline{\Omega}^{h}}, \\
&& \|W_{\alpha}(\cdot,t_{m})\|_{\overline{\Omega}^{h}}=\max_{p\in\Omega^{h}}|W_{\alpha}(p,t_{m})|,\quad \alpha=1,2.
\end{eqnarray*}
From here, by the assumption on $\tau$ in (\ref{bp62}) and taking into account that $W(t_{m})\geq0$, we conclude that $W(t_{1})=0$.

If $\underline{c}_{1}<0$ in (\ref{bp62}), from (\ref{bp115}) for $m=1$,  using (\ref{bp13}), (\ref{bp61}) and (\ref{bp112}),  we conclude that
 \begin{equation*}
   W(t_{1})\leq\frac{\tau q_{1}}{1-\tau | \underline{c}_{1}|}W(t_{1}).
 \end{equation*}
 From here, by the assumption on $\tau$ in (\ref{bp62}) and taking into account that $W(t_{m})\geq0$, we conclude that $W(t_{1})=0$.

By induction on $m$, we can prove that $W(t_{m})=0$, $m\geq1$.
Thus, we prove the theorem.
\end{proof}
\subsection{Convergence analysis}
For the sequences $\{\overline{U}^{(n)}_{1,i,m}, \underline{U}^{(n)}_{2,i,m}\}$
and $\{\underline{U}^{(n)}_{1,i,m}, \overline{U}^{(n)}_{2,i,m}\}$ generated by (\ref{bp94}), we introduce the notation
\begin{equation}\label{bp115a}
\mathcal{G}_{1}(t_{m})=\left\{ \begin{array}{ll}
\left\|\mathcal{G}_{1}\left(\overline{U}_{1}^{(n)}(\cdot,t_{m}), \overline{U}_{1}(\cdot,t_{m-1}), \underline{U}_{2}^{(n)}(\cdot,t_{m})\right)\right\|_{p\in \Omega^{h}},\ \
 \mbox{for}\  (\ref{bp94a}), \\ \\
\left\|\mathcal{G}_{1}\left(\underline{U}_{1}^{(n)}(\cdot,t_{m}), \underline{U}_{1}(\cdot,t_{m-1}), \overline{U}_{2}^{(n)}(\cdot,t_{m})\right)\right\|_{p\in \Omega^{h}},
\ \ \mbox{for}\  (\ref{bp95b}),
\end{array}\right.
 \end{equation}
 \begin{equation*}
\mathcal{G}_{2}(t_{m})=\left\{ \begin{array}{ll}
\left\|\mathcal{G}_{2}\left(\underline{U}_{2}^{(n)}(\cdot,t_{m}), \underline{U}_{2}(\cdot,t_{m-1}), \overline{U}_{1}^{(n)}(\cdot,t_{m})\right)\right\|_{p\in \Omega^{h}},
\ \ \mbox{for}\  (\ref{bp94a}), \\ \\
\left\|\mathcal{G}_{2}\left(\overline{U}_{2}^{(n)}(\cdot,t_{m}), \overline{U}_{2}(\cdot,t_{m-1}), \underline{U}_{1}^{(n)}(\cdot,t_{m})\right)\right\|_{p\in \Omega^{h}}, \ \ \mbox{for}\  (\ref{bp95b}),
\end{array}\right.
 \end{equation*}
 where the residuals   $ \mathcal{G}_{\alpha}\left(U_{\alpha}^{(n)}(p,t_{m}), U_{\alpha}(p,t_{m-1}), U_{\alpha^{\prime}}^{(n)}(p,t_{m})\right) $, $\alpha^{\prime}\neq \alpha $,
$\alpha, \alpha^{\prime}\\=1,2$, are
defined in (\ref{bp18a}),  the notation  of the norm  from (\ref{bp13}) is in use.

A stopping test for the block monotone iterative methods (\ref{bp94}) is chosen in the following  form
\begin{equation}\label{bp116}
  \max_{m\geq1}\left[\mathcal{G}_{1}(t_{m}), \mathcal{G}_{2}(t_{m})\right]\leq \delta,
\end{equation}
where  $\mathcal{G}_{\alpha}(t_{m}) $, $\alpha=1,2$, are defined in (\ref{bp115a}),  $\delta$ is a
prescribed accuracy. On each time level $t_{m}$,  $m\geq1$,
we set up $U_{\alpha}(p,t_{m})=U^{(n_{m})}_{\alpha}(p,t_{m})$, $ p \in \Omega^{h}$,  $\alpha=1,2$,
 such that $m_{n}$
is the minimal number of iterations  subject to (\ref{bp116}).
\begin{theorem}\label{bp117}
Let $\widetilde{U}(p,t_{m})$ and $\widehat{U}(p,t_{m})$ be ordered
 upper and lower solutions (\ref{bp40}) of (\ref{bp9}). Suppose that functions $f_{\alpha}(p,t_{m},U)$, $\alpha=1,2$, satisfy (\ref{bp65a})
 and (\ref{bp112}).
  Then for the sequences  of solutions $\{\overline{U}_{1}^{(n)}(p,t_{m}),\underline{U}^{(n)}_{2}(p,t_{m}) \}$ and
 $\{\underline{U}_{1}^{(n)}(p,t_{m}),\overline{U}^{(n)}_{2}(p,t_{m}) \}$   generated
 by (\ref{bp94}), (\ref{bp116}), the following estimate holds
 \begin{equation}\label{bp118}
\max_{m\geq1} \max_{\alpha=1,2}\|U_{\alpha}(\cdot, t_{m})-U_{\alpha}^{*}(\cdot, t_{m})\|_{\overline{\Omega}^{h}}\leq T \delta.
  \end{equation}
where $U_{\alpha}(p,t_{m})=U_{\alpha}^{(n_{m})}(p,t_{m})$, $n_{m}$ is a minimal
  number of iterations subject to (\ref{bp116}), and   $U_{\alpha}^{*}(p,t_{m})$, $\alpha=1,2$, $m\geq1$, are the unique solutions to the  nonlinear difference scheme (\ref{bp9}).
  \end{theorem}
  \begin{proof}
    We consider the case of  the sequence $\{\overline{U}_{1}^{(n)}(p,t_{m}),\underline{U}^{(n)}_{2}(p,t_{m}) \}$. On a time level $t_{m}$, $m\geq1$, from  (\ref{bp9})
for $\overline{U}_{1}(p,t_{m})$, $\underline{U}_{2}(p,t_{m})$  and $U^{*}_{\alpha}(p,t_{m})$, $\alpha=1,2$ , we have
\begin{eqnarray*}
&&\left(\mathcal{L}_{1}^{h}(p,t_{m})+\tau^{-1}\right)\overline{U}_{1}(p,t_{m}) +f_{1}(p,t_{m},\overline{U}_{1}, \underline{U}_{2})-\tau^{-1} \overline{U}_{1}(p,t_{m-1})=\\
&& \mathcal{G}_{1}\left(\overline{U}_{1}(p,t_{m}),\overline{U}_{1}(p,t_{m-1}), \underline{U}_{2}(p,t_{m})\right),\quad  p\in \Omega^{h}, \\
&& \overline{U}_{1}(p,t_{m})=g_{1}(p,t_{m}),\ \  p\in \partial \Omega^{h},\quad m\geq1, \ \  \overline{U}_{1}(p,0)=\psi_{1}(p),\ \ p\in\overline{\Omega}^{h},\nonumber
\end{eqnarray*}
\begin{eqnarray*}
&&\left(\mathcal{L}_{2}^{h}(p,t_{m})+\tau^{-1}\right)\underline{U}_{2}(p,t_{m}) +f_{2}(p,t_{m},\underline{U}_{2}, \overline{U}_{1})-\tau^{-1} \underline{U}_{2}(p,t_{m-1})=\\
&& \mathcal{G}_{2}\left(\underline{U}_{2}(p,t_{m}),\underline{U}_{2}(p,t_{m-1}), \overline{U}_{1}(p,t_{m})\right),\quad  p\in \Omega^{h}, \\
&& \underline{U}_{2}(p,t_{m})=g_{2}(p,t_{m}),\ \ p\in \partial \Omega^{h},\ \  m\geq1,\ \ \underline{U}_{2}(p,0)=\psi_{2}(p),\ \ p\in\overline{\Omega}^{h},\nonumber
\end{eqnarray*}
\begin{eqnarray*}
&&\left(\mathcal{L}_{\alpha}^{h}(p,t_{m})+\tau^{-1}\right)U^{*}_{\alpha}(p,t_{m}) +f_{\alpha}(p,t_{m},U^{*})-\tau^{-1} U_{\alpha}^{*}(p,t_{m-1})=0,\\
&&p\in \Omega^{h},\quad \alpha=1,2,\quad U^{*}(p,t_{m})=g(p,t_{m}),\quad p\in \partial \Omega^{h},\\
&& U^{*}(p,0)=\psi(p),\quad p\in\overline{\Omega}^{h},\quad m\geq1.\nonumber
\end{eqnarray*}
Letting  $\overline{W}_{1}(p,t_{m})=\overline{U}_{1}(p,t_{m})-U_{1}^{*}(p,t_{m})$ and
$\underline{W}_{2}(p,t_{m})=U_{2}^{*}(p,t_{m})-\underline{U}_{2}(p,t_{m})$, $p\in \overline{\Omega}^{h}$,
 $m\geq1$, from here and using the mean-value theorem, we obtain
\begin{eqnarray*}
&&\left(\mathcal{L}_{1}^{h}(p,t_{m})+\left(\tau^{-1}+\frac{\partial f_{1}(p,t_{m},\overline{K}_{1}, \underline{U}_{2})}{\partial u_{1}}\right)I\right)\overline{W}_{1}(p,t_{m})=
 \\
&&+\frac{\partial f_{1}(p,t_{m}, U^{*}_{1},\underline{K}_{2})}{\partial u_{2}}\underline{W}_{2}(p,t_{m})+\tau^{-1} \overline{W}_{1}(p,t_{m-1}) \\
&&+\mathcal{G}_{1}\left(\overline{U}_{1}(p,t_{m}),\overline{U}_{1}(p,t_{m-1}), \underline{U}_{2}(p,t_{m})\right) ,\quad p\in \Omega^{h},\\
&&\overline{W}_{1}(p,t_{m})=0,\quad p\in \partial \Omega^{h},\quad W_{1}(p,0)=0,\quad p\in\overline{\Omega}^{h},\quad  m\geq1,
\end{eqnarray*}
\begin{eqnarray*}
&&\left(\mathcal{L}_{2}^{h}(p,t_{m})+\left(\tau^{-1}+\frac{\partial f_{2}(p,t_{m},\underline{K}_{2}, \overline{U}_{1})}{\partial u_{1}}\right)I\right)\underline{W}_{2}(p,t_{m})=
 \\
&&+\frac{\partial f_{2}(p,t_{m}, U^{*}_{1},\overline{K}_{1})}{\partial u_{2}}\overline{W}_{1}(p,t_{m})+\tau^{-1} \underline{W}_{2}(p,t_{m-1}) \\
&&+\mathcal{G}_{2}\left(\underline{U}_{2}(p,t_{m}),\underline{U}_{2}(p,t_{m-1}), \overline{U}_{1}(p,t_{m})\right) ,\quad p\in \Omega^{h},\\
&&\underline{W}_{2}(p,t_{m})=0,\quad p\in \partial \Omega^{h},\quad \underline{W}_{2}(p,0)=0,\quad p\in\overline{\Omega}^{h},\quad  m\geq1,
\end{eqnarray*}
where
\begin{eqnarray*}
&& U_{1}^{*}(p,t_{m})\leq   \overline{K}_{1}(p,t_{m})\leq  \overline{U}_{1}(p,t_{m}),\\
 &&\underline{U}_{2}(p,t_{m})\leq   \underline{K}_{2}(p,t_{m})\leq  U^{*}_{2}(p,t_{m}),\quad
 m\geq1.
\end{eqnarray*}
From here, (\ref{bp65a}) and  (\ref{bp112}), by using (\ref{bp13}), we obtain that
\begin{eqnarray}\label{bp119a}
&&\|\overline{W}_{1}(\cdot,t_{m})\|_{\overline{\Omega}^{h}} \leq\frac{1}{\tau^{-1}+q} \left( q \| \underline{W}_{2}(\cdot, t_{m})\|_{\Omega^{h}}+ \delta +\tau^{-1} \| \overline{W}_{1}(\cdot, t_{m-1})\|_{\Omega^{h}}   
\right),\nonumber\\
&&\|\underline{W}_{2}(\cdot, t_{m})\|_{\overline{\Omega}^{h}} \leq\frac{1}{\tau^{-1}+q} \left( q \| \overline{W}_{1}(\cdot, t_{m})\|_{\Omega^{h}}+ \delta +\tau^{-1}\| \underline{W}_{2}(\cdot, t_{m-1})\|_{\Omega^{h}}\right),\nonumber\\
\end{eqnarray}
where the notation of the norm from (\ref{bp13}) is in use.
Letting $W(t_{m})=\max \left\{\|\overline{W}_{1}(\cdot,t_{m})\|_{\overline{\Omega}^{h}}, \|\underline{W}_{2}(\cdot, t_{m})\|_{\overline{\Omega}^{h}}\right\}$, from (\ref{bp119a}), we have
\begin{equation*}
W(t_{m})\leq\frac{1}{\tau^{-1}+q} \left( q W(t_{m})+ \delta +\tau^{-1} W(t_{m-1})  
\right),
\end{equation*}
Taking into account that
\begin{equation*}
1-  \frac{q}{\tau^{-1}+q}>0,
\end{equation*}
it follows that
\begin{equation*}
W(t_{m})\leq \tau \delta + W(t_{m-1}).   
\end{equation*}
From here, taking into account that $W(t_{0})=0$, by induction on $m$,  we obtain that
\begin{equation*}
W(t_{m}) \leq \delta \sum_{\rho=1}^{m} \tau \leq \delta T.
\end{equation*}
Thus, we conclude that
\begin{equation*}
\|\overline{W}_{1}(\cdot,t_{m})\|_{\overline{\Omega}^{h}}   \leq \delta T,\quad  \|\underline{W}_{2}(\cdot, t_{m})\|_{\overline{\Omega}^{h}}\leq \delta T.
\end{equation*}
By a similar argument, for the sequence  $\{\underline{U}_{1}^{(n)}(p,t_{m}),\overline{U}^{(n)}_{2}(p,t_{m}) \}$,  we can prove that
\begin{equation*}
\|\underline{W}_{1}(\cdot,t_{m})\|_{\overline{\Omega}^{h}}   \leq \delta T,\quad  \|\overline{W}_{2}(\cdot, t_{m})\|_{\overline{\Omega}^{h}}\leq \delta T.
\end{equation*}
Thus, we prove the theorem.
 \end{proof}
 \begin{theorem}\label{bp191b}
  Let the assumptions in Theorem \ref{bp117} be satisfied.  Then for the sequences   $\{\overline{U}_{1}^{(n)}(p,t_{m}),\underline{U}^{(n)}_{2}(p,t_{m}) \}$ and
 $\{\underline{U}_{1}^{(n)}(p,t_{m}),\overline{U}^{(n)}_{2}(p,t_{m}) \}$   generated
 by (\ref{bp94}), (\ref{bp116}),  the following estimate holds
\begin{eqnarray}\label{bp119b1}
&& \max_{m\geq1}  \max_{\alpha=1,2}\|U_{\alpha}(\cdot, t_{m})-u_{\alpha}^{*}(\cdot, t_{m})\|_{\overline{\Omega}^{h}}\leq T\left(\delta +\max_{m\geq1} E_{m}\right), \\
&& E_{m}=\max_{\alpha=1,2}\|E_{\alpha}(\cdot,t_{m})\|_{\overline{\Omega}^{h}},\quad m\geq1, \nonumber
\end{eqnarray}
where $U_{\alpha}(p,t_{m})=U_{\alpha}^{(n_{m})}(p,t_{m})$,
$\alpha=1,2$, $m\geq1$, $n_{m}$ is the minimal number of iterations subject to the stopping test (\ref{bp116}),    $u_{\alpha}^{*}(x,y,t)$, $\alpha=1,2$, are the exact solutions to (\ref{bp1}), and   $E_{\alpha}(p,t_{m})$,
$\alpha=1,2$, $m\geq1$,
are the truncation errors of the exact solutions  $u_{\alpha}^{*}(x,y,t)$, $\alpha=1,2$,
 on the nonlinear
difference scheme (\ref{bp9}).
\end{theorem}
\begin{proof}
We denote $V(p,t_{m})=u^{*}(p,t_{m})- U^{*}(p,t_{m})
$, where the mesh vector function  $U^{*}(p,t_{m})$ is the unique solution of the nonlinear difference scheme (\ref{bp9}).
From (\ref{bp9}), by using the mean-value theorem, we obtain that
\begin{eqnarray*}
&&\left(\mathcal{L}_{\alpha}^{h}(p,t_{m})+\left(\tau^{-1}+\frac{\partial f_{\alpha}(p,t_{m},Y)}{\partial u_{\alpha}}\right)I\right)V_{\alpha}(p,t_{m})
-\tau^{-1} V_{\alpha}(p,t_{m-1}) \\
&&+\frac{\partial f_{\alpha}(p,t_{m},Y)}{\partial u_{\alpha^{\prime}}}V_{\alpha^{\prime}}(p,t_{m})=E_{\alpha}(p,t_{m}),\quad p\in \Omega^{h},\quad \alpha^{\prime}\neq \alpha,  \nonumber\\
&&\alpha, \alpha^{\prime}=1,2,\quad  V(p,t_{m})=0,\quad p\in \partial \Omega^{h},\quad V(p,0)=0,\quad p\in\overline{\Omega}^{h},\\
&& m\geq1,\nonumber
\end{eqnarray*}
where $Y_{\alpha}(p,t_{m})$, $\alpha=1,2$ lie between $u^{*}_{\alpha}(p,t_{m}) $ and $U^{*}_{\alpha}(p,t_{m})$, $\alpha=1,2$.
From here, (\ref{bp65a})
 and (\ref{bp112}), by using notation (\ref{bp13}), it follows that
\begin{eqnarray*}
&&\|V_{\alpha}(\cdot, t_{m})\|_{\overline{\Omega}^{h}}\leq \\
&&\frac{1}{\tau^{-1}+q}\left(q \|V_{\alpha^{\prime}}(\cdot,t_{m})\|_{\Omega^{h}}
+ \tau^{-1}\|V_{\alpha}(\cdot,t_{m-1})\|_{\Omega^{h}}+\|E_{\alpha}(\cdot,t_{m})\|_{\Omega^{h}} \right).
\end{eqnarray*}
Letting $V_{m}=\max_{\alpha=1,2}\|V_{\alpha}(\cdot, t_{m})\|_{\overline{\Omega}^{h}}$, $m\geq1$, we have
\begin{equation*}
V_{m}\leq
\frac{1}{\tau^{-1}+q}\left(q V_{m}
+ \tau^{-1}V_{m-1}+E_{m} \right).
\end{equation*}
From here and  taking into account that
\begin{equation*}
  1-\frac{q}{\tau^{-1}+q}>0,
\end{equation*}
we conclude
\begin{equation}\label{bp119b2}
  V_{m}\leq V_{m-1}+ \tau E_{m} .
\end{equation}
Since $V_{0}=0$, for $m=1$ in (\ref{bp119b2}), we have
\begin{equation*}
  V_{1}\leq  \tau E_{1} .
\end{equation*}
For $m=2$ in (\ref{bp119b2}), we obtain
\begin{equation*}
  V_{2}\leq  \tau ( E_{1} + E_{2}),
\end{equation*}
and by induction on $m$, we can prove that
\begin{equation*}
  V_{m}\leq  \tau \sum_{\rho=1}^{m}E_{\rho}=\left(\sum_{\rho=1}^{m}\tau\right) \max_{\rho\geq1}E_{\rho},\quad m\geq1.
\end{equation*}
Since $\sum_{\rho=1}^{m}\tau\leq T$, where $T$ is the final time, we have
\begin{equation}\label{bp119b3}
  V_{m}\leq  T \max_{\rho\geq1}E_{\rho}.
\end{equation}
We estimate the left hand side in (\ref{bp119b1}) as follows
\begin{eqnarray*}
 \|U_{\alpha}^{( n_{m})}(\cdot,t_{m})\pm U_{\alpha}^{*}(\cdot, t_{m}) -u_{\alpha}^{*}( \cdot, t_{m})\|_{\overline{\Omega}^{h}} &\leq& \|U_{\alpha}^{(n_{m})}(\cdot, t_{m})-U_{\alpha}^{*}(\cdot, t_{m})\|_{\overline{\Omega}^{h}} \\
&&+\|U_{\alpha}^{*}(\cdot, t_{m})-u_{\alpha}^{*}(\cdot, t_{m})\|_{\overline{\Omega}^{h}},
\end{eqnarray*}
where $U_{\alpha}^{*}(p,t_{m})$, $\alpha=1,2$, are the exact solutions of (\ref{bp9}).
From here and  (\ref{bp119b3}), we prove (\ref{bp119b1}).
\end{proof}
  \subsection{Construction of upper and lower solutions}
To start  the monotone iterative methods (\ref{bp94}), on each time level $t_{m}$, $m\geq1$,  initial iterations are needed.
In this section, we discuss the construction of initial iterations   $\widetilde{U}_{\alpha}(p, t_{m})$ and $\widehat{U}_{\alpha}(p, t_{m})$, $\alpha=1,2$.
\subsubsection{Bounded $f_{u}$}
Assume that  the functions $f_{\alpha}$, $g_{\alpha}$ and $\psi_{\alpha}$, $\alpha=1,2$, in (\ref{bp1}) satisfy the conditions
\begin{align}\label{bp119}
&f_{\alpha}(x,y,t,0_{\alpha}, u_{\alpha^{\prime}})\leq 0,\quad  f_{\alpha}(x,y,t,u_{\alpha}, 0_{\alpha^{\prime}})\geq -M_{\alpha},\quad u_{\alpha}(x,y,t)\geq0, \nonumber   \\
&(x,y,t)\in \overline{Q}_{T},\quad g_{\alpha}(x,y,t)\geq0,\quad (x,y,t) \partial Q_{T},\quad \psi_{\alpha}(x,y)\geq0,\nonumber \\
&(x,y)\in \overline{\omega},\quad\alpha=1,2,
\end{align}
where $ M_{\alpha}$, $\alpha=1,2$, are positive  constants and $0_{\alpha}$ means $u_{\alpha}(x,y,t)=0$. 
 We introduce the functions
 \begin{equation}\label{bp120}
 \widehat{U}_{\alpha}(p,t_{m})=
 \left\{ \begin{array}{ll}
\psi_{\alpha}(p), \quad m=0,\\
0,\quad\quad\quad m\geq1,
\end{array}\right.
 p\in\overline{\Omega}^{h},\quad \alpha=1,2,
 \end{equation}
 and the linear problems
\begin{align} \label{bp121}
&\left(\mathcal{L}_{\alpha}^{h}(p,t_{m})+\tau^{-1}\right)\widetilde{U}_{\alpha}(p,t_{m})  =\tau^{-1}\widetilde{U}_{\alpha}(p,t_{m-1})+M_{\alpha},\quad p\in \Omega^{h}, \\
& \widetilde{U}_{\alpha}(p,t_{m})=g_{\alpha}(p,t_{m}),\quad p\in \partial \Omega^{h},\quad \widetilde{U}_{\alpha}(p,0)= \psi_{\alpha}(p),\quad p\in \overline{\Omega}^{h}, \nonumber \\
& \alpha=1,2,\quad m\geq1. \nonumber
\end{align}
\begin{theorem}\label{bp122}
  Let the  assumptions in (\ref{bp119}) be satisfied. Then $\widehat{U}_{\alpha}$, $\alpha=1,2$, from  (\ref{bp120})
  and solutions  $\widetilde{U}_{\alpha}$, $\alpha=1,2$, of the linear problems  (\ref{bp121}) are ordered lower and upper
  solutions   (\ref{bp40}) to (\ref{bp9}).
\end{theorem}
\begin{proof}
From (\ref{bp119}) and (\ref{bp121}) with  $m=1$, by using the maximum principle in Lemma \ref{bp12}, we obtain that
\begin{equation*}
 \widetilde{U}_{\alpha}(p,t_{1})\geq0,\quad p\in \overline{\Omega}^{h},\quad \alpha=1,2.
\end{equation*}
From here and (\ref{bp121})  with $m=2$, by using the maximum principle in Lemma \ref{bp12}, we have
\begin{equation*}
 \widetilde{U}_{\alpha}(p,t_{2})\geq0,\quad p\in \overline{\Omega}^{h},\quad \alpha=1,2.
\end{equation*}
By induction on $m$, we can prove that
\begin{equation*}
 \widetilde{U}_{\alpha}(p,t_{m})\geq0,\quad p\in \overline{\Omega}^{h},\quad \alpha=1,2,\quad m\geq1.
\end{equation*}
 From here and (\ref{bp120}), we prove (\ref{bp40a}).

We now prove (\ref{bp40b})  for $(\widetilde{U}_{1}(p,t_{m}), \widehat{U}_{2}(p,t_{m}))$.
 We present the left hand side of (\ref{bp40b}) in the form
     \begin{align}\label{bp124}
  &  \mathcal{G}_{1}\left(\widetilde{U}_{1}(p,t_{m}), \widetilde{U}_{1}(p,t_{m-1}), \widehat{U}_{2}(p,t_{m})\right)=\\
  &\left(\mathcal{L}_{1}^{h}(p,t_{m})+\tau^{-1}\right)\widetilde{U}_{1}(p,t_{m})+f_{1}(p,t_{m},\widetilde{U}_{1}, \widehat{U}_{2})
  -\tau^{-1}\widetilde{U}_{1}(p,t_{m-1}), \nonumber\\
 & p\in \Omega^{h}, \quad m\geq1. \nonumber
  \end{align}
  Using  (\ref{bp121}) for $m\geq1$, we obtain that
   \begin{eqnarray*}
   &&   \mathcal{G}_{1}\left(\widetilde{U}_{1}(p,t_{m}), \widetilde{U}_{1}(p,t_{m-1}), \widehat{U}_{2}(p,t_{m})\right)= M_{1}+f_{1}(p,t_{m},\widetilde{U}_{1},0_{2}),\\
   && p\in \Omega^{h},\quad m\geq1.
   \end{eqnarray*}
   From here and using (\ref{bp119}), it follows that
  \begin{eqnarray*}
  &&  \mathcal{G}_{1}\left(\widetilde{U}_{1}(p,t_{m}), \widetilde{U}_{1}(p,t_{m-1}), \widehat{U}_{2}(p,t_{m})\right)\geq0,\quad
    p\in \Omega^{h},\quad m\geq1.
  \end{eqnarray*}
 Similarly, we can prove that
  \begin{eqnarray*}
  &&  \mathcal{G}_{2}\left(\widehat{U}_{2}(p,t_{m}), \widehat{U}_{2}(p,t_{m-1}), \widetilde{U}_{1}(p,t_{m})\right)\leq0,\quad
    p\in \Omega^{h},\quad m\geq1.
  \end{eqnarray*}
  Thus, we prove (\ref{bp40b}) for $(\widetilde{U}_{1}(p,t_{m}), \widehat{U}_{2}(p,t_{m}))$.  By following a similar argument, we can prove (\ref{bp40b}) for $(\widehat{U}_{1}(p,t_{m}), \widetilde{U}_{2}(p,t_{m}))$, that is,
    \begin{eqnarray*}
  &&  \mathcal{G}_{1}\left(\widehat{U}_{1}(p,t_{m}), \widehat{U}_{1}(p,t_{m-1}), \widetilde{U}_{2}(p,t_{m})\right)\leq0,\quad
 \\
 &&  \mathcal{G}_{2}\left(\widetilde{U}_{2}(p,t_{m}), \widetilde{U}_{2}(p,t_{m-1}), \widehat{U}_{1}(p,t_{m})\right)\geq0,\quad   p\in \Omega^{h},\quad m\geq1.
  \end{eqnarray*}
 Since $\widetilde{U}_{\alpha}(p,t_{m})$, $\alpha=1,2$,   satisfy the boundary and initial conditions (\ref{bp40c}),  $\widehat{U}_{\alpha}(p,0)$, $\alpha=1,2$, satisfy
  the initial condition and  $\widehat{U}_{\alpha}(p,t_{m})=0\leq g_{\alpha}(p,t_{m})$, $p\in \partial \Omega^{h}$,  $\alpha=1,2$,
   we conclude  that $\widehat{U}_{\alpha}(p,t_{m})$ and $\widetilde{U}_{\alpha}(p,t_{m})$, $\alpha=1,2$,  from, respectively, (\ref{bp120}) and (\ref{bp121}), are ordered lower and upper solutions (\ref{bp40}) to (\ref{bp9}).
\end{proof}
\subsubsection{Constant upper and lower solutions}
Let  the functions $f_{\alpha}$, $g_{\alpha}$ and $\psi_{\alpha}$, $\alpha=1,2$, in (\ref{bp1}) satisfy the conditions
\begin{align}\label{bp125}
&f_{\alpha}(x,y,t,0_{\alpha},u_{\alpha^{\prime}})\leq 0,\quad f_{\alpha}(x,y,t,K_{\alpha},0_{\alpha^{\prime}})\geq0,\quad u_{\alpha}(x,y,t)\geq0,\nonumber\\
&(x,y,t)\in \overline{Q}_{T},\quad 0\leq g_{\alpha}(x,y,t)\leq K_{\alpha},\ \ (x,y,t)\in \partial Q_{T},  \\
&0\leq\psi_{\alpha}(x,y)\leq K_{\alpha},\quad (x,y)\in\overline{\omega},\quad \alpha^{\prime}\neq \alpha,\quad  \alpha, \alpha^{\prime}=1,2, \nonumber
\end{align}
where $K_{1}$, $K_{2}$ are positive
constants. Introduce  the mesh functions
\begin{equation}\label{bp126}
\widetilde{U}_{\alpha}(p,t_{m})= \left\{ \begin{array}{ll}
\psi_{\alpha}(p),\quad p\in \overline{\Omega}^{h},\quad   m=0,\\
K_{\alpha},\quad  m\geq1,
\end{array}\right.
 \alpha=1,2.
\end{equation}
\begin{theorem}\label{bp127}
Suppose that the assumptions in  (\ref{bp125}) are satisfied. Then the mesh functions $\widehat{U}_{\alpha}(p,t_{m})$ and $\widetilde{U}_{\alpha}(p,t_{m})$,
$\alpha=1,2$,  from,
respectively, (\ref{bp120}) and (\ref{bp126}), are ordered lower and upper solutions (\ref{bp40}) to (\ref{bp9}).
\end{theorem}
\begin{proof}
From (\ref{bp120}) and (\ref{bp126}), it is clear that the inequalities in  (\ref{bp40a}) are satisfied.
We now prove (\ref{bp40b}) for $(\widetilde{U}_{1}(p,t_{m}), \widehat{U}_{2}(p,t_{m})) $.
 Using (\ref{bp126}), we write  the left hand side of (\ref{bp40b}) for $m=1$  in the form
     \begin{align*}
  &  \mathcal{G}_{1}\left(\widetilde{U}_{1}(p,t_{1}), \psi_{1}(p), \widehat{U}_{2}(p,t_{1})\right)=f_{1}(p,t_{1},K_{1}, 0_{2})+\tau^{-1}(K_{1}-\psi_{1}(p)), \\
 &p\in \Omega^{h}.
  \end{align*}
From here and  (\ref{bp125}), we conclude that
\begin{equation*}
 \mathcal{G}_{1}\left(\widetilde{U}_{1}(p,t_{1}), \psi_{1}(p), \widehat{U}_{2}(p,t_{1})\right)\geq
 0,\quad p\in \Omega^{h}.
\end{equation*}
From (\ref{bp125}) and (\ref{bp126}), using (\ref{bp40b}) for $m\geq2$,  we have
\begin{equation*}
 \mathcal{G}_{1}\left(\widetilde{U}_{1}(p,t_{m}), \widetilde{U}_{1}(p,t_{m-1}), \widehat{U}_{2}(p,t_{m})\right)= f_{1}(p,t_{m},K_{1}, 0_{2})\geq0,\quad p\in \Omega^{h}.
\end{equation*}
Similarly, we can prove
\begin{equation*}
 \mathcal{G}_{2}\left(\widehat{U}_{2}(p,t_{m}), \widehat{U}_{2}(p,t_{m-1}), \widetilde{U}_{1}(p,t_{m})\right)\leq0,\quad p\in \Omega^{h},\quad m\geq1.
\end{equation*}
Thus, we prove (\ref{bp40b}) for $(\widetilde{U}_{1}(p,t_{m}), \widehat{U}_{2}(p,t_{m}) ) $.

By a similar argument, we can prove (\ref{bp40b}) for $(\widehat{U}_{1}(p,t_{m}), \widetilde{U}_{2}(p,t_{m}))$, that is,
\begin{eqnarray*}
&& \mathcal{G}_{1}\left(\widehat{U}_{1}(p,t_{m}), \widehat{U}_{1}(p,t_{m-1}), \widetilde{U}_{2}(p,t_{m})\right)\leq0,
\\ && \mathcal{G}_{2}\left(\widetilde{U}_{2}(p,t_{m}), \widetilde{U}_{2}(p,t_{m-1}), \widehat{U}_{1}(p,t_{m})\right)\geq0 ,\quad p\in \Omega^{h} \quad m\geq1.
\end{eqnarray*}
Since $\widetilde{U}_{\alpha}(p,t_{0})$,  $\widehat{U}_{\alpha}(p,t_{0})$ , $\alpha=1,2$, satisfy the  initial condition
 and $\widetilde{U}_{\alpha}(p,t_{m})\\ \geq g_{\alpha}(p,t_{m})$, $\widehat{U}_{\alpha}(p,t_{m})\leq g_{\alpha}(p,t_{m})$, $p\in\partial \Omega^{h}$, $\alpha=1,2$, at $m\geq1$,
   we conclude that $\widehat{U}$ and $\widetilde{U}$ from, respectively,
 (\ref{bp120}) and (\ref{bp126}),
are ordered lower and upper solutions (\ref{bp40}) to (\ref{bp9}).
\end{proof}
\subsection{Applications}
\subsection{The Belousov-Zhabotinskii reaction diffusion system}
 The Belousov-Zhabotinskii reaction diffusion  model \cite{P92} includes   the metal-ion-catalyzed oxidation
 by bromate ion of brominated organ materials. the chemical reaction scheme is given by
 \begin{equation*}
  A_{1}+Y \rightarrow X,\quad X+Y\rightarrow P_{1}  , \quad
  A_{2}+X\rightarrow 2 X +Z ,\quad 2X\rightarrow P_{2}, \quad Z\rightarrow \lambda Y,
\end{equation*}
 where $A_{1}$ and $A_{2}$ are constants which represent reactants, $P_{1}$ and $P_{2}$ are products, $\lambda$ is the stoichiometric factor, and $X$, $Y$ and $Z$ are, respectively, the concentrations of the intermediates HBrO2 (bromous acid), B$\mbox{r}^{-}$ (bromide ion) and Ce(IV)(cerium). A simplified system of two
 equations \cite{Fi74} of the above reactant scheme is governed  by  (\ref{bp1}) with $L_{\alpha}u_{\alpha}=\varepsilon_{\alpha}\triangle u_{\alpha}$, $ \alpha=1,2$, where
 $u_{1}$ and $u_{2}$  represent, respectively, the concentrations $X$ and $Y$.
The reaction functions are given by
\begin{equation}\label{bp128}
  f_{1}=-u_{1}(a-bu_{1}-\sigma_{1}u_{2}),\quad f_{2}=\sigma_{2}u_{1}u_{2},
\end{equation}
where $a$, $b$, $\sigma_{\alpha}$, $\alpha=1,2$, are positive constants. 
 It is clear from (\ref{bp128}) that $f_{\alpha}$, $\alpha=1,2$,
are quasi-monotone nonincreasing functions (\ref{bp42}).
The nonlinear difference scheme (\ref{bp9}) is reduced to
\begin{eqnarray}\label{bp129}
 && (\mathcal{L}_{\alpha}^{h}(p,t_{m})+\tau^{-1}) U_{\alpha}(p,t_{m})+f_{\alpha}(U)-\tau^{-1}U_{\alpha}(p,t_{m-1})=0,\quad p \in \Omega^{h},\nonumber \\
&&  U_{\alpha}(p,t_{m})=g_{\alpha}(p,t_{m}),\quad p\in \partial \Omega^{h},\quad m\geq1, \\
&&  U_{\alpha}(p,0)=\psi_{\alpha}(p),\quad p\in \overline{\Omega}^{h},\quad \alpha=1,2,\nonumber \\
&& \mathcal{L}_{\alpha}^{h}(p,t_{m})U_{\alpha}(p,t_{m})=-\varepsilon_{\alpha}\left( D^{2}_{x}U_{\alpha}(p,t_{m})+D^{2}_{y}U_{\alpha}(p,t_{m})\right),\nonumber
  \end{eqnarray}
where $f_{\alpha}$, $\alpha=1,2$, are defined in (\ref{bp128}).  To satisfy  the assumptions in
(\ref{bp125}),  we choose constants    $K_{\alpha}$, $\alpha=1,2$,  in the following form
\begin{eqnarray*}
 && K_{1}\geq \max \left(a/b, \max_{(x,y,t)\in \partial Q_{T}}g_{1}(x,y,t), \max_{(x,y)\in \overline{\omega}}\psi_{1}(x,y)\right), \\
 && K_{2}\geq \max \left( \max_{(x,y,t)\in \partial Q_{T}}g_{2}(x,y,t), \max_{(x,y)\in \overline{\omega}}\psi_{2}(x,y)\right),\nonumber
\end{eqnarray*}
it follows that the mesh functions $\widehat{U}_{\alpha}(p,t_{m})$ and $\widetilde{U}_{\alpha}(p,t_{m})$
 from, respectively, (\ref{bp120}) and (\ref{bp126}) are ordered lower and upper solutions to (\ref{bp129}).

From (\ref{bp128}), in the sector $\langle \widehat{U}(t_{m}), \widetilde{U}(t_{m})\rangle=\langle 0, K_{\alpha}\rangle$, we have
\begin{eqnarray*}
&&\frac{\partial f_{1}}{\partial u_{1}}(U_{1},U_{2})=2 b U_{1}(p,t_{m})+\sigma_{1} U_{2}(p,t_{m})-a\leq 2 b K_{1}+\sigma_{1}K_{2},\ \ p\in \overline{\Omega}^{h},\\
&& \frac{\partial f_{2}}{\partial u_{2}}(U_{1},U_{2})=\sigma_{2}U_{1}(p,t_{m})\leq \sigma_{2}K_{1},\quad p\in \overline{\Omega}^{h},\\
&&-\frac{\partial f_{1}}{\partial u_{2}}(U_{1},U_{2})=-\sigma_{1}U_{1}(p,t_{m})\leq0,\quad p\in \overline{\Omega}^{h},\\
&&-\frac{\partial f_{2}}{\partial u_{1}}(U_{1},U_{2})=-\sigma_{2}U_{2}(p,t_{m})\leq0,\quad p\in \overline{\Omega}^{h},\quad m\geq1,
\end{eqnarray*}
and the assumptions in (\ref{bp20}) and (\ref{bp42}) are satisfied  with
\begin{equation*}
c_{1}=2 b K_{1}+\sigma_{1} K_{2},\quad c_{2}= \sigma_{2}K_{1}.
\end{equation*}
From here, (\ref{bp120}) and (\ref{bp126}), we conclude that Theorem \ref{bp97}
holds  for the Belousov-Zhabotinskii reaction diffusion model  (\ref{bp129}).
\subsection{Enzyme-substrate reaction diffusion   model}
In the enzyme-substrate model \cite{P92}, the chemical reaction scheme
is given by $E+S\rightleftharpoons ES\rightarrow E+P$, where $E$, $S$ and $P$ are, respectively,  enzyme,
substrate and reaction product. Denote by  $u_{1}(x,y,t)$ and $u_{2}(x,y,t)$ the concentrations of $S$ and $E$, respectively.
Then the above reactant scheme is governed by (\ref{bp1}) with $L_{\alpha}u_{\alpha}=\varepsilon_{\alpha}
\triangle u_{\alpha}$, $\alpha=1,2$. The reaction functions are given by  
\begin{equation}\label{bp83}
  f_{1}=a_{1}u_{1}u_{2}-b_{1}(E_{0}-u_{2}),\quad  f_{2}=a_{2}u_{1}u_{2}-b_{2}(E_{0}-u_{2}),
\end{equation}
 where a positive constant $E_{0}$ is the total enzyme,
 $a_{\alpha}>0$, $b_{\alpha}>0$, $\alpha=1,2$, are reaction constants.
It is clear from (\ref{bp83}) that $f_{\alpha}$, $\alpha=1,2$, are quasi-monotone nonincreasing functions (\ref{bp42}).
%
The nonlinear difference scheme (\ref{bp9}) is reduced to
\begin{eqnarray}\label{bp84}
 && (\mathcal{L}_{\alpha}^{h}(p,t_{m})+\tau^{-1}) U_{\alpha}(p,t_{m})+f_{\alpha}(U)-\tau^{-1}U_{\alpha}(p,t_{m-1})=0,\quad p \in \Omega^{h},\nonumber \\
&&  U_{\alpha}(p,t_{m})=g_{\alpha}(p,t_{m})\geq0,\quad p\in \partial \Omega^{h},\quad m\geq1,   \\
&& U_{\alpha}(p,0)=\psi_{\alpha}(p),\quad p\in \overline{\Omega}^{h},\quad \alpha=1,2,\nonumber \\
&& \mathcal{L}_{\alpha}^{h}(p,t_{m})U_{\alpha}(p,t_{m})=-\varepsilon_{\alpha}\left( D^{2}_{x}U_{\alpha}(p,t_{m})+D^{2}_{y}U_{\alpha}(p,t_{m})\right),\nonumber
  \end{eqnarray}
where $f_{\alpha}$, $\alpha=1,2$, are defined in (\ref{bp83}).

Introduce the linear problem
\begin{eqnarray}\label{bp85}
 && (\mathcal{L}_{1}^{h}(p,t_{m})+\tau^{-1}) V(p,t_{m})=\tau^{-1}V(p,t_{m-1})+M_{0},\quad p\in \Omega^{h},\quad m\geq1,\nonumber \\
 && V(p,t_{m})=g_{1}(p,t_{m}),\quad p\in\partial \Omega^{h},\quad V(p,0)=\psi_{1}(p),\quad p\in\overline{\Omega}^{h},\nonumber\\
 && M_{0}=\mbox{const}>0,\quad M_{0}>b_{1}E_{0}.
\end{eqnarray}
We now prove that    $(V(p,t_{m}), E_{0})$ and $(0,0)$  are ordered  upper and lower  solutions (\ref{bp40}) to (\ref{bp84}).
Firstly, we prove that $V(p,t_{m})\geq0$. From (\ref{bp85}), for $m=1$, we obtain that
\begin{eqnarray*}
 && (\mathcal{L}_{1}^{h}(p,t_{1})+\tau^{-1}) V(p,t_{1})=\tau^{-1}\psi_{1}(p)+M_{0},\quad p\in \Omega^{h},\nonumber \\
 && V(p,t_{1})=g_{1}(p,t_{1}),\quad p\in\partial \Omega^{h},\quad V(p,0)=\psi_{1}(p),\quad p\in\overline{\Omega}^{h}.
\end{eqnarray*}
From here and taking into account that $\psi_{1}(p)\geq0$, we have
\begin{eqnarray*}
&& (\mathcal{L}_{1}^{h}(p,t_{1})+\tau^{-1}) V(p,t_{1})\geq0,\quad p\in \Omega^{h}.\\
 && V(p,t_{1})=g_{1}(p,t_{1}),\quad p\in\partial \Omega^{h},\quad V(p,0)=\psi_{1}(p),\quad p\in\overline{\Omega}^{h}.
\end{eqnarray*}
Using the maximum principle in Lemma \ref{bp12}, we conclude that
\begin{equation*}
  V(p,t_{1})\geq0,\quad p\in \overline{\Omega}^{h}.
\end{equation*}
From here and (\ref{bp85}), for $m=2$, by a similar manner, we conclude that
\begin{equation*}
  V(p,t_{2})\geq0,\quad p\in \overline{\Omega}^{h}.
\end{equation*}
By induction on $m$, we can prove that $V(p,t_{m})\geq0$, $p\in\overline{\Omega}^{h}$, $m\geq1$. From here,  taking into account that the total enzyme
$E_{0}>0$ and (\ref{bp120}), it follows that $(V(p,t_{m}),E_{0})$ and $(0,0)$ satisfy (\ref{bp40a}). We now prove (\ref{bp40b}) for $(\widetilde{U}_{1}(p,t_{m}), \widehat{U}_{2}(p,t_{m}))= (V(p,t_{m}),0)$.
From  (\ref{bp40b}),  by using  (\ref{bp85}), we obtain that
\begin{equation*}
   \mathcal{G}_{1}\left(\widetilde{U}_{1}(p,t_{m}), \widetilde{U}_{1}(p,t_{m-1}), \widehat{U}_{2}(p,t_{m})\right)=f_{1}(V, 0)+M_{0},\quad p\in \Omega^{h},\quad m\geq1.
  \end{equation*}
From here, (\ref{bp83}) and (\ref{bp85}), we conclude that
\begin{equation*}
   \mathcal{G}_{1}\left(\widetilde{U}_{1}(p,t_{m}), \widetilde{U}_{1}(p,t_{m-1}), \widehat{U}_{2}(p,t_{m})\right)\geq0,\quad p\in \Omega^{h},\quad m\geq1.
 \end{equation*}
 Similarly, we prove that
 \begin{equation*}
   \mathcal{G}_{2}\left(\widehat{U}_{2}(p,t_{m}), \widehat{U}_{2}(p,t_{m-1}), \widetilde{U}_{1}(p,t_{m})\right)\leq0,\quad p\in \Omega^{h},\quad m\geq1.
 \end{equation*}
 Thus, we prove (\ref{bp40b}) for $(\widetilde{U}_{1}(p,t_{m}),\widehat{U}_{2}(p,t_{m}))=(V(p,t_{m}), 0)$.

Now, from (\ref{bp40b}) for $(\widehat{U}_{1}(p,t_{m}),\widetilde{U}_{2}(p,t_{m}))=(0, E_{0})$,  we have
 \begin{equation*}
   \mathcal{G}_{1}\left(\widehat{U}_{1}(p,t_{m}), \widehat{U}_{1}(p,t_{m-1}), \widetilde{U}_{2}(p,t_{m})\right)=f_{1}(0, E_{0}),\quad p\in \Omega^{h},\quad m\geq1.
  \end{equation*}
 From here and  (\ref{bp83}), we conclude that
 \begin{equation*}
   \mathcal{G}_{1}\left(\widehat{U}_{1}(p,t_{m}), \widehat{U}_{1}(p,t_{m-1}), \widetilde{U}_{2}(p,t_{m})\right)=0,\quad p\in \Omega^{h},\quad m\geq1.
  \end{equation*}
 Similarly, we obtain that
  \begin{equation*}
   \mathcal{G}_{2}\left(\widetilde{U}_{2}(p,t_{m}), \widetilde{U}_{2}(p,t_{m-1}), \widehat{U}_{1}(p,t_{m})\right)=0,\quad p\in \Omega^{h},\quad m\geq1.
  \end{equation*}
 Thus, we prove (\ref{bp40b}) for $(\widehat{U}_{1}(p,t_{m}),\widetilde{U}_{2}(p,t_{m}))=(0, E_{0})$.

Taking into account that the total enzyme $E_{0}$ satisfies  $E_{0}\geq u_{2}$,  we conclude that  $(\widetilde{U}_{1}, \widetilde{U}_{2})= (V,E_{0})$ and
 $(\widehat{U}_{1}, \widehat{U}_{2})= (0,0)$ satisfy (\ref{bp40c}). Thus,  we prove that
$(\widetilde{U}_{1},\widetilde{U}_{2})=(V,E_{0})$ and $(\widehat{U}_{1},\widehat{U}_{2})=(0,0)$   are ordered  upper and lower  solutions (\ref{bp40}) to (\ref{bp84}).
From (\ref{bp83}), in the sector $\langle \widehat{U}(t_{m}), \widetilde{U}(t_{m})\rangle$, $ \widehat{U}=(0,0)$, $ \widetilde{U}=(V,E_{0}) $,  we have
\begin{eqnarray*}
&&\frac{\partial f_{1}}{\partial u_{1}}(U_{1},U_{2})=a_{1} U_{2}(p,t_{m})\leq a_{1}E_{0},\quad p\in \overline{\Omega}^{h},\quad m\geq 1,\\
&& \frac{\partial f_{2}}{\partial u_{2}}(U_{1},U_{2})=a_{2}U_{1}(p,t_{m})+b_{2}\leq a_{2} V(p,t_{m})+b_{2},\quad p\in \overline{\Omega}^{h},\quad m\geq1,\\
&&-\frac{\partial f_{1}}{\partial u_{2}}(U_{1},U_{2})=-(a_{1}U_{1}(p,t_{m})+b_{2})\leq0,\quad p\in \overline{\Omega}^{h},\quad m\geq1,\\
&&-\frac{\partial f_{2}}{\partial u_{1}}(U_{1},U_{2})=-a_{2}U_{2}(p,t_{m})\leq0,\quad p\in \overline{\Omega}^{h},\quad m\geq1.
\end{eqnarray*}
Thus, the assumptions in (\ref{bp20}) and (\ref{bp42}) are satisfied  with
\begin{equation*}
c_{1}=a_{1}E_{0},\quad c_{2}(p,t_{m})= a_{2} V(p,t_{m})+b_{2},\quad p\in \overline{\Omega}^{h},\quad m\geq1.
\end{equation*}
From here, (\ref{bp120}) and (\ref{bp85}), we conclude that Theorem \ref{bp97}
holds  for the enzyme-substrate reaction diffusion   model (\ref{bp84}).
\section{Comparison of the block monotone Jacobi and block monotone Gauss--Seidel methods}
The following theorem shows that the block monotone Gauss--Seidel  method (\ref{bp94}), $(\eta=1)$, converges not slower than the block monotone Jacobi method
(\ref{bp27}), $(\eta=0)$.
\begin{theorem}\label{bp130}
Let $f(p,t_{m},U)$ in (\ref{bp9})  satisfy  (\ref{bp20}) and (\ref{bp42}),
 where    $\widetilde{U}(p,t_{m})=(\widetilde{U}_{1}(p,t_{m}),\widetilde{U}_{2}(p,t_{m}))$ and
 $\widehat{U}(p,t_{m})=(\widehat{U}_{1}(p,t_{m}),\widehat{U}_{2}(p,t_{m}))$
 are ordered upper and lower solutions (\ref{bp40}) of the nonlinear difference scheme (\ref{bp9}).
 Suppose that   $\{(\overline{U}^{(n)}_{1,i,m})_{J},(\underline{U}^{(n)}_{2,i,m})_{J}\}$, $\{(\underline{U}^{(n)}_{1,i,m})_{J},(\overline{U}^{(n)}_{2,i,m})_{J}\}$
  and $\{(\overline{U}^{(n)}_{1,i,m})_{GS},(\underline{U}^{(n)}_{2,i,m})_{GS}\}$, $\{(\underline{U}^{(n)}_{1,i,m})_{GS},(\overline{U}^{(n)}_{2,i,m})_{GS}\}$
 , $i=0,1,\ldots,N_{x}$, $\alpha=1,2$, $m\geq1$,
 are, respectively, the sequences generated by the block monotone Jacobi method (\ref{bp94}), $(\eta=0)$
 and the block monotone Gauss--Seidel method (\ref{bp94}), $(\eta=1)$,
 where $(\overline{U}^{(0)})_{J}=(\overline{U}^{(0)})_{GS}=\widetilde{U}$ and
  $(\underline{U}^{(0)})_{J}=(\underline{U}^{(0)})_{GS}=\widehat{U}$, then
  \begin{eqnarray}\label{bp131}
&&(\underline{U}^{(n)}_{\alpha,i,m})_{J}\leq  (\underline{U}^{(n)}_{\alpha,i,m})_{GS}\leq (\overline{U}^{(n)}_{\alpha,i,m})_{GS} \leq (\overline{U}^{(n)}_{\alpha,i,m})_{J},\quad i=0,1,\ldots,N_{x},\nonumber \\
&&  \alpha=1,2,\quad m\geq1.
  \end{eqnarray}
\end{theorem}
\begin{proof}
We consider the case of the sequences  $\{(\overline{U}^{(n)}_{1,i,m})_{J},(\underline{U}^{(n)}_{2,i,m})_{J}\}$ and
   $\{(\overline{U}^{(n)}_{1,i,m})_{GS},(\underline{U}^{(n)}_{2,i,m})_{GS}\}$.  From (\ref{bp27}),  we have
 \begin{eqnarray*}
 &&  A_{1,i,m}(\overline{U}_{1,i,m}^{(n)})_{J}+c_{1,m}(\overline{U}^{(n)}_{1,i,m})_{J} =c_{1,m}(\overline{U}^{(n-1)}_{1,i,m})_{J}+ L_{1,i,m}(\overline{U}_{1,i-1,m}^{(n-1)})_{J} \\
    &&+R_{1,i,m}(\overline{U}_{1,i+1,m}^{(n-1)})_{J} - F_{1,i,m}(\overline{U}_{1,i,m}^{(n-1)}, \underline{U}_{2,i,m}^{(n-1)} )_{J}+
    \tau^{-1} (\overline{U}_{1,i,m-1})_{J}\nonumber \\
    &&-G^{*}_{1,i,m}, \quad  i=1,2,\ldots,N_{x}-1,\quad (\overline{U}^{(n)}_{1,i,m})_{J}=g_{1,i,m}, \\
    &&i=0,N_{x},\quad m\geq1,\quad (U_{1,i,0})_{J}=\psi_{1,i},\quad i=0,1,\ldots,N_{x}.
 \end{eqnarray*}
  \begin{eqnarray*}
 &&  A_{1,i,m}(\overline{U}_{1,i,m}^{(n)})_{GS}+c_{1,m}(\overline{U}^{(n)}_{1,i,m})_{GS} =c_{1,m}(\overline{U}^{(n-1)}_{1,i,m})_{GS} \\
    &&+ L_{1,i,m}(\overline{U}_{1,i-1,m}^{(n)})_{GS} +R_{1,i,m}(\overline{U}_{1,i+1,m}^{(n-1)})_{GS} - F_{1,i,m}(\overline{U}_{1,i,m}^{(n-1)}, \underline{U}_{2,i,m}^{(n-1)} )_{GS}\nonumber \\
    &&+\tau^{-1} (\overline{U}_{1,i,m-1})_{GS}-G^{*}_{1,i,m}, \quad  i=1,2,\ldots,N_{x}-1, \\
    &&(\overline{U}^{(n)}_{1,i,m})_{GS}=g_{1,i,m},\quad i=0,N_{x},\quad m\geq1,\quad (U_{1,i,0})_{GS}=\psi_{1,i},\\
    && i=0,1,\ldots,N_{x}.
 \end{eqnarray*}
 From here, letting $W^{(n)}_{\alpha,i,m}=\left(U^{(n)}_{\alpha,i,m}\right)_{GS}-\left(U^{(n)}_{\alpha,i,m}\right)_{J}$, $i=0,1,\ldots,N_{x}$, $\alpha=1,2$, $m\geq1$, we have
 \begin{eqnarray}\label{bp132}
   &&   A_{1,i,m}\overline{W}^{(n)}_{1,i,m}+c_{1,m}\overline{W}^{(n)}_{1,i,m} = c_{1,m}\overline{W}^{(n-1)}_{1,i,m} \\ 
       &&+L_{1,i,m}\left((\overline{U}^{(n)}_{1,i-1,m})_{GS}-(\overline{U}^{(n-1)}_{1,i-1,m})_{J}\right)+R_{1,i,m}\overline{W}^{(n-1)}_{1,i+1,m} \nonumber\\
       &&- F_{1,i,m}\left(\overline{U}_{1,i,m}^{(n-1)}, \underline{U}_{2,i,m}^{(n-1)} \right)_{GS}+ F_{1,i,m}\left(\overline{U}_{1,i,m}^{(n-1)}, \underline{U}_{2,i,m}^{(n-1)} \right)_{J}\nonumber \\
       && +\tau^{-1} \left((\overline{U}_{1,i,m-1})_{GS}- (\overline{U}_{1,i,m-1})_{J}\right),\quad i=1,2,\ldots,N_{x}-1, \nonumber \\
    &&\overline{W}^{(n)}_{1,i,m}=\mathbf{0},\quad i=0,N_{x},\quad m\geq1.\nonumber
        \end{eqnarray}
By using  Theorem \ref{bp97}, we have $\left(\overline{U}^{(n)}_{1,i,m}\right)_{GS}\leq \left(\overline{U}^{(n-1)}_{1,i,m}\right)_{GS}$, $i=0,1,\ldots,N_{x}$, $m\geq1$. From here
and (\ref{bp132}), we obtain
 \begin{eqnarray}\label{bp133}
   &&   A_{1,i,m}\overline{W}^{(n)}_{1,i,m}+c_{1,m}\overline{W}^{(n)}_{1,i,m} \leq c_{1,m}\overline{W}^{(n-1)}_{1,i,m} +L_{1,i,m}\overline{W}^{(n-1)}_{1,i,m}\\ 
       &&+R_{1,i,m}\overline{W}^{(n-1)}_{1,i+1,m}- F_{1,i,m}\left(\overline{U}_{1,i,m}^{(n-1)}, \underline{U}_{2,i,m}^{(n-1)} \right)_{GS} \nonumber\\
       &&+ F_{1,i,m}\left(\overline{U}_{1,i,m}^{(n-1)}, \underline{U}_{2,i,m}^{(n-1)} \right)_{J}+\tau^{-1} \left((\overline{U}_{1,i,m-1})_{GS} - (\overline{U}_{1,i,m-1})_{J}\right),\nonumber\\
       && i=1,2,\ldots,N_{x}-1,\quad\overline{W}^{(n)}_{1,i,m}=\mathbf{0},\quad i=0,N_{x},\quad m\geq1.\nonumber
        \end{eqnarray}
Taking into account that $(A_{1,i,m}+c_{1,m}I)^{-1}\geq \emph{O}$,
$L_{1,i,m}\geq\emph{O}$, $R_{1,i,m}\geq\emph{O}$, $i=1,2,\ldots,N_{x}-1$,
 $m\geq1$, for $n=1$ in (\ref{bp133}), on the first time level $m=1$, in view
of $(\overline{U}^{(0)}_{1,i,m})_{GS}=(\overline{U}^{(0)}_{1,i,m})_{J}$
and $\overline{W}^{(0)}_{1,i,m}=\mathbf{0}$,  we conclude that
\begin{equation*}
  \overline{W}^{(1)}_{1,i,1}\leq\mathbf{0},\quad i=0,1,\ldots,N_{x}.
\end{equation*}
For $n=2$ in (\ref{bp133}) and using  notation (\ref{bp22}), we obtain
 \begin{eqnarray*}
   &&   \left(A_{1,i,1}+c_{1,1}\right)\overline{W}^{(2)}_{1,i,1} \leq
   L_{1,i,1}\overline{W}^{(1)}_{1,i,1}+R_{1,i,1}\overline{W}^{(1)}_{1,i+1,1} \\
   &&+\Gamma_{1,i,1}\left((\overline{U}^{(1)}_{1,i,1},\underline{U}^{(1)}_{2,i,1} )_{GS}\right)-
   \Gamma_{1,i,1}\left((\overline{U}^{(1)}_{1,i,1},\underline{U}^{(1)}_{2,i,1} )_{J}\right) \nonumber \\
   && \quad i=1,2,\ldots,N_{x}-1,\quad W^{(2)}_{1,i,1}=0,\quad i=0,N_{x}.
        \end{eqnarray*}
Taking into account that $(A_{1,i,1}+c_{1,1}I)^{-1}\geq \emph{O}$,
$L_{1,i,1}\geq\emph{O}$, $R_{1,i,1}\geq\emph{O}$, $i=1,2,\ldots,N_{x}-1$,
 and $ \overline{W}^{(1)}_{1,i,1}\leq\mathbf{0}$, by using (\ref{bp44}), we have
\begin{equation*}
  \overline{W}^{(2)}_{1,i,1}\leq\mathbf{0},\quad i=0,1,\ldots,N_{x},
\end{equation*}
where $U$ and $V$ in (\ref{bp44}) are taken in the form
\begin{equation}\label{bp134}
  U=\left((\overline{U}^{(1)}_{1,i,1})_{J}, (\underline{U}^{(1)}_{2,i,1})_{GS} \right),\quad V=\left( (\overline{U}^{(1)}_{1,i,1})_{GS}, (\underline{U}^{(1)}_{2,i,1})_{J}\right).
\end{equation}
By induction on $n$, we can prove that
\begin{equation*}
  \overline{W}^{(n)}_{1,i,1}\leq\mathbf{0},\quad i=0,1,\ldots,N_{x},\quad n\geq1 .
\end{equation*}
Similarly, by using the property $\left(\underline{U}^{(n-1)}_{2,i,m}\right)_{GS}\leq \left(\underline{U}^{(n)}_{2,i,m}\right)_{GS}$
in Theorem \ref{bp97}, we prove that
\begin{equation*}
  \underline{W}^{(n)}_{2,i,1}\geq\mathbf{0},\quad i=0,1,\ldots,N_{x},\quad n\geq1 .
\end{equation*}

On the second time level $m=2$, taking into account that $(A_{1,i,2}+c_{1,2}I)^{-1}\geq \emph{O}$,
$L_{1,i,2}\geq\emph{O}$, $R_{1,i,2}\geq\emph{O}$, $i=1,2,\ldots,N_{x}-1$,
 $  \overline{W}^{(0)}_{1,i,2}=\mathbf{0}$ and $\overline{W}_{1,i,1}\leq\mathbf{0}$,  from (\ref{bp133}), we have
\begin{equation*}
  \overline{W}^{(1)}_{1,i,2}\leq\mathbf{0},\quad i=0,1,\ldots,N_{x}.
\end{equation*}
For $n=2$ in (\ref{bp133}) and using  notation (\ref{bp22}), we obtain
 \begin{eqnarray*}
   &&   \left(A_{1,i,2}+c_{1,2}\right)\overline{W}^{(2)}_{1,i,2} \leq
   L_{1,i,2}\overline{W}^{(1)}_{1,i,2}+R_{1,i,2}\overline{W}^{(1)}_{1,i+1,2} \\
   &&+\Gamma_{1,i,2}\left((\overline{U}^{(1)}_{1,i,2},\underline{U}^{(1)}_{2,i,2} )_{GS}\right)-
   \Gamma_{1,i,2}\left((\overline{U}^{(1)}_{1,i,2},\underline{U}^{(1)}_{2,i,2} )_{J}\right) \nonumber \\
 &&+\tau^{-1} \left((\overline{U}_{1,i,1})_{GS} - (\overline{U}_{1,i,1})_{J}\right),\quad i=1,2,\ldots,N_{x}-1,\\
  &&W^{(2)}_{1,i,1}=0,\quad i=0,N_{x}.
        \end{eqnarray*}
Taking into account that $(A_{1,i,2}+c_{1,2}I)^{-1}\geq \emph{O}$,
$L_{1,i,2}\geq\emph{O}$, $R_{1,i,2}\geq\emph{O}$, $i=1,2,\ldots,N_{x}-1$,
  $ \overline{W}^{(1)}_{1,i,2}\leq\mathbf{0}$ and $\overline{W}_{1,i,1}\leq\mathbf{0}$, by using (\ref{bp44}), we have
\begin{equation*}
  \overline{W}^{(2)}_{1,i,2}\leq\mathbf{0},\quad i=0,1,\ldots,N_{x},
\end{equation*}
where $U$ and $V$ in (\ref{bp44}) are taken similar to (\ref{bp134}) with $m=2$.

By induction on $n$, we  can prove that
\begin{equation*}
  \overline{W}^{(n)}_{1,i,2}\leq\mathbf{0},\quad i=0,1,\ldots,N_{x}.
\end{equation*}
By induction on $m$, we can prove that
\begin{equation*}
  \overline{W}^{(n)}_{1,i,m}\leq\mathbf{0},\quad i=0,1,\ldots,N_{x},\quad m\geq1.
\end{equation*}
By a similar argument, we can prove that
\begin{equation*}
 \underline{W}^{(n)}_{2,i,m}\geq\mathbf{0},\quad i=0,1,\ldots,N_{x},\quad m\geq1.
\end{equation*}
Thus, we prove (\ref{bp131}) for $\{(\overline{U}^{(n)}_{1,i,m})_{J},(\underline{U}^{(n)}_{2,i,m})_{J}\}$ 
  and $\{(\overline{U}^{(n)}_{1,i,m})_{GS},(\underline{U}^{(n)}_{2,i,m})_{GS}\}$. 
By a similar manner,  we can prove (\ref{bp131}) for  $\{(\underline{U}^{(n)}_{1,i,m})_{J},(\overline{U}^{(n)}_{2,i,m})_{J}\}$ and \\
  $\{(\underline{U}^{(n)}_{1,i,m})_{GS},(\overline{U}^{(n)}_{2,i,m})_{GS}\}$.
\end{proof}

\end{document}